\documentclass[11pt, twoside, leqno]{article}

\usepackage{amssymb}
\usepackage{amsmath}
\usepackage{mathrsfs}
\usepackage{amsthm}
\usepackage{amsfonts}
\usepackage{color}
\usepackage{latexsym}
\usepackage{txfonts}
\usepackage{indentfirst}

\allowdisplaybreaks

\def\red{\color{red}}

\def\rr{{\mathbb R}}
\def\rn{{\mathbb{R}^n}}

\def\nn{{\mathbb N}}
\def\zz{{\mathbb Z}}
\def\cc{{\mathbb C}}
\def\CG{{\mathcal G}}

\def\CS{{\mathcal S}}

\def\CM{{\mathcal M}}
\def\CA{{\mathcal M}}
\def\CX{{\mathcal X}}
\def\CY{{\mathcal Y}}
\def\CB{{\mathcal B}}
\def\CJ{{\mathcal J}}
\def\CD{{\mathcal D}}
\newcommand{\RA}{\mathrm A}
\newcommand{\cw}{\mathrm{cw}}
\newcommand{\RB}{\mathrm B}
\newcommand{\RZ}{\mathrm Z}

\newcommand{\CN}{\mathcal{N}}

\newcommand{\CQ}{\mathcal{Q}}

\newcommand{\FN}{\mathfrak{N}}
\newcommand{\FM}{\mathfrak{M}}

\def\fz{\infty }
\def\az{\alpha}
\def\bz{\beta}
\def\dz{\delta}
\def\ez{\epsilon}
\def\kz{\kappa}
\def\thz{\theta}
\def\vz{\varphi}

\def\lf{\left}
\def\r{\right}
\def\ls{\lesssim}
\def\noz{\nonumber}
\def\wz{\widetilde}

\newcommand{\om}{\omega/(\omega+\eta)}

\def\loc{{\mathrm{loc}}}
\DeclareMathOperator{\supp}{supp}

\def\XXint#1#2#3{{\setbox0=\hbox{$#1{#2#3}{\int}$ }
\vcenter{\hbox{$#2#3$ }}\kern-.6\wd0}}

\def\lz{{\lambda}}

\def\CG{{\mathcal G}}
\def\CA{{\mathcal A}}
\def\CS{{\mathcal S}}
\def\CY{\mathcal Y}
\def\RY{{\mathrm Y}}
\def\gz{\gamma}

\newcommand{\RJ}{\mathrm J}
\def\UC{{\mathrm{UC}}}

\DeclareMathOperator{\diam}{diam}

\newcommand{\fin}{\mathrm{fin}}
\newcommand{\bmo}{\mathrm{bmo}}
\newcommand{\lip}{\mathrm{lip}}

\newtheorem{theorem}{Theorem}[section]
\newtheorem{lemma}[theorem]{Lemma}
\newtheorem{corollary}[theorem]{Corollary}
\newtheorem{proposition}[theorem]{Proposition}

\theoremstyle{definition}
\newtheorem{remark}[theorem]{Remark}
\newtheorem{definition}[theorem]{Definition}

\renewcommand{\appendix}{\par
   \setcounter{section}{0}%
   \setcounter{subsection}{0}%
   \setcounter{subsubsection}{0}%
   \gdef\thesection{\@Alph\c@section}%
   \gdef\thesubsection{\@Alph\c@section.\@arabic\c@subsection}%
   \gdef\theHsection{\@Alph\c@section.}%
   \gdef\theHsubsection{\@Alph\c@section.\@arabic\c@subsection}%
   \csname appendixmore\endcsname
 }

\newcommand{\go}[1]{\CG_0^\eta(#1)}

\newcommand{\GO}[1]{\mathring{\CG}(#1)}
\newcommand{\GOO}[1]{\mathring{\CG}^\eta_0(#1)}

\newcommand{\at}{\mathrm{at}}

\numberwithin{equation}{section}

\begin{document}
\title{\bf\Large Real-Variable Characterizations of Local Hardy Spaces on Spaces of Homogeneous Type
\footnotetext{\hspace{-0.35cm} 2010 {\it Mathematics Subject Classification}. Primary 42B30;
Secondary 42B35, 42B20, 30L99.\endgraf
{\it Key words and phrases.} local Hardy space, space of homogeneous type, maximal function,
Littlewood--Paley function, atom.\endgraf
This project is supported by the National Natural Science Foundation of China (Grant Nos. 11571039,
11761131002, 11671185 and 11871100)}}
\author{Ziyi He, Dachun Yang\footnote{Corresponding author/{\red July 20, 2019}/Final version.}
\ and Wen Yuan}
\date{}
\maketitle
	
\vspace{-0.8cm}

\begin{center}
\begin{minipage}{13cm}
{\small {\bf Abstract}\quad Suppose that $(X,d,\mu)$ is a space of homogeneous type, with upper dimension
$\mu$, in the sense of R. R. Coifman and G. Weiss. Let $\eta$ be the H\"{o}lder regularity index of wavelets
constructed by P. Auscher and T. Hyt\"{o}nen. In this article, the authors introduce the local Hardy space
$h^{*,p}(X)$ via local grand maximal functions and also characterize $h^{*,p}(X)$ via local
radial maximal functions, local non-tangential maximal functions, locally atoms and
local Littlewood--Paley functions. Furthermore, the authors establish the relationship between the global and
the local Hardy spaces. Finally, the authors also obtain the finite atomic characterizations of $h^{*,p}(X)$.
As an application, the authors give the dual spaces of $h^{*,p}(X)$ when $p\in(\omega/(\omega+\eta),1)$,
which further completes the result of G. Dafni and H. Yue on the dual space of $h^{*,1}(X)$. This article
also answers the question of R. R. Coifman and G. Weiss on the nonnecessity of any additional geometric
assumptions except the doubling condition for the radial maximal function characterization of
$H^1_{\mathrm{cw}}(X)$ when $\mu(X)<\fz$.}
\end{minipage}
\end{center}

\tableofcontents

\vspace{0.2cm}

\section{Introduction}\label{intro}
The real-variable theory of global Hardy spaces on the Euclidean space $\rn$ were essentially developed by
Stein and Weiss \cite{sw60} and later by Fefferman and Stein \cite{fs72}. Moreover, Coifman \cite{coi74}
obtained the atomic characterization of global Hardy spaces on $\rr$ and, later, Latter
\cite{lat78} generalized the result of Coifman to $\rn$ for any $n\in\nn$. For any $p\in(0,1]$, by the
atomic characterization of $H^p(\rn)$, it can be seen that any element in $H^p(\rn)$ has a kind of
cancellations, which implies that the space $\CS(\rn)$ of all Schwartz functions is not dense in $H^p(\rn)$.
For more studies on $H^p(\rn)$, see, for instance, \cite{stein93,lu95,gr85,gra14,gra14b}.

In 1979, Goldberg \cite{goldberg79} introduced the local Hardy space $h^p(\rn)$, with any given $p\in(0,1]$,
via using the local Riesz transforms (see \cite[p.\ 28]{goldberg79} for the details), obtained their
atomic characterizations and proved the boundedness of pseudo-differential operators on
$h^p(\rn)$. Differently from the global Hardy space $H^p(\rn)$, in the atomic characterization of
$h^p(\rn)$, the cancellation property is needed only for those atoms with small supports, while, in the
atomic characterization of $H^p(\rn)$, the cancellation property is needed for all atoms.
This makes $\CS(\rn)$ dense in $h^p(\rn)$. Therefore, local Hardy spaces are more suitable to
problems of strong locality such as these problems associated with partial differential equations and
manifolds.

From then on, local Hardy spaces have become an important topic in harmonic analysis and attracted many
attentions. In \cite{goldberg79}, Goldberg obtained the maximal function characterizations of
$h^p(\rn)$ when $p\in((n-1)/n,1]$. Peloso and Silvia \cite{ps08} extended the result of Goldberg to all
$p\in(0,1]$. Later, Chang et al.\ \cite{cdy09} obtained a div-curl decomposition for local
Hardy space $h^1(\rn)$.
In 1983, Triebel \cite{tri83} proved that $h^p(\rn)$ coincides with the Triebel--Lizorkin
space $F^0_{p,2}(\rn)$, via first establishing the Littlewood--Paley characterizations of $h^p(\rn)$. In
1981, Bui \cite{bui81} introduced the weighted local Hardy space $h^p_w(\rn)$ for any Muckenhoupt
$A_\fz(\rn)$-weight $w$. Later,
Rychkov \cite{rych01} generalized the result of Bui to $A^\loc_\fz(\rn)$-weights and obtained the
Littlewood--Paley function characterizations of $h^p_w(\rn)$. In 2012, Nakai and Sawano \cite{ns12}
introduced the local Hardy spaces with exponent variables. Meanwhile, Yang et al. \cite{yy12}
introduced the local Musielak--Orlicz Hardy spaces (see also \cite{ylk17}). Recently, Sawano et al.\
\cite{shyy17} introduced the local Hardy type spaces built on ball quasi-Banach spaces. On another hand,
Betancor and Dami\'{a}n \cite{bd10} introduced the anisotropic local Hardy spaces, whose underlying space is
more likely to be a space of homogeneous type.

Let us recall the notion of spaces of homogeneous type introduced by Coifman and Weiss \cite{cw71}.
A \emph{quasi-metric space} $(X,d)$ is a non-empty set $X$ equipped with a \emph{quasi-metric} $d$, that is,
a non-negative function on $X\times X$ satisfying that, for any $x,\ y,\ z\in X$,
\begin{enumerate}
\item $d(x,y)=0$ if and only if $x=y$;
\item $d(x,y)=d(y,x)$;
\item there exists a constant $A_0\in[1,\fz)$, independent of $x$, $y$ and $z$, such that
$d(x,y)\le A_0[d(x,z)+d(z,y)]$.
\end{enumerate}
It is easy to see that, if $A_0=1$, then $(X,d)$ is a \emph{metric space}. The \emph{ball $B$} centered at
$x_0\in X$ with radius $r\in(0,\fz)$ is defined by setting
$$
B=B(x_0,r):=\{y\in X:\ d(x,y)<r\}.
$$
Recall that a subset $U$ of $X$ is said to be \emph{open} if, for any $x\in U$, there exists $\ez\in(0,\fz)$
such that $B(x,\ez)\subset U$. It can be seen that the family of all open sets of $X$ determines a topology
of $X$. A triple $(X,d,\mu)$ is called a \emph{space of homogeneous type} if $(X,d)$ is a quasi-metric space
and $\mu$ is a positive measure, defined on the $\sigma$-algebra generated by all open sets and balls,
satisfying the following \emph{doubling condition}: there exists a constant $C_{(\mu)}\in[1,\fz)$ such that,
for any ball $B$,
$$
\mu(2B)\le C_{(\mu)}\mu(B).
$$
Here and hereafter, for any $\tau\in(0,\fz)$ and ball $B$, $\tau B$ denotes the ball centered at the same
center as $B$ with radius $\tau$ times that of $B$. It follows from the doubling condition that, for any
$\lz\in[1,\fz)$ and ball $B$,
\begin{equation}\label{eq-doub}
\mu(\lz B)\le C_{(\mu)}\lz^\omega\mu(B),
\end{equation}
where $\omega:=\log_2 C_{(\mu)}$ is called the \emph{upper dimension} of $X$. If $d$ is a metric, then we
call $X$ a \emph{doubling metric measure space}.

According to \cite[pp.\,587--588]{cw77}, we \emph{always make} the following assumptions throughout this
article. For any point $x\in X$, assume that the balls $\{B(x,r)\}_{r\in(0,\infty)}$ form a \emph{basis} of
open neighborhoods of $x$, but these balls themselves are not necessary to be open;
assume that $\mu$ is \emph{Borel regular}, which means that open sets are
measurable and every set $A\subset X$ is contained in a Borel set $E$ satisfying that $\mu(A)=\mu(E)$; we
also assume that $\mu(B(x, r))\in(0,\fz)$ for any $x\in X$ and $r\in(0,\infty)$. For the presentation
concision, we \emph{always assume} that $(X,d,\mu)$ is non-atomic [namely, $\mu(\{x\})=0$ for any $x\in X$].

In 1977, Coifman and Weiss \cite{cw77} introduced Hardy spaces $H^p_{\cw}(X)$ on a space $X$ of
homogeneous type by using atoms. Observe that, to some extend, when $\mu(X)<\fz$,
$H^p_\cw(X)$ has many similarities to $h^p(\rn)$, such as the property of atoms, dual spaces and so on.
Recently, Dafni and Yue \cite{dy12} introduced the local Hardy space $h^1(X)$ on a space $X$ of
homogeneous type and proved the dual space of $h^1(X)$ is $\mathrm{bmo}(X)$, the local space of functions of
bounded mean oscillation (see also \cite{dmy16}). Later, Gong et al.\ \cite{gly13} introduced the local
Hardy spaces associated with operators on doubling metric measure spaces.
In 2014, Bui and Duong \cite{bd14} introduced Hardy spaces associated with discrete Laplacian on
doubling graphs which can be seen as special case of spaces of homogeneous type, via atomic decomposition.
In 2018, Bui et al. \cite{bdl18a} obtained the maximal function characterizations of local Hardy spaces associated with some
non-negative self-adjoint operator $L$.
Recently, Bui et al.\ \cite{bdl18}
showed that, when $\mu(X)<\fz$, for any operator $L$ with its heat kernels having good decay properties,
then the Hardy spaces associated with $L$ coincide with the Hardy spaces of Coifman and Weiss.

Recall that $(X,d,\mu)$ is called an \emph{RD-space} if it is a doubling metric measure space and satisfies
the following \emph{reverse doubling condition}: there exist positive constants $\wz{C}\in(0,1]$
and $\kz\in(0,\omega]$ such that, for any ball $B(x,r)$ with $x\in X$, $r\in(0,\diam X/2)$ and
$\lz\in[1,\diam X/[2r])$,
$$
\wz C\lz^\kz\mu(B(x,r))\le\mu(B(x,\lz r)),
$$
here and hereafter, $\diam X:=\sup_{x,\ y\in X}d(x,y)$.
In \cite{hmy08}, Han et al.\ introduced the inhomogeneous Besov and Triebel--Lizorkin spaces on RD-spaces by
using the \emph{inhomogeneous approximations of the identity}. They also introduced the local Hardy space
$h^p(X)$ with $p\le 1$ but near to $1$ on RD-spaces as a special case of inhomogeneous Triebel--Lizorkin
spaces and obtained its atomic characterization by using
the inhomogeneous Calder\'{o}n reproducing formulae on RD-spaces established in \cite{hmy08}. As an
application, Yang and Zhou \cite{yz11} characterized inhomogeneous Besov and Triebel--Lizorkin spaces by
using these local Hardy spaces. Moreover, the real-variable theory of Hardy-type spaces on RD-spaces
has been completed; see, for instance, \cite{hmy06,hmy08,gly08,zsy16}.

Now we return to the case that $(X,d,\mu)$ is a space of homogeneous type. Recently, the
inhomogeneous Calder\'{o}n reproducing formulae on spaces of homogeneous type without any additional reverse
doubling condition were established in \cite{hlyy18}. It is well known that
Calder\'on reproducing formulae are very useful and important tools in the study on function spaces,
including Hardy spaces, Besov spaces and Triebel--Lizorkin spaces. Therefore, the existence of inhomogeneous
Calder\'on reproducing formulae provides the possibilities to study local Hardy spaces on $X$. This is one
motivation of this article.

Another motivation of this article is the real-variable theory of Hardy spaces established in
\cite{hhllyy18}. As an application of the Calder\'on reproducing formulae obtained in \cite{hlyy18}, He et
al.\ \cite{hhllyy18} developed a complete real-variable theory of global Hardy spaces $H^p(X)$ on a space $X$
of homogeneous type with $p\in(\om,1]$, where $\omega$, as in \eqref{eq-doub}, denotes the upper dimension of
$X$ and $\eta$ represents the H\"{o}lder regularity of inhomogeneous approximations of the identity with
exponential decay (see Definition \ref{def-iati} below). In what follows, to distinguish the local Hardy
space studied in this article, we call $H^p(X)$ the \emph{global Hardy space}. As
is well known from \cite{goldberg79b}, the local Hardy space has similar properties to the global Hardy
space. This motivates us to establish the related theory of local Hardy spaces on spaces of homogeneous type.
Moreover, it should be mentioned that Hardy spaces on spaces of homogeneous type can be used in
endpoint estimates on commutators and bilinear decompositions; see, for instance, \cite{fcy17,lcfy17,lcfy18}.

In this article, we introduce various local Hardy spaces with any given $p\in(\omega/(\omega+\eta),\fz]$ on
a space $X$ of homogeneous type via various maximal functions, and then prove that all these local Hardy
spaces are mutually equivalent, which are hence denoted by $h^p(X)$. When $p\in(1,\fz]$, the equivalence
between the local Hardy space $h^p(X)$ and the Lebesgue space $L^p(X)$ is also obtained. Moreover, using the
Calder\'{o}n--Zygmund decomposition proved in \cite{hhllyy18}, we obtain the atomic characterization of
$h^p(X)$ for any given $p\in(\om,1]$.
We also characterize $h^p(X)$ via local Littlewood--Paley functions by using the molecular characterization
of $h^p(X)$. Moreover, we compare the local Hardy space $h^p(X)$ with the global Hardy space $H^p(X)$.
Finally, we obtain the finite atomic characterization of local Hardy spaces and, as an application, we give
the dual space of local Hardy spaces.

Comparing our results with those in \cite{dy12}, we introduce the local Hardy space $h^p(X)$, not only
$h^1(X)$, with any given $p\in(\om,1]$ which, according to \cite{hhllyy18}, should be the optimal range of
$p$. Moreover, in \cite{dy12}, Dafni and Yue obtained the atomic characterization of $h^1(X)$,
while, in this article, we not only obtain the atomic characterization of $h^p(X)$, but also
characterize $h^p(X)$ via various maximal functions and Littlewood--Paley functions.

Also, note that the anisotropic Euclidean space $\rn$ is a special space of homogeneous type. Thus, our
results cover the results in \cite{bd10} when $p\in(n/(n+1),1]$. Moreover, compared with the results in
\cite{bd10}, ours also contain the Littlewood--Paley function and finite atomic characterizations of local
Hardy spaces.

In addition, we point out that local Hardy spaces also play an important role in geometric analysis and
partial differential equations. For instance, Taylor \cite{tay09} introduced the Hardy space on a completely
Riemannian manifold with bounded geometry by using a similar way to the definition of local Hardy
spaces; Chang et al.\ \cite{cks93} introduced two kinds of Hardy spaces on a smooth and bounded domain by
using local Hardy spaces on $\rr^N$ and studied the boundary problems of elliptic equations.

Recall that, in \cite{cw77}, Coifman and Weiss obtained the radial maximal function characterization of
$H^1_\cw(X)$ under an additional geometrical assumption. They also asked whether or not such an
additional geometrical assumption can be removed to obtain the radial maximal function characterization of
$H^1_\cw(X)$. In \cite{hhllyy18}, He et al.\ answered this question when $\mu(X)=\fz$ by showing that no
additional geometrical assumption is necessary to guarantee the radial maximal function characterization of
$H^1_\cw(X)$. In this article, we give the same answer in the case $\mu(X)<\fz$ by establishing the radial
maximal function characterization of $h^1(X)$ and the equivalence between $H^1_\cw(X)$ and $h^1(X)$.
Moreover, a similar conclusion also holds true for $h^p(X)$ when $p\le 1$ but near $1$.

The organization of this article is as follows.

In Section \ref{s-pre}, we recall the notions of test functions and distributions introduced in
\cite{hmy08} (see also \cite{hmy06}). Moreover, we recall the notion of inhomogeneous approximations of
the identity with exponential decaying (for short, $\exp$-IATI) which play a very important role in
introducing the local Hardy spaces.

In Section \ref{s-max}, we introduce three kinds of local Hardy spaces via grand maximal functions, radial
maximal functions and non-tangential maximal functions, respectively. We prove that all three kinds of Hardy
spaces are mutually equivalent by using the inhomogeneous Calder\'on reproducing formulae established in
\cite{hlyy18}. Moreover, when $p\in(1,\fz]$, we further show that the aforementioned three kinds of Hardy
spaces are essentially the Lebesgue space. Observe that the definition of the radial maximal function is
different from the Euclidean case, due to the special forms of inhomogeneous discrete Calder\'on reproducing
formulae.

Section \ref{s-at} mainly concerns about the atomic characterizations of local Hardy spaces. We introduce the
notion of local atoms. Compared with the global atom of Hardy spaces, the local atom has the cancellation
property only when its support is small. We first show that any local atom belongs to local Hardy spaces with
its (quasi-)norm uniformly bounded. To show the contrary conclusion, we find a dense space of the local Hardy
space and obtain an atomic decomposition for elements in this dense subspace. Then we use a density argument
(see \cite{hhllyy18} or \cite{ms79b}) to obtain an atomic decomposition for any element in the local Hardy
space.

In Section \ref{s-lp}, we obtain the Littlewood--Paley function characterizations of local Hardy spaces
via using the molecular characterization. We first introduce another kind of local Hardy spaces by using Lusin
area functions. We show that these local Hardy spaces are independent of the choice of $\exp$-IATIs (see
Theorem \ref{thm-la} below). Then we prove that these local Hardy spaces have molecular
characterizations, which are equivalent to the atomic ones (see Theorem \ref{thm-m=a} below). Moreover, by
the atomic characterizations, we further establish the Littlewood--Paley $g$-function and
$g_\lz^*$-function characterizations of these local Hardy spaces.

In Section \ref{s-hh}, we consider the differences and the relationship between global and local Hardy
spaces. Generally speaking, when $\mu(X)=\fz$, we show that all test functions belong to local Hardy spaces,
while only test functions with cancellation property belong to $H^p(X)$. Moreover, we obtain a relationship
between the local and the global Hardy spaces; see Section \ref{ss-infz} below. When $\mu(X)<\fz$, we
show that these two kinds of Hardy spaces are essentially the same; see Section \ref{ss-fz} below.

In Section \ref{s-dual}, we mainly concern about the dual space of local Hardy spaces. To this end, we first
obtain the finite atomic characterization of local Hardy spaces, whose proof is similar to the proof of the
finite atomic characterization of global Hardy spaces obtained in \cite{hhllyy18}; see Section
\ref{ss-fin}. By this, we introduce the local versions of Campanato spaces and Lipschitz spaces. We prove
that these spaces are the dual spaces of local Hardy spaces by using some ideas from the proof of
\cite[Theorem B]{cw77}.

At the end of this section, we make some conventions on notation. We \emph{always assume} that $\omega$ is as
in \eqref{eq-doub} and $\eta$ is the H\"{o}lder regularity index of inhomogeneous approximations of the
identity with exponential decay (see Definition \ref{def-iati} below). We assume that $\dz$ is a very small
positive number, for instance, $\dz\le(2A_0)^{-10}$ in order to construct the dyadic cube system and the
wavelet system on $X$ (see \cite[Theorem 2.2]{hk12} or Lemma \ref{cube} below).
For any $x,\ y\in X$ and $r\in(0,\fz)$, let
$$
V_r(x):=\mu(B(x,r))\quad \mathrm{and}\quad V(x,y):=\mu(B(x,d(x,y))),
$$
where $B(x,r):=\{y\in X:\ d(x,y)<r\}$.
We always let $\nn:=\{1,2,\ldots\}$ and $\zz_+:=\nn\cup\{0\}$.
For any $p\in[1,\fz]$, we use $p'$ to denote its \emph{conjugate index}, namely, $1/p+1/p'=1$.
The symbol $C$ denotes a positive constant which is independent of the main parameters, but it may vary from
line to line. We also use $C_{(\az,\bz,\ldots)}$ to denote a positive constant depending on the indicated
parameters $\az$, $\bz$, \ldots. The symbol $A \ls B$ means that there exists a positive constant $C$ such
that $A \le CB$. The symbol $A \sim B$ is used as an abbreviation of $A \ls B \ls A$. We also use
$A\ls_{\az,\bz,\ldots}B$ to indicate that here the implicit positive constant depends on $\az$, $\bz$,
$\ldots$ and, similarly, $A\sim_{\az,\bz,\ldots}B$. We use the following convention: If $f\le Cg$ and $g=h$ or
$g\le h$, we then write $f\ls g\sim h$ or $f\ls g\ls h$, \emph{rather than} $f\ls g=h$
or $f\ls g\le h$. For any $s,\ t\in\rr$, denote the \emph{minimum} of $s$
and $t$ by $s\wedge t$ and the \emph{maximum} by $s\vee t$.
For any finite set $\CJ$, we use $\#\CJ$ to denote its \emph{cardinality}. Also, for any set
$E$ of $X$, we use $\mathbf{1}_E$ to denote its \emph{characteristic function} and $E^\complement$ the set
$X\setminus E$.

\section{Inhomogeneous Calder\'on reproducing formulae}\label{s-pre}
In this section, we recall the inhomogeneous Calder\'on reproducing formulae established in \cite{hlyy18}. To
this end, we first recall the definitions of test functions and distributions.
\begin{definition}\label{def-test}
Let $x_1\in X$, $r\in(0,\fz)$, $\bz\in(0,1]$ and $\gz\in(0,\fz)$. A function $f$ defined on $X$ is called a
\emph{test function of type $(x_1,r,\bz,\gz)$}, denoted by $f\in\CG(x_1,r,\bz,\gz)$, if there exists a
positive constant $C$ such that
\begin{enumerate}
\item (the \emph{size condition}) for any $x\in X$,
$$
|f(x)|\le C\frac{1}{V_1(x_1)+V(x_1,x)}\lf[\frac{r}{r+d(x_1,x)}\r]^\gz;
$$
\item (the \emph{regularity condition}) for any $x,\ y\in X$ satisfying $d(x,y)\le (2A_0)^{-1}[r+d(x_1,x)]$,
\begin{equation*}
|f(x)-f(y)|\le C\lf[\frac{d(x,y)}{r+d(x_1,x)}\r]^\bz
\frac{1}{V_r(x_1)+V(x_1,x)}\lf[\frac r{r+d(x_1,x)}\r]^\gz.
\end{equation*}
\end{enumerate}
For any $f\in\CG(x_1,r,\bz,\gz)$, define its norm
$$
\|f\|_{\CG(x_1,r,\bz,\gz)}:=\inf\{C\in(0,\fz):\ C \textup{\ satisfies (i) and (ii)}\}.
$$
Moreover, define
$$
\GO{x_1,r,\bz,\gz}:=\lf\{f\in\CG(x_1,r,\bz,\gz):\ \int_X f(x)\,d\mu(x)=0\r\}
$$
equipped with the norm $\|\cdot\|_{\GO{x_1,r,\bz,\gz}}:=\|\cdot\|_{\CG(x_1,r,\bz,\gz)}$.
\end{definition}

The above version of $\CG(x_1,r,\bz,\gz)$ was originally introduced by Han et al.\ \cite{hmy08}
(see also \cite{hmy06}).

Fix $x_0\in X$. It can be seen that, for any $x\in X$ and $r\in(0,\fz)$, $\CG(x,r,\bz,\gz)=\CG(x_0,1,\bz,\gz)$
with equivalent norms, but the positive equivalence  constants depend on $x$ and $r$. Moreover,
$\CG(x_0,1,\bz,\gz)$ is a Banach space. In what follows, we simply write $\CG(\bz,\gz):=\CG(x_0,1,\bz,\gz)$
and $\GO{\bz,\gz}:=\GO{x_1,1,\bz,\gz}$.

Fix $\ez\in(0,1]$ and $\bz,\ \gz\in(0,\ez]$. Denote by $\CG_0^\ez(\bz,\gz)$ [resp.,
$\mathring{\CG}_0^\ez(\bz,\gz)$] the completion of $\CG(\ez,\ez)$ [resp., $\GO{\ez,\ez}$]
in $\CG(\bz,\gz)$. More precisely, for any $f\in\CG_0^\ez(\bz,\gz)$
[resp., $f\in\mathring\CG_0^\ez(\bz,\gz)$], there exists $\{f_n\}_{n=1}^\fz\subset\CG(\ez,\ez)$
[resp., $\{f_n\}_{n=1}^\fz\subset\GO{\ez,\ez}$] such that $\|f_n-f\|_{\CG(\bz,\gz)}\to 0$ as $n\to\fz$.
For any $f\in\CG_0^\ez(\bz,\gz)$ [resp., $f\in\mathring\CG^\epsilon_0(\bz,\gz)$],
let $\|f\|_{\CG_0^\ez(\bz,\gz)}:=\|f\|_{\CG(\bz,\gz)}$ [resp.,
$\|f\|_{\mathring\CG^\epsilon_0(\bz,\gz)}:=\|f\|_{\CG(\bz,\gz)}$]. Then $\CG_0^\ez(\bz,\gz)$
and $\mathring\CG^\epsilon_0(\bz,\gz)$ are closed subspaces of $\CG(\bz,\gz)$. The \emph{dual space}
$(\CG^\ez_0(\bz,\gz))'$ is defined to be the set of all continuous linear functionals on $\CG^\ez_0(\bz,\gz)$
and equipped with the weak-$*$ topology, which is also called the \emph{space of distributions}.

Let the \emph{symbol $L_\loc^1(X)$} denote the space of all locally integrable functions on $X$. For any
$f\in L^1_\loc(X)$, define the \emph{Hardy--Littlewood maximal function $\CM(f)$} of $f$ by setting, for any
$x\in X$,
\begin{equation}\label{2.1x}
\CM(f)(x):=\sup_{B\ni x}\frac 1{\mu(B)}\int_{B} |f(y)|\,d\mu(y),
\end{equation}
where the supremum is taken over all balls $B$ of $X$ that contain $x$. For any $p\in(0,\fz]$,  the
\emph{Lebesgue space $L^p(X)$} is defined to be the set of all $\mu$-measurable functions $f$ such that
$$
\|f\|_{L^p(X)}:=\lf[\int_X |f(x)|^p\,d\mu(x)\r]^{1/p}<\fz
$$
with the usual modification made when $p=\fz$; the \emph{weak Lebesgue space  $L^{p,\fz}(X)$} is defined to be
the set of all $\mu$-measurable functions $f$ such that
$$
\|f\|_{L^{p,\fz}(X)}:=\sup_{\lz\in(0,\fz)}\lz[\mu(\{x\in X:\ |f(x)|>\lz\})]^{1/p}<\fz.
$$
It is shown in \cite{cw71} that $\CM$ is bounded on $L^p(X)$ when $p\in(1,\fz]$ and bounded from $L^1(X)$ to
$L^{1,\fz}(X)$.

Then we state some estimates from \cite[Lemma 2.1]{hmy08} (see also \cite[Lemma 2.2]{hhllyy18}).

\begin{lemma}\label{lem-add}
Let $\bz,\ \gz\in(0,\infty)$.
\begin{enumerate}
\item For any $x,\ y\in X$ and $r\in(0,\fz)$, $V(x,y)\sim V(y,x)$ and
$$
V_r(x)+V_r(y)+V(x,y)\sim  V_r(x)+V(x,y)\sim V_r(y)+V(x,y)\sim\mu(B(x,r+d(x,y))),
$$
where the positive equivalence constants are independent of $x$, $y$ and $r$.

\item There exists a positive constant $C$ such that, for any $x_1\in X$ and $r\in(0,\infty)$,
$$
\int_X\frac{1}{V_r(x_1)+V(x_1,x)}\lf[\frac r{r+d(x_1,x)}\r]^\gz\,d\mu(x)\le C.
$$

\item There exists a positive constant $C$ such that, for any $x\in X$ and $R\in(0,\infty)$,
$$
\int_{d(x,y)\le R}\frac 1{V(x,y)}\lf[\frac{d(x,y)}{R}\r]^\bz\,d\mu(y)\le C
\quad \textit{and}\quad\int_{d(x, y)\ge R}\frac{1}{V(x,y)}\lf[\frac R{d(x,y)}\r]^\bz\,d\mu(y)\le C.
$$

\item There exists a positive constant $C$ such that, for any $x_1\in X$ and $R,\ r\in(0,\infty)$,
$$
\int_{d(x, x_1)\ge R}\frac{1}{V_r(x_1)+V(x_1,x)}\lf[\frac r{r+d(x_1,x)}\r]^\gz\,d\mu(x)\
\le C\lf(\frac{r}{r+R}\r)^\gz.
$$

\item There exists a positive constant $C$ such that, for any $r\in(0,\fz)$, $f\in L^1_\loc(X)$ and $x\in X$,
$$
\int_X \frac{1}{V_r(x)+V(x,y)}\lf[\frac{r}{r+d(x,y)}\r]^\gz|f(y)|\,d\mu(y)\le C\CM(f)(x),
$$
where $\CM$ is as in \eqref{2.1x}.
\end{enumerate}
\end{lemma}

Next we recall the system of dyadic cubes established in \cite[Theorem 2.2]{hk12} (see also \cite{ah13}),
which is restated as in the following version.

\begin{lemma}\label{cube}
Fix constants $0<c_0\le C_0<\fz$ and $\dz\in(0,1)$ such that $12A_0^3C_0\dz\le c_0$. Assume that
a set of points, $\{z_\az^k:\ k\in\zz,\ \az\in\CA_k\}\subset X$ with $\CA_k$ for any $k\in\zz$ being a
countable set of indices, has the following properties: for any $k\in\zz$,
\begin{enumerate}
\item[\rm (i)] $d(z_\az^k,z_\bz^k)\ge c_0\dz^k$ if $\az\neq\bz$;
\item[\rm (ii)] $\min_{\az\in\CA_k} d(x,z_\az^k)\le C_0\dz^k$ for any $x\in X$.
\end{enumerate}
Then there exists a family of sets, $\{Q_\az^k:\  k\in\zz,\ \az\in\CA_k\}$, satisfying
\begin{enumerate}
\item[\rm (iii)] for any $k\in\zz$, $\bigcup_{\az\in\CA_k} Q_\az^k=X$ and $\{ Q_\az^k:\;\; {\az\in\CA_k}\}$
is disjoint;
\item[\rm (iv)] if $k,\ l\in\zz$ and $k\le l$, then either $Q_\az^k\supset Q_\bz^l$ or
$Q_\az^k\cap Q_\bz^l=\emptyset$;
\item[\rm (v)] for any $k\in\zz$ and $\az\in\CA_k$, $B(z_\az^k,c_\natural\dz^k)\subset Q_\az^k\subset
B(z_\az^k,C^\natural\dz^k)$,
where $c_\natural:=(3A_0^2)^{-1}c_0$, $C^\natural:=2A_0C_0$ and $z_\az^k$ is called ``the center'' of
$Q_\az^k$.
\end{enumerate}
\end{lemma}

Throughout this article, we keep the notation appearing in Lemma \ref{cube}. Moreover, for any $k\in\zz$, let
$$
\CX^k:=\{z_\az^k\}_{\az\in\CA_k},\qquad \CG_k:=\CA_{k+1}\setminus\CA_k\qquad \textup{and}\qquad
\CY^k:=\{z_\az^{k+1}\}_{\az\in\CG_k}=:\{y_\az^{k}\}_{\az\in\CG_k}.
$$
The existence of $\{\CX_k\}_{k=-\fz}^\fz$ was ensured in \cite[Section 2.21]{hk12} (see also
\cite[Section 2.3]{ah13}). Moreover, by the construction of $\{\CX_k\}_{k=-\fz}^\fz$, we find that, for any
$k\in\zz$, $\CX_k\subset\CX_{k+1}$, which makes the definition of $\CG_k$ well defined.

Next we recall the definition of $\exp$-IATIs, which can be used to establish the inhomogeneous Calder\'on
reproducing formulae.
\begin{definition}\label{def-iati}
A sequence $\{Q_k\}_{k=0}^\fz$ of bounded operators on $L^2(X)$ is called an \emph{inhomogeneous
approximation of the identity with exponential decay} (for short, $\exp$-IATI) if there exist constants
$C,\ \nu\in(0,\fz)$, $a\in(0,1]$ and $\eta\in(0,1)$ such that, for any $k\in\zz_+$, the kernel of the
operator $Q_k$, which is still denoted by $Q_k$, has the following properties:
\begin{enumerate}
\item (the \emph{identity condition}) $\sum_{k=0}^\fz Q_k=I$ in $L^2(X)$, where $I$ is the identity operator
on $L^2(X)$;
\item when $k\in\nn$, then $Q_k$ satisfies
\begin{enumerate}
\item (the \emph{size condition}) for any $x,\ y\in X$,
\begin{align}\label{eq:etisize}
|Q_k(x,y)|&\le C\frac1{\sqrt{V_{\dz^k}(x)\,V_{\dz^k}(y)}}
\exp\lf\{-\nu\lf[\frac{d(x,y)}{\dz^k}\r]^a\r\}\\
&\quad\times\exp\lf\{-\nu\lf[\frac{\max\{d(x, \CY^k),\,d(y,\CY^k)\}}{\dz^k}\r]^a\r\};\noz
\end{align}
\item (the \emph{regularity condition}) for any $x,\ x',\ y\in X$ with $d(x,x')\le\dz^k$,
\begin{align}\label{eq:etiregx}
&|Q_k(x,y)-Q_k(x',y)|+|Q_k(y,x)-Q_k(y, x')|\\
&\quad\le C\lf[\frac{d(x,x')}{\dz^k}\r]^\eta
\frac1{\sqrt{V_{\dz^k}(x)\,V_{\dz^k}(y)}} \exp\lf\{-\nu\lf[\frac{d(x,y)}{\dz^k}\r]^a\r\}\noz\\
&\qquad\times\exp\lf\{-\nu\lf[\frac{\max\{d(x, \CY^k),\,d(y,\CY^k)\}}{\dz^k}\r]^a\r\};\noz
\end{align}
\item (the \emph{second difference regularity condition}) for any $x,\ x',\ y,\ y'\in X$ with $d(x,x')\le\dz^k$
and $d(y,y')\le\dz^k$, then
\begin{align}\label{eq:etidreg}
&|[Q_k(x,y)-Q_k(x',y)]-[Q_k(x,y')-Q_k(x',y')]|\\
&\quad \le C\lf[\frac{d(x,x')}{\dz^k}\r]^\eta\lf[\frac{d(y,y')}{\dz^k}\r]^\eta
\frac1{\sqrt{V_{\dz^k}(x)\,V_{\dz^k}(y)}} \exp\lf\{-\nu\lf[\frac{d(x,y)}{\dz^k}\r]^a\r\}\noz\\
&\qquad\times\exp\lf\{-\nu\lf[\frac{\max\{d(x, \CY^k),\,d(y,\CY^k)\}}{\dz^k}\r]^a\r\};\noz
\end{align}
\item (the \emph{cancelation condition}) for any $x\in X$,
\begin{equation*}
\int_X Q_k(x,y)\,d\mu(y)=0=\int_X Q_k(y,x)\,d\mu(y);
\end{equation*}
\end{enumerate}
\item $Q_0$ satisfies (a), (b) and (c) in (ii) with $k=0$ but without the term
$$
\exp\lf\{-\nu\lf[\max\lf\{d\lf(x,\CY^0\r),d\lf(y,\CY^0\r)\r\}\r]^a\r\};
$$
moreover, for any $x\in X$, $\int_X Q_k(x,y)\,d\mu(y)=1=\int_X Q_k(y,x)\,d\mu(y)$.
\end{enumerate}
\end{definition}

\begin{remark}
We point out that \cite[Lemma 10.1]{ah13} and the proof of \cite[Theorem 10.2]{ah13} guarantee the existence
of an $\exp$-IATI on $X$ no matter $\diam X=\fz$ or not. Using this, we can establish the inhomogeneous
Calder\'{o}m reproducing formulae (see \cite[Section 6]{hlyy18} or Theorems \ref{thm-icrf} and \ref{thm-idrf}
below). Here and hereafter, the \emph{H\"{o}lder regularity index $\eta$} is the same as the H\"{o}lder
regularity index of the wavelet system appearing in \cite[Theorem 7.1]{ah13}.
\end{remark}

Next we recall the inhomogeneous continuous Calder\'{o}n reproducing formula established in \cite{hlyy18}.
\begin{theorem}\label{thm-icrf}
Let $\{Q_k\}_{k=0}^\fz$ be an $\exp$-{\rm IATI} and $\bz,\ \gz\in(0,\eta)$. Then there exist $N\in\nn$ and
a sequence $\{\wz Q_k\}_{k=0}^\fz$ of bounded linear operators on $L^2(X)$ such that, for any
$f\in(\go{\bz,\gz})'$,
\begin{equation}\label{eq-icrf}
f=\sum_{k=0}^\fz\wz Q_kQ_kf,
\end{equation}
where the series converges in $(\go{\bz,\gz})'$. Moreover, there exists a positive constant $C$ such that,
for any $k\in\zz_+$, the kernel of $\wz Q_k$, still denoted by $\wz Q_k$, has the following properties:
\begin{enumerate}
\item for any $x,\ y\in X$,
$$
\lf|\wz Q_k(x,y)\r|\le C\frac{1}{V_{\dz^k}(x)+V(x,y)}\lf[\frac{\dz^k}{\dz^k+d(x,y)}\r]^\gz;
$$
\item for any $x,\ x',\ y\in X$ with $d(x,x')\le(2A_0)^{-1}[\dz^k+d(x,y)]$,
$$
\lf|\wz Q_k(x,y)-\wz Q_k(x',y)\r|\le C\lf[\frac{d(x,x')}{\dz^k+d(x,y)}\r]^\bz
\frac{1}{V_{\dz^k}(x)+V(x,y)}\lf[\frac{\dz^k}{\dz^k+d(x,y)}\r]^\gz;
$$
\item for any $x\in X$,
$$
\int_X \wz Q_k(x,y)\,d\mu(y)=\int_X \wz Q_k(y,x)\,d\mu(y)=\begin{cases}
1 & \textit{if}\ k\in\{0,\ldots,N\},\\
0 & \textit{if}\ k\in\{N+1,N+2,\ldots\}.
\end{cases}
$$
\end{enumerate}
\end{theorem}
Next we recall the inhomogeneous discrete Calder\'{o}n reproducing formulae established in \cite{hlyy18}. To
this end, let $j_0\in\nn$ be a sufficiently large integer such that $\dz^{j_0}\le(2A_0)^{-4}C^\natural$,
where $C^\natural$ is as in Lemma \ref{cube}. Based on Lemma \ref{cube}, for any $k\in\zz$ and $\az\in\CA_k$,
we let
$$
\CN(k,\az):=\{\tau\in\CA_{k+j_0}:\ Q_\tau^{k+j_0}\subset Q_\az^k\}
$$
and $N(k,\az)$ be the \emph{cardinality} of the set $\CN(k,\az)$. For any $k\in\zz$ and $\az\in\CA_k$, we
rearrange the set $\{Q_\tau^{k+j_0}:\ \tau\in\CN(k,\az)\}$ as $\{Q_\az^{k,m}\}_{m=1}^{N(k,\az)}$, whose
centers are denoted, respectively, by $\{z_\az^{k,m}\}_{m=1}^{N(k,\az)}$.

\begin{theorem}\label{thm-idrf}
Let $\{Q_k\}_{k=0}^\fz$ be an $\exp$-{\rm IATI} and $\bz,\ \gz\in(0,\eta)$. Then there exist $N,\ j_0\in\nn$
such that, for any $y_\az^{k,m}\in Q_\az^{k,m}$ with $k\in\{N+1,N+2,\ldots\}$, $\az\in\CA_k$ and
$m\in\{1,\ldots,N(k,\az)\}$, there exists a sequence $\{\wz Q_k\}_{k=0}^\fz$ of bounded linear operators on
$L^2(X)$ such that, for any $f\in(\GOO{\bz,\gz})'$,
\begin{align}\label{eq-idrf}
f(\cdot)&=\sum_{k=0}^N\sum_{\az\in\CA_k}\sum_{m=1}^{N(k,\az)}\int_{Q_\az^{k,m}}\wz{Q}_k(\cdot,y)\,d\mu(y)
Q_{\az,1}^{k,m}(f)\\
&\quad+\sum_{k=N+1}^\fz\sum_{\az\in\CA_k}\sum_{m=1}^{N(k,\az)}\mu\lf(Q_\az^{k,m}\r)
\wz{Q}_k\lf(\cdot,y_\az^{k,m}\r)Q_kf\lf(y_\az^{k,m}\r),\noz
\end{align}
where \eqref{eq-idrf} converges in $(\GOO{\bz,\gz})'$ and, for any $k\in\{0,\ldots,N\}$, $\az\in\CA_k$ and
$m\in\{1,\ldots,N(k,\az)\}$,
$$
Q_{\az,1}^{k,m}(f):=\frac 1{\mu(Q_\az^{k,m})}\int_{Q_\az^{k,m}}Q_kf(u)\,d\mu(u).
$$
Moreover, for any $k\in\zz_+$, the kernel of $\wz{Q}_k$ satisfies (i), (ii) and (iii) of Theorem
\ref{thm-icrf}, with the implicit positive constant independent of the choice of $y_\az^{k,m}$.
\end{theorem}

\section{Local Hardy spaces via maximal functions}\label{s-max}

In this section, we introduce Hardy spaces via maximal functions. To this end, we first introduce
$1$-$\exp$-ITAI, which is used to define the radial and the non-tangential maximal functions.
\begin{definition}\label{def-1ieti}
A sequence $\{P_k\}_{k=0}^\fz$ of bounded linear operators on $L^2(X)$ is called an \emph{inhomogeneous
approximation of the identity with exponential decay and integration $1$} (for short, $1$-$\exp$-IATI) if
$\{P_k\}_{k=0}^\fz$ has the following properties:
\begin{enumerate}
\item for any $k\in\zz_+$, $P_k$ satisfies (a) and (b) of Definition \ref{def-iati}(ii) but without the term
$$
\exp\lf\{-\nu\lf[\frac{\max\{d(x, \CY^k),\,d(y,\CY^k)\}}{\dz^k}\r]^a\r\};
$$
\item (the \emph{conservation property}) for any $k\in\zz_+$ and $x\in X$,
$$
\int_X P_k(x,y)\,d\mu(y)=1=\int_X P_k(y,x)\,d\mu(y);
$$
\item letting $Q_0=P_0$ and $Q_k=P_k-P_{k-1}$ for any $k\in\nn$, then $\{Q_k\}_{k=0}^\fz$ is an $\exp$-IATI.
\end{enumerate}
\end{definition}
Fix $\bz,\ \gz\in(0,\eta)$. Let $\{P_k\}_{k=0}^\fz$ be a $1$-$\exp$-IATI as in Definition
\ref{def-1ieti}. Define the \emph{local radial maximal function $\CM^+_0(f)$} of $f$ by setting, for any
$x\in X$,
$$
\CM^+_0(f)(x):=\max\lf\{\max_{k\in\{0,\ldots,N\}}\lf\{\sum_{\az\in\CA_k}\sum_{m=1}^{N(k,\az)}
\sup_{z\in Q_\az^{k,m}}|P_kf(z)|\mathbf{1}_{Q_\az^{k,m}}(x)\r\},\sup_{k\in\{N,N+1,\ldots\}}|P_kf(x)|\r\}.
$$
Define the \emph{local non-tangential maximal function $\CM_{\thz,0}(f)$ of $f$ with aperture
$\thz\in(0,\fz)$} by setting, for any $x\in X$,
$$
\CM_{\thz,0}(f)(x):=\sup_{k\in\zz_+}\sup_{y\in B(x,\thz\dz^k)}|P_kf(y)|.
$$
Also, define the \emph{local grand maximal function $f^*$} of $f$ by setting, for any $x\in X$,
$$
f^*_0(x):=\sup\lf\{|\langle f,\vz \rangle|:\ \vz\in\go{\bz,\gz} \textup{ and } \|\vz\|_{\CG(x,r_0,\bz,\gz)}
\le 1\textup{ for some } r_0\in(0,1]\r\}.
$$
Corresponding, for any $p\in(0,\fz]$, the \emph{local Hardy spaces} $h^{+,p}(X)$, $h^p_\thz(X)$ with
$\thz\in(0,\fz)$ and $h^{*,p}(X)$ are defined, respectively, by setting
\begin{align*}
h^{+,p}(X)&:=\lf\{f\in\lf(\go{\bz,\gz}\r)':\ \|f\|_{h^{+,p}(X)}:=\|\CM^+_0(f)\|_{L^p(X)}<\fz\r\},\\
h^p_\thz(X)&:=\lf\{f\in\lf(\go{\bz,\gz}\r)':\ \|f\|_{h_\thz^{p}(X)}:=\|\CM_{\thz,0}(f)\|_{L^p(X)}<\fz\r\}
\end{align*}
and
$$
h^{*,p}(X):=\lf\{f\in\lf(\go{\bz,\gz}\r)':\ \|f\|_{h^{*,p}(X)}:=\|f^*_0\|_{L^p(X)}<\fz\r\}.
$$
By the definitions of the above local maximal functions, we easily conclude that, for any fixed
$\bz,\ \gz\in(0,\eta)$ and $\thz\in(0,\fz)$, there exists a positive constant $C$ such that, for any
$f\in(\go{\bz,\gz})'$ and $x\in X$,
\begin{equation}\label{eq-bmax}
\max\lf\{\CM^+_0(f)(x),\CM_{\thz,0}(f)(x)\r\}\le Cf^*_0(x).
\end{equation}

\begin{remark}\label{rem-max}
Recall that, in \cite{hhllyy18}, the \emph{radial maximal function} $\CM^+(f)$ of $f\in(\go{\bz,\gz})'$ is
defined by setting, for any $x\in X$,
$$
\CM^+(f)(x):=\sup_{k\in\zz} |P_kf(x)|,
$$
where $\{P_k\}_{k\in\zz}$ is a $1$-$\exp$-ATI (see \cite[Definition 2.8]{hhllyy18}). The main difference
between $\CM^+(f)$ and $\CM^+_0(f)$ lies that, in the definition of the local radial maximal function
$\CM^+_0(f)$, we use the term
$$
\max_{k\in\{0,\ldots,N\}}\sum_{\az\in\CA_k}\sum_{m=1}^{N(k,\az)}
\sup_{z\in Q_\az^{k,m}}|P_kf(z)|\mathbf{1}_{Q_\az^{k,m}}
$$
to replace $\max_{k\in\{0,\ldots,N\}}|P_kf(y)|$ appearing in the definition of the radial maximal function
$\CM^+(f)$. This is because the discrete inhomogeneous reproducing
formula has the term $Q_{\az,1}^{k,m}(f)$ [see \eqref{eq-idrf}]. Similar way is used to define the
inhomogeneous Littlewood--Paley $g$-functions; see Section \ref{s-lp} below and \cite{hmy08,hy02}.
\end{remark}

Using a similar argument to that used in the proof of \cite[Theorem 3.4]{hhllyy18}, we obtain the following
relation between the local Hardy spaces and the Lebesgue space when $p\in(1,\fz]$ and we omit the details here.

\begin{theorem}\label{thm-p>1}
Let $p\in(1,\fz]$ and $\bz,\ \gz\in(0,\eta)$ with $\eta$ as in Definition \ref{def-iati}. Then, for any
fixed $\thz\in(0,\fz)$, as subspaces of $(\go{\bz,\gz})'$, $h^{+,p}(X)=h^p_\thz(X)=h^{*,p}(X)=L^p(X)$ in the
sense of representing the same distributions and having equivalent norms.
\end{theorem}
Next we consider the case $p\in(\om,1]$. Indeed, we have the following conclusion.
\begin{proposition}\label{prop-rmax}
Let $p\in(\om,1]$ and $\bz,\ \gz\in(\omega(1/p-1),\eta)$, with $\omega$ and $\eta$, respectively, as in
\eqref{eq-doub} and Definition \ref{def-iati}. Then, as subspaces of $(\go{\bz,\gz})'$,
$h^{*,p}(X)=h^{+,p}(X)$ in the sense of equivalent norms.
\end{proposition}

To show Proposition \ref{prop-rmax}, we need two technical lemmas.

\begin{lemma}[{\cite[Lemma 3.7]{hhllyy18} or \cite[Lemma 3.12]{gly08}}]\label{lem-max}
Let all the notation be as in Theorem \ref{thm-idrf}. Let $k,\ k'\in\zz_+$,
$\{a_\az^{k,m}\in\cc:\ k\in\zz_+,\ \az\in\CA_k,\ m\in\{1,\ldots,N(k,\az)\}\}$,
$\gz\in(0,\eta)$ and $r\in({\omega}/{(\omega+\gz)},1]$, with $\omega$ and $\eta$, respectively, as in
\eqref{eq-doub} and Definition \ref{def-iati}. Then there exists a positive constant $C$,
independent of $k$, $k'$, $y_\az^{k,m}\in Q_\az^{k,m}$ and $a_\az^{k,m}$
with $k\in\zz_+$, $\az\in\CA_k$ and $m\in\{1,\ldots,N(k,\az)\}$, such that, for any $x\in X$,
\begin{align*}
&\sum_{\az\in\CA_k}\sum_{m=1}^{N(k,\az)}\mu\lf(Q_\az^{k,m}\r)\frac{1}{V_{\dz^{k\wedge k'}}(x)+V(x,y_\az^{k,m})}
\lf[\frac{\dz^{k\wedge k'}}{\dz^{k\wedge k'}+d(x,y_\az^{k,m})}\r]^\gz\lf|a_\az^{k,m}\r|\\
&\quad\le C\dz^{[k-(k\wedge k')]\omega(1-1/r)}\lf[\CM\lf(\sum_{\az\in\CA_k}\sum_{m=1}^{N(k,\az)}
\lf|a_\az^{k,m}\r|^r\mathbf{1}_{Q_\az^{k,m}}\r)(x)\r]^{1/r},
\end{align*}
where $\CM$ is as in \eqref{2.1x}.
\end{lemma}

\begin{lemma}\label{lem-est1}
Let $N\in\nn$, $\{\wz Q_k\}_{k=0}^\fz$ be a sequence of bounded operators on $L^2(X)$ and
$\bz,\ \gz\in(0,\eta]$, with $\eta$ as in Definition \ref{def-1ieti}. Assume that there exists a positive
constant $C$ such that $\{\wz Q_k\}_{k=0}^\fz$ satisfies (i), (ii) and (iii) of Theorem \ref{thm-icrf}. Then,
for any fixed $\eta'\in(0,\bz\wedge\gz)$, there exists a positive constant $C_1$, depending on $N$,
such that, for any $x,\ y\in X$, $k,\ l\in\zz_+$ and $\phi\in\CG(x,\dz^l,\bz,\gz)$ satisfying
$\int_X\phi(z)\,d\mu(z)=0$ when $l\in\{N+1,N+2,\ldots\}$,
\begin{equation}\label{eq-est1}
\lf|\wz Q_k^*\phi(y)\r|\le C_1\|\phi\|_{\CG(x,\dz^l,\bz,\gz)}\dz^{|k-l|\eta'}
\frac 1{V_{\dz^{k\wedge l}}(x)+V(x,y)}\lf[\frac{\dz^{k\wedge l}}{\dz^{k\wedge l}+d(x,y)}\r]^\gz.
\end{equation}
Here and hereafter, $\wz Q_k^*$ denotes the \emph{dual operator} of $\wz Q_k$ on $L^2(X)$.
\end{lemma}

\begin{proof}
Let all the notation be the same as in this lemma. Without loss of generality, we may
assume $\|\phi\|_{\CG(x,\dz^l,\bz,\gz)}=1$. Then we consider three case.

{\it Case 1) $k,\ l\in\{0,\ldots,N\}$}. In this case, we have $\dz^k\sim 1\sim\dz^l\sim\dz^{k\wedge l}\sim
\dz^{|k-l|}$. Observe that
$$
X\subset\{z\in X:\ d(z,y)\ge(2A_0)^{-1}d(x,y)\}\cup\{z\in X:\ d(z,x)\ge(2A_0)^{-1}d(x,y)\}.
$$
Therefore, by this, the definition of $\wz Q_k^*\phi$ and Theorem \ref{thm-icrf}(i), we conclude
that, for any fixed $\eta'\in(0,\bz\wedge\gz)$ and for any $y\in X$,
\begin{align*}
\lf|\wz Q_k^*\phi(y)\r|&=\lf|\int_X\wz Q_k(z,y)\phi(y)\,d\mu(z)\r|\\
&\ls\int_X\frac 1{V_{\dz^k}(z)+V(z,y)}\lf[\frac{\dz^k}{\dz^k+d(z,y)}\r]^\gz
\frac{1}{V_{\dz^l}(x)+V(x,z)}\lf[\frac{\dz^l}{\dz^l+d(x,z)}\r]^\gz\,d\mu(z)\\
&\ls\frac 1{V_{1}(z)+V(x,y)}\lf[\frac{1}{1+d(x,y)}\r]^\gz\lf\{\int_{d(z,y)\ge(2A_0)^{-1}d(x,y)}
\frac{1}{V_{\dz^l}(x)+V(x,z)}\lf[\frac{\dz^l}{\dz^l+d(x,z)}\r]^\gz\,d\mu(z)\r.\\
&\quad+\lf.
\int_{d(z,x)\ge(2A_0)^{-1}d(x,y)}\frac 1{V_{\dz^k}(z)+V(z,y)}\lf[\frac{\dz^k}{\dz^k+d(z,y)}\r]^\gz
\,d\mu(z)\r\}\\
&\ls\dz^{|l-k|\bz}\frac 1{V_{\dz^{k\wedge l}}(x)+V(x,y)}
\lf[\frac{\dz^{k\wedge l}}{\dz^{k\wedge l}+d(x,y)}\r]^\gz,
\end{align*}
which implies \eqref{eq-est1} in this case.

{\it Case 2) $l\in\{N+1,N+2,\ldots\}$ and $k\le l$.} In this case, we have $\dz^k\ge\dz^l$. By the
cancellation of $\phi$, we conclude that, for any $y\in X$,
\begin{align*}
\lf|\wz Q_k^*\phi(y)\r|&\le\int_X\lf|\wz Q_k(x,y)-\wz Q_k(z,y)\r||\phi(z)|\,d\mu(z)\\
&=\int_{d(z,x)\le(2A_0)^{-1}[\dz^k+d(x,y)]}\lf|\wz Q_k(x,y)-\wz Q_k(z,y)\r||\phi(z)|\,d\mu(z)+
\int_{d(z,x)>(2A_0)^{-1}[\dz^k+d(x,y)]}\cdots\\
&=:\RA_1+\RA_2.
\end{align*}
For the term $\RA_1$, by the size condition of $\phi$ and Theorem \ref{thm-icrf}(ii), we have
\begin{align*}
\RA_1&\ls\frac 1{V_{\dz^k}(x)+V(x,y)}\lf[\frac{\dz^k}{\dz^k+d(x,y)}\r]^\gz\\
&\quad\times\int_{d(x,z)\le(2A_0)^{-1}[\dz^k+d(x,y)]}\lf[\frac{d(x,z)}{\dz^k+d(x,y)}\r]^\eta
\frac 1{V_{\dz^l}(x)+V(x,z)}\lf[\frac{\dz^l}{\dz^l+d(x,z)}\r]^\gz\,d\mu(z)\\
&\ls\dz^{(l-k)\eta'}\frac 1{V_{\dz^k}(x)+V(x,y)}\lf[\frac{\dz^k}{\dz^k+d(x,y)}\r]^\gz
\int_X\frac 1{V_{\dz^l}(x)+V(x,z)}\lf[\frac{\dz^l}{\dz^l+d(x,z)}\r]^{\gz-\eta'}\,d\mu(z)\\
&\sim\dz^{(l-k)\eta'}\frac 1{V_{\dz^k}(x)+V(x,y)}\lf[\frac{\dz^k}{\dz^k+d(x,y)}\r]^\gz.
\end{align*}
For the term $\RA_2$, by the size conditions of $\wz Q_k$ and $\phi$, we conclude that
\begin{align*}
\RA_2&\le\lf|\wz Q_k(x,y)\r|\int_{d(z,x)>(2A_0)^{-1}[\dz^k+d(x,y)]}|\phi(z)|\,d\mu(z)+
\int_{d(z,x)>(2A_0)^{-1}[\dz^k+d(x,y)]}\lf|\wz Q_k(z,y)\phi(z)\r|\,d\mu(z)\\
&\ls\dz^{(l-k)\gz}\frac 1{V_{\dz^k}(x)+V(x,y)}\lf[\frac{\dz^k}{\dz^k+d(x,z)}\r]^\gz
\lf\{\int_{d(z,x)>(2A_0)^{-1}\dz^k}\frac 1{V(x,z)}\lf[\frac{\dz^k}{d(x,z)}\r]^\gz\,d\mu(z)\r.\\
&\quad\lf.+\int_X\frac 1{V_{\dz^k}(z)+V(z,y)}\lf[\frac{\dz^k}{\dz^k+d(z,y)}\r]^\gz\,d\mu(z)\r\}\\
&\ls\dz^{(l-k)\gz}\frac 1{V_{\dz^k}(x)+V(x,y)}\lf[\frac{\dz^k}{\dz^k+d(x,z)}\r]^\gz.
\end{align*}
Combining the estimates of $\RA_1$ and $\RA_2$, we obtain \eqref{eq-est1} in this case.

{\it Case 3) $k\in\{N+1,N+2,\ldots\}$ and $k>l$.} In this case, we have $\dz^l>\dz^k$. Similarly, by the
cancellation of $\wz Q_k$, we conclude that, for any $y\in X$,
\begin{align*}
\lf|\wz Q_k^*\phi(y)\r|&\le\int_X\lf|\wz Q_k(z,y)\r||\phi(z)-\phi(y)|\,d\mu(z)\\
&=\int_{d(z,y)\le(2A_0)^{-1}[\dz^l+d(x,y)]}\lf|\wz Q_k(z,y)\r||\phi(z)-\phi(y)|\,d\mu(z)+
\int_{d(z,y)>(2A_0)^{-1}[\dz^l+d(x,y)]}\cdots\\
&=:\RB_1+\RB_2.
\end{align*}

For the term $\RB_1$, from the regularity condition of $\phi$ and Theorem \ref{thm-icrf}(i), we deduce that
\begin{align*}
\RB_1&\ls\frac 1{V_{\dz^l}(x)+V(x,y)}\lf[\frac{\dz^l}{\dz^l+d(x,y)}\r]^\gz\\
&\quad\times\int_{d(y,z)\le(2A_0)^{-1}[\dz^l+d(x,y)]}\lf[\frac{d(y,z)}{\dz^l+d(x,y)}\r]^\eta
\frac 1{V_{\dz^k}(y)+V(y,z)}\lf[\frac{\dz^k}{\dz^k+d(y,z)}\r]^\gz\,d\mu(z)\\
&\ls\dz^{(k-l)\eta'}\frac 1{V_{\dz^l}(x)+V(x,y)}\lf[\frac{\dz^l}{\dz^l+d(x,y)}\r]^\gz
\int_X\frac 1{V_{\dz^k}(y)+V(y,z)}\lf[\frac{\dz^k}{\dz^k+d(y,z)}\r]^{\gz-\eta'}\,d\mu(z)\\
&\sim\dz^{(l-k)\eta'}\frac 1{V_{\dz^l}(x)+V(x,y)}\lf[\frac{\dz^l}{\dz^l+d(x,y)}\r]^\gz.
\end{align*}
To establish the desired estimate of $\RB_2$, we use the size conditions of $\wz Q_k$ and $\phi$ to conclude
that
\begin{align*}
\RB_2&\le\lf|\phi(y)\r|\int_{d(z,y)>(2A_0)^{-1}[\dz^l+d(x,y)]}\lf|\wz Q_k(z,y)\r|\,d\mu(z)+
\int_{d(z,y)>(2A_0)^{-1}[\dz^l+d(x,y)]}\lf|\wz Q_k(z,y)\phi(z)\r|\,d\mu(z)\\
&\ls\dz^{(k-l)\gz}\frac 1{V_{\dz^l}(x)+V(x,y)}\lf[\frac{\dz^l}{\dz^l+d(x,z)}\r]^\gz
\lf\{\int_{d(z,y)>(2A_0)^{-1}\dz^l}\frac 1{V(y,z)}\lf[\frac{\dz^l}{d(y,z)}\r]^\gz\,d\mu(z)\r.\\
&\quad\lf.+\int_X\frac 1{V_{\dz^l}(x)+V(x,z)}\lf[\frac{\dz^l}{\dz^l+d(x,z)}\r]^\gz\,d\mu(z)\r\}\\
&\ls\dz^{(k-l)\gz}\frac 1{V_{\dz^l}(x)+V(x,y)}\lf[\frac{\dz^l}{\dz^l+d(x,z)}\r]^\gz.
\end{align*}
Thus, by the estimates of $\RB_1$ and $\RB_2$, we find that \eqref{eq-est1} also holds true in this case.

Combining the previous three cases altogether, we obtain \eqref{eq-est1} and hence complete the proof of Lemma
\ref{lem-est1}.
\end{proof}

Now we show Proposition \ref{prop-rmax}. In what follows, we use $\ez\to 0^+$ to denote $\ez\in(0,\fz)$ and
$\ez\to 0$.

\begin{proof}[Proof of Proposition \ref{prop-rmax}]
Let $p\in(\om,1]$ and $\bz,\ \gz\in(\omega(1/p-1),\eta)$. By \eqref{eq-bmax}, we obviously have
$h^{*,p}(X)\subset h^{+,p}(X)$.

It remains to show $h^{+,p}(X)\subset h^{*,p}(X)$. Suppose that $\{P_k\}_{k=0}^\fz$ is a $1$-$\exp$-IATI and
$f\in h^{*,p}(X)$. Then $f\in(\go{\bz,\gz})'$, $\CM^+_0(f)\in L^p(X)$ and, to obtain the desired
inclusion relations, it suffices to show that, for any
$x\in X$ and $\phi\in\GOO{\bz,\gz}$ satisfying $\|\phi\|_{\CG(x,\dz^l,\bz,\gz)}\le 1$ for some $l\in\zz_+$,
\begin{equation}\label{eq-H=}
|\langle f,\phi\rangle|\ls\lf\{\CM\lf(\lf[\CM^+_0(f)\r]^r\r)(x)\r\}^{1/r}
\end{equation}
holds true for some $r\in(0,p)$, where $\CM$ is as in \eqref{2.1x}.

Indeed, assuming this for the moment, for any $x\in X$ and $\vz\in\go{\bz,\gz}$ satisfying
$\|\vz\|_{\CG(x,r,\bz,\gz)}\le 1$ for some $r\in(0,1]$, choose
$l\in\zz$ such that $\dz^{l+1}<r\le\dz^l$. Then $\|\vz\|_{\CG(x,\dz^l,\bz,\gz)}\ls 1$. Let
$\sigma:=\int_X\phi(z)\,d\mu(z)$ and $\phi(\cdot):=\vz(\cdot)-\sigma P_l(x,\cdot)$. By the size conditions
of $\vz$ and $P_l$, and \cite[Lemma 3.6]{hhllyy18}, we find that
$|\sigma|\ls 1$, $\phi\in\GOO{\bz,\gz}$ and
$$
\|\phi\|_{\CG(x,\dz^l,\bz,\gz)}\ls\|\vz\|_{\CG(x,\dz^l,\bz,\gz)}+|\sigma|\cdot
\|P_k(x,\cdot)\|_{\CG(x,\dz^l,\bz,\gz)}\sim 1.
$$
Therefore, by \eqref{eq-H=}, we conclude that
\begin{align*}
|\langle f,\vz \rangle|=|\langle f,\phi(\cdot)+\sigma P_l(x,\cdot)\rangle|
\le|\langle f,\phi\rangle|+|\sigma||P_lf(x)|\ls\lf\{\CM\lf(\lf[\CM^+_0(f)\r]^r\r)(x)\r\}^{1/r}+\CM_0^+(f)(x).
\end{align*}
By the arbitrariness of $\vz$ and the boundedness of $\CM$ on $L^{p/r}(X)$ (see, for instance
\cite[(3.6)]{cw77}), we obtain $\|f^*\|_{L^p(X)}\ls\|\CM_0^+(f)\|_{L^p(X)}$. This finishes the proof of
Proposition \ref{prop-rmax} under the assumption \eqref{eq-H=}.

Now we show \eqref{eq-H=}. Suppose $x\in X$ and $\phi\in\GOO{\bz,\gz}$ with
$\|\phi\|_{\CG(x,\dz^l,\bz,\gz)}\le 1$. Let $Q_k:=P_k-P_{k-1}$ for any $k\in\nn$ and $Q_0:=P_0$. Then, by
Definition \ref{def-1ieti}, we find that $\{Q_k\}_{k=0}^\fz$ is an $\exp$-IATI and hence, by Theorem
\ref{thm-idrf}, we have
\begin{align}\label{eq-max1}
|\langle f,\phi\rangle|
&\le\sum_{k=0}^N\sum_{\az\in\CA_k}\sum_{m=1}^{N(k,\az)}\lf|\int_{Q_\az^{k,m}}\wz{Q}_k^*\phi(y)\,d\mu(y)
Q_{\az,1}^{k,m}(f)\r|\\
&\quad+\sum_{k=N+1}^\fz\sum_{\az\in\CA_k}\sum_{m=1}^{N(k,\az)}\mu\lf(Q_\az^{k,m}\r)
\lf|\wz{Q}_k^*\phi\lf(y_\az^{k,m}\r)Q_kf\lf(y_\az^{k,m}\r)\r|=:\RY_1+\RY_2,\noz
\end{align}
where $N$, $\wz Q_k$, $Q_{\az,1}^{k,m}$ and $y_\az^{k,m}$ are defined as in Theorem \ref{thm-idrf}
and $\wz Q_k^*$ denotes the \emph{dual operator} of $\wz Q_k$ on $L^2(X)$.

We first estimate $\RY_1$. Indeed, by Lemma \ref{lem-est1}, we find that, for any fixed
$\eta'\in(0,\bz\wedge\gz)$, any $k\in\{0,\ldots,N\}$, $\az\in\CA_k$, $m\in\{1,\ldots,N(k,\az)\}$ and
$y\in Q_\az^{k,m}$
\begin{align}\label{eq-qk1}
\lf|\wz Q_k^*\phi(y)\r|&\ls\dz^{|l-k|\eta'}\frac 1{V_{\dz^{k\wedge l}}(x)+V(x,y)}
\lf[\frac{\dz^{k\wedge l}}{\dz^{k\wedge l}+d(x,y)}\r]^\gz\\
&\sim \dz^{|l-k|\eta'}\frac 1{V_{\dz^{k\wedge l}}(x)+V(x,z_\az^{k,m})}
\lf[\frac{\dz^{k\wedge l}}{\dz^{k\wedge l}+d(x,z_\az^{k,m})}\r]^\gz.\noz
\end{align}
Moreover, for any $y,\ z\in Q_\az^{0,m}$, we have $|Q_0f(y)|=|P_0f(y)|\le \CM^+_0(f)(z)$, which further
implies that $|Q_{\az,1}^{0,m}(f)|\le\inf_{z\in Q_\az^{0,m}}\CM^+_0(f)(z)$. When $k\in\{1,\ldots,N\}$, then
$Q_k=P_k-P_{k-1}$ and hence, for any $y,\ z\in Q_\az^{k,m}$,
$|Q_kf(y)|\le|P_kf(y)|+|P_{k-1}f(y)|\le 2\CM^+_0(f)(z)$. Thus,
$|Q_{\az,1}^{k,m}(f)|\le2\inf_{z\in Q_\az^{k,m}}\CM^+_0(f)(z)$. By this, \eqref{eq-qk1} and Lemma
\ref{lem-max}, we find that, for any fixed $r\in(\omega/(\omega+\gz),p)$,
\begin{align}\label{eq-y1a}
\RY_1&\ls\sum_{k=0}^N\dz^{|k-l|\eta'}\sum_{\az\in\CA_k}\sum_{m=1}^{N(k,\az)}\mu\lf(Q_\az^{k,m}\r)
\frac{1}{V_{\dz^{k\wedge l}}(x)+V(x,z_\az^{k,m})}
\lf[\frac{\dz^{k\wedge l}}{\dz^{k\wedge l}+d(x,z_\az^{k,m})}\r]^\gz
\inf_{z\in Q_\az^{k,m}}\CM^+_0(f)(z)\\
&\ls\sum_{k=0}^N\dz^{|k-l|\eta'}\dz^{[k-(k\wedge l)]\omega(1-1/r)}
\lf\{\CM\lf(\sum_{\az\in\CA_k}\sum_{m=1}^{N(k,\az)}\inf_{z\in Q_\az^{k,m}}\lf[\CM^+_0(f)(z)\r]^r
\mathbf{1}_{Q_\az^{k,m}}\r)(x)\r\}^{1/r}\noz\\
&\ls\sum_{k=0}^N\dz^{|k-l|\eta'}\dz^{[k-(k\wedge l)]\omega(1-1/r)}
\lf\{\CM\lf(\lf[\CM^+_0(f)\r]^r\r)(x)\r\}^{1/r}.\noz
\end{align}
Choose $\eta'$ and $r$ such that $\omega/(\omega+\eta')<r<p$. Therefore,
\eqref{eq-y1a} becomes
\begin{equation}\label{eq-y1}
\RY_1\ls\lf\{\CM\lf(\lf[\CM^+_0(f)\r]^r\r)(x)\r\}^{1/r},
\end{equation}
which is the desired estimate.

Next we estimate $\RY_2$. For any $k\in\{N+1,N+2,\ldots\}$, $\az\in\CA_k$ and $m\in\{1,\ldots,N(k,\az)\}$,
choose $y_\az^{k,m}$ such that, for any fixed $\ez\in(0,\fz)$,
$|Q_kf(y_\az^{k,m})|<\inf_{z\in Q_\az^{k,m}}|Q_kf(z)|+\ez$, which implies that
$$
\lf|Q_kf\lf(y_\az^{k,m}\r)\r|<\inf_{z\in Q_\az^{k,m}}[|P_kf(z)|+|P_{k-1}f(z)|]+\ez
\le 2\inf_{z\in Q_\az^{k,m}}\CM^+_0(f)(z)+\ez.
$$
By this, Lemmas \ref{lem-est1} and \ref{lem-max}, we conclude that, for any fixed $\eta'\in(0,\bz\wedge\gz)$
and $r\in(\omega/(\omega+\gz),p)$,
\begin{align*}
\RY_2&\ls\sum_{k=N+1}^\fz\dz^{|k-l|\bz'}\sum_{\az\in\CA_k}\sum_{m=1}^{N(k,\az)}\mu\lf(Q_\az^{k,m}\r)
\frac {\inf_{z\in Q_{\az}^{k,m}}\CM^+_0(f)(z)+\ez}{V_{\dz^{k\wedge l}}(x)+V(x,y_\az^{k,m})}
\lf[\frac{\dz^{k\wedge l}}{\dz^{k\wedge l}+d(x,y_\az^{k,m})}\r]^\gz\\
&\ls\sum_{k=N+1}^\fz\dz^{|k-l|\eta'}\dz^{[k-(k\wedge l)]\omega(1-1/r)}
\lf\{\CM\lf(\sum_{\az\in\CA_k}\sum_{m=1}^{N(k,\az)}\lf[\inf_{z\in Q_{\az}^{k,m}}\CM^+_0(f)(z)+\ez\r]^r
\mathbf{1}_{Q_\az^{k,m}}\r)(x)\r\}^{\frac 1r}\\
&\ls\sum_{k=N+1}^\fz\dz^{|k-l|\eta'}\dz^{[k-(k\wedge l)]\omega(1-1/r)}
\lf\{\CM\lf(\lf[\CM^+_0(f)\r]^r\r)(x)+\ez^r\r\}^{1/r}.
\end{align*}
Choosing $\eta'$ such that $\omega/(\omega+\eta')<p$ and $r$ as in \eqref{eq-y1} and taking $\ez\to 0^+$, we
have
\begin{equation*}
\RY_2\ls \lf\{\CM\lf(\lf[\CM^+_0(f)\r]^r\r)(x)\r\}^{1/r}.
\end{equation*}
Combining this with \eqref{eq-y1} and \eqref{eq-max1}, we obtain \eqref{eq-H=}, which completes the proof of
Proposition \ref{prop-rmax}.
\end{proof}

Next we show $h^{*,p}(X)=h^{p}_{\thz}(X)$ for any given $p\in(\om,1]$ and $\thz\in(0,\fz)$. Note that, in
the global case, we have \cite[(3.1) and (3.4)]{hhllyy18}, which imply this equivalence
directly. However, in the local case, it is still \emph{unknown} whether or not there exists a positive
constant $C$ such that, for any $f\in(\go{\bz,\gz})'$ and $x\in X$,
$$
\CM_0^+(f)(x)\le C\CM_{\thz,0}(f)(x).
$$
This occurs because of the different structures of the radial maximal functions between the global
case and the local case. Thus, we need the following proposition.

\begin{proposition}\label{prop-tmax}
Let $p\in(\om,1]$, $\bz,\ \gz\in(\omega(1/p-1),\eta)$ and $\thz\in(0,\fz)$. Then, as subspaces of
$(\go{\bz,\gz})'$, $h^{*,p}(X)=h^p_\thz(X)$ in the sense of equivalent norms.
\end{proposition}
\begin{proof}
For any fixed $\thz\in(0,\fz)$, in Theorem \ref{thm-idrf}, we may assume $j_0$ sufficiently large such that
$4A_0^2C^\natural\dz^{j_0}\le\thz$. Thus, for such a fixed $j_0$, any $k\in\zz_+$, $\az\in\CA_k$ and
$m\in\{1,\ldots,N(k,\az)\}$, by Lemma \ref{cube}(v), we have
$$
Q_\az^{k,m}\subset B\lf(z_\az^{k,m},2A_0C^\natural\dz^{k+j_0}\r)
\subset B\lf(z,4A_0^2C^\natural\dz^{k+j_0}\r)\subset B(z,\thz\dz^k)\subset B(z,\thz\dz^{k-1})
$$
holds true for any $z\in Q_\az^{k,m}$, which further implies that, for any $z\in Q_\az^{k,m}$,
$$
\lf|Q_{\az,1}^{k,m}(f)\r|\le\frac{1}{\mu(Q_\az^{k,m})}\int_{Q_\az^{k,m}}[|P_kf(u)|+|P_{k-1}f(u)|]\,d\mu(u)
\le 2\CM_{\thz,0}(f)(z)
$$
and, for any $y_\az^{k,m}$, $z\in Q_\az^{k,m}$, $|Q_kf(y_\az^{k,m})|\le\CM_{\thz,0}(f)(z)$. Thus, we have
\begin{equation}\label{eq-Qest}
\lf|Q_{\az,1}^{k,m}(f)\r|+\lf|Q_kf\lf(y_\az^{k,m}\r)\r|\ls\inf_{z\in Q_{\az}^{k,m}}\CM_{\thz,0}(f)(z).
\end{equation}
Using this and following the proof of \eqref{eq-H=}, we find that, for any $f\in(\go{\bz,\gz})'$ with
$\bz,\ \gz\in(\omega(1/p-1),\eta)$, $x\in X$ and $\phi\in\GOO{\bz,\gz}$ with
$\|\phi\|_{\CG(x,\dz^l,\bz,\gz)}\le 1$,
\begin{equation}\label{eq-max2}
|\langle f,\phi\rangle|\le\lf\{\CM\lf(\lf[\CM_{\thz,0}(f)\r]^r\r)(x)\r\}^{1/r}.
\end{equation}
Now, fix $x\in X$ and suppose that $\vz\in\go{\bz,\gz}$ with $\|\vz\|_{\CG(x,r_0,\bz,\gz)}\le 1$ for some
$r_0\in(0,1]$. Let $\sigma:=\int_X\vz(y)\,d\mu(y)$. Then $|\sigma|\ls 1$. Choose $l\in\zz_+$ such that
$\dz^{l+1}<r_0\le\dz^l$ and let $\phi(\cdot):=\vz(\cdot)-\sigma P_l(x,\cdot)$. Then, by
\cite[Lemma 3.6]{hhllyy18}, we have $\phi\in\GOO{\bz,\gz}$ and $\|\phi\|_{\CG(x,\dz^l,\bz,\gz)}\ls 1$.
Therefore, by \eqref{eq-max2}, we find that, for some $r\in(0,p)$ and any $x\in X$,
$$
|\langle f,\vz\rangle|\le|\langle f,\phi\rangle|+|P_l(f)(x)|\ls
\lf\{\CM\lf(\lf[\CM_{\thz,0}(f)\r]^r\r)(x)\r\}^{1/r}+\CM_{\thz,0}(f)(x),
$$
which, by the arbitrariness of $\vz$, further implies that
$$
f^*_0(x)\ls\lf\{\CM\lf(\lf[\CM_{\thz,0}(f)\r]^r\r)(x)\r\}^{1/r}+\CM_{\thz,0}(f)(x).
$$
Taking the $L^p(X)$ norm on the both sides of the above inequality and using the boundedness of $\CM$ on
$L^{p/r}(X)$ (see, for instance \cite[(3.6)]{cw77}), we obtain
$\|f^*_0\|_{L^p(X)}\ls\|\CM_{\thz,0}(f)\|_{L^p(X)}$.
This, combined with \eqref{eq-bmax}, then finishes the proof of Proposition \ref{prop-tmax}.
\end{proof}
Combining Propositions \ref{prop-rmax} and \ref{prop-tmax}, we immediately obtain the following theorem
and we omit the details.
\begin{theorem}\label{thm-max}
Let $p\in(\om,1]$, $\bz,\ \gz\in(\omega(1/p-1),\eta)$ and $\thz\in(0,\fz)$. As subspaces of
$(\go{\bz,\gz})'$, $h^{+,p}(X)=h^{*,p}(X)=h^p_\thz(X)$ with equivalent norms.
\end{theorem}

The next proposition shows that the space $h^{*,p}(X)$ is independent of the choice of $\bz$ and $\gz$, whose
proof is similar to that of \cite[Theorem 3.8]{hhllyy18}; we omit the details here.

\begin{proposition}\label{prop-hpin}
Let $p\in(\om,1]$ and $\bz_1,\ \bz_2,\ \gz_1,\ \gz_2\in(\omega(1/p-1),\eta)$. If
$f\in(\go{\bz_1,\gz_1})'$ belongs to $h^{*,p}(X)$, then $f\in(\go{\bz_2,\gz_2})'$ and there exists a
positive constant $C$, independent of $f$, such that $\|f\|_{(\go{\bz_2,\gz_2})'}\le C\|f\|_{h^{*,p}(X)}$.
\end{proposition}

\section{Maximal function characterizations of atomic local Hardy spaces}\label{s-at}

In this section, we discuss the atomic characterizations of $h^{*,p}(X)$ with $p\in(\om,1]$. First we
introduce the ``basic'' elements of $h^{*,p}(X)$. To distinguish them from those of global Hardy spaces,
we call them \emph{local atoms}.
\begin{definition}\label{def-lat}
\begin{enumerate}
\item Let $p\in(0,1]$ and $q\in(p,\fz]\cap[1,\fz]$. A function $a\in L^q(X)$ is called a
\emph{local $(p,q)$-atom supported on a ball $B:=B(x_0,r_0)$} for some $x_0\in X$ and $r_0\in (0,\fz)$ if $a$
has the following properties:
\begin{enumerate}
\item[$\textup{(i)}_1$] $\supp a:=\{x\in X:\ a(x)\neq 0\}\subset B$;
\item[$\textup{(i)}_2$] $\|a\|_{L^q(X)}\le[\mu(B)]^{1/q-1/p}$;
\item[$\textup{(i)}_3$] $\int_X a(x)\,d\mu(x)=0$ if $r_0\in(0,1]$.
\end{enumerate}
\item Let $p\in(\om,1]$, $q\in(p,\fz]\cap[1,\fz]$ and $\bz,\ \gz\in(\omega(1/p-1),\eta)$, with $\omega$ and
$\eta$, respectively, as in \eqref{eq-doub} and Definition \ref{def-iati}. The \emph{atomic local
Hardy space $h^{p,q}_\at(X)$} is defined to be the set of all $f\in(\go{\bz,\gz})'$ such that
$f=\sum_{j=1}^\fz \lz_ja_j$ in $(\go{\bz,\gz})'$, where $\{a_j\}_{j=1}^\fz$ are local $(p,q)$-atoms and
$\{\lz_j\}_{j=1}^\fz\subset\cc$ satisfies $\sum_{j=1}^\fz|\lz_j|^p<\fz$. For any $f\in h^{p,q}_\at(X)$, let
$$
\|f\|_{h^{p,q}_\at(X)}:=\inf\lf\{\lf(\sum_{j=1}^\fz|\lz_j|^p\r)^{1/p}\r\},
$$
where the infimum is taken over all the decompositions of $f$ as above.
\end{enumerate}
\end{definition}

Recall that, when $\mu(X)=\fz$, letting $p$ and $q$ be as in Definition \ref{def-lat}(i), then a function
$a\in L^q(X)$ is called a \emph{$(p,q)$-atom supported on a ball $B:=B(x_0,r_0)$} for some $x_0\in X$ and
$r_0\in(0,\fz)$ if $a$ satisfies $\textup{(i)}_1$ and $\textup{(i)}_2$ of Definition \ref{def-lat}(i),
and $\int_X a(x)\,d\mu(x)=0$. The definition of $(p,q)$-atoms when $\mu(X)<\fz$ can be seen in Definition
\ref{def-atomb} below.

The next lemma gives the properties of $a_0^*$ if $a$ is a local atom.

\begin{lemma}\label{lem-a*}
Let $p\in(\om,1]$ and $q\in(p,\fz]\cap[1,\fz]$, where $\omega$ and $\eta$ are, respectively, as in
\eqref{eq-doub} and Definition \ref{def-iati}. Suppose that $a$ is a local $(p,q)$-atom supported on a ball
$B:=B(x_0,r_0)$ for some $x_0\in X$ and $r_0\in(0,\fz)$. Then there exists a positive constant
$C$, independent of $a$, such that, for any $x\in X$,
\begin{align}\label{eq-add2}
a^*_0(x)\le C\CM(a)(x)\mathbf{1}_{B(x_0,2A_0r_0)}(x)+C\lf[\frac{r_0}{d(x_0,x)}\r]^{\bz\wedge\gz}
\frac{[\mu(B)]^{1-1/p}}{V(x_0,x)}\mathbf{1}_{[B(x_0,2A_0r_0)]^\complement}(x)
\end{align}
and
\begin{align}\label{eq-add3}
\lf\|a^*_0\r\|_{L^p(X)}\le C,
\end{align}
where the local atom $a$ is viewed as a distribution on $\go{\bz,\gz}$ with
$\bz,\ \gz\in(\omega(1/p-1),\eta)$.
\end{lemma}
\begin{proof}
Indeed, if \eqref{eq-add2} holds true, then, for any given $p\in(\om,1]$, by
$\bz,\ \gz\in(\omega(1/p-1),\eta)$ and \eqref{eq-doub}, we find that
\begin{align*}
\lf\|a_0^*\mathbf{1}_{[B(x_0,2A_0r_0)]^\complement}\r\|_{L^p(X)}^p&=
\int_{d(x,x_0)\ge 2A_0r_0}\lf[a_0^*(x)\r]^p\,d\mu(x)\\
&\ls\int_{d(x,x_0)\ge 2A_0r_0}\lf[\frac{r_0}{d(x_0,x)}\r]^{(\bz\wedge\gz)p}
\lf[\frac{1}{\mu(B)}\r]^{1-p}\lf[\frac 1{V(x_0,x)}\r]^p\,d\mu(x)\\
&\ls\sum_{k=1}^\fz 2^{-k(\bz\wedge\gz)p}2^{k\omega(1-p)}\int_{2^kA_0r_0\le d(x,x_0)<2^{k+1}A_0r_0}
\frac 1{V(x_0,x)}\,d\mu(x)\ls 1.
\end{align*}
To estimate $\|a_0^*\mathbf{1}_{B(x_0,2A_0r_0)}\|_{L^p(X)}$, we consider two cases.

{\it Case 1) $p\in(\om,1]$ and $q\in(1,\fz]$.} In this case, by \eqref{eq-add2}, the H\"{o}lder inequality,
\eqref{eq-doub} and the boundedness of $\CM$ on $L^q(X)$ (see, for instance, \cite[(3.6)]{cw77}), we conclude
that
\begin{equation*}
\lf\|a_0^*\mathbf{1}_{B(x_0,2A_0r_0)}\r\|_{L^p(X)}^p
\ls\lf\|\CM(a)\mathbf{1}_{B(x_0,2A_0r_0)}\r\|_{L^p(X)}^p\ls\lf\|\CM(a)\r\|_{L^q(X)}^p
[\mu(B)]^{1-q/p}\ls 1.
\end{equation*}

{\it Case 2) $p\in(\om,1)$ and $q=1$.} In this case, by the boundedness of $\CM$ from $L^1(X)$ to
$L^{1,\fz}(X)$ (see, for instance, \cite[pp.\ 71--72, Theorem 2.1]{cw71}), we have
\begin{align*}
\lf\|a_0^*\mathbf{1}_{B(x_0,2A_0r_0)}\r\|_{L^p(X)}^p
&\ls\int_{B(x_0,2A_0r_0)}[\CM(a)(x)]^p\,d\mu(x)\\
&\ls\int_0^\fz\mu(\{x\in B(x_0,2A_0r_0):\ \CM(a)(x)>\lz\})\,d\lz^p\\
&\ls\int_0^\fz\min\lf\{\mu(B),\frac{\|a\|_{L^1(X)}}{\lz}\r\}\,d\lz^p\\
&\ls\int_0^{\|a\|_{L^1(X)}/\mu(B)}\mu(B)\,d\lz^p+\int_{\|a\|_{L^1(X)}/\mu(B)}^\fz
\|a\|_{L^1(X)}\lz^{-1}\,d\lz^p\\
&\ls\|a\|_{L^1(X)}^p[\mu(B)]^{1-p}\ls 1.
\end{align*}
Therefore, \eqref{eq-add3} holds true and, to complete the proof of Lemma \ref{lem-a*}, it suffices to show
\eqref{eq-add2}.

If $r_0\in (0,1]$, then \eqref{eq-add2} holds true by an argument similar to that used in the
estimation of \cite[(4.1)]{hhllyy18}.
It remains to consider the case $r_0\in(1,\fz)$ of \eqref{eq-add2}. In this case, suppose that $a$ is a local
$(p,q)$-atom supported on the ball $B:=B(x_0,r_0)$ for some $x_0\in X$ and $r_0\in(1,\fz)$. Fix $x\in X$ and
let $\vz\in\go{\bz,\gz}$ satisfy $\|\vz\|_{\CG(x,r,\bz,\gz)}\le 1$ for some
$r\in(0,1]$. If $x\in B(x_0,2A_0r_0)$, then
$$
|\langle a,\vz\rangle|=\lf|\int_X a(y)\vz(y)\,d\mu(y)\r|
\le\int_X |a(y)|\frac{1}{V_r(x)+V(x,y)}\lf[\frac{r}{r+d(x,y)}\r]^\gz\,d\mu(y)\ls\CM(a)(x),
$$
which, together with the arbitrariness of $\vz$, further implies that $a^*_0(x)\ls\CM(a)(x)$.

If $x\notin B(x_0,2A_0r_0)$, then, for any $y\in B(x_0,r_0)$, $d(y,x_0)\le(2A_0)^{-1}d(x,x_0)$ and hence
$$
d(x,y)\ge\frac{1}{A_0}d(x,x_0)-d(y,x_0)\ge\frac{1}{2A_0}d(x_0,x).
$$
By this, the H\"{o}lder inequality and $r\le 1<r_0$, we conclude that
\begin{align*}
|\langle a,\vz\rangle|&\le\int_B |a(y)|\frac{1}{V_r(x)+V(x,y)}\lf[\frac{r}{r+d(x,y)}\r]^\gz\,d\mu(y)
\ls\frac 1{V(x_0,x)}\lf[\frac{r}{d(x_0,x)}\r]^\gz\int_B|a(y)|\,d\mu(y)\\
&\ls\lf[\frac{r}{d(x_0,x)}\r]^\gz \frac{[\mu(B)]^{1-1/p}}{V(x_0,x)}.
\end{align*}
This, combined with the arbitrariness of $\vz$, implies the desired estimate. Thus, \eqref{eq-add2} also holds
true when $a$ is any local $(p,q)$-atom. This finishes the proof of Lemma \ref{lem-a*}.
\end{proof}
By Lemma \ref{lem-a*} and the definition of the atomic local Hardy space, we immediately obtain the following
proposition and we omit the details.
\begin{proposition}\label{prop-atd}
Let $p\in(\om,1]$, $q\in(p,\fz]\cap[1.\fz]$ and $\bz,\ \gz\in(\omega(1/p-1),\eta)$ with $\omega$ and
$\eta$, respectively, as in \eqref{eq-doub} and Definition \ref{def-iati}. If
$f\in(\go{\bz,\gz})'$ belongs to $h^{p,q}_\at(X)$, then $f\in h^{*,p}(X)$ and there exists a positive
constant $C$, independent of $f$, such that $\|f^*_0\|_{L^p(X)}\le C\|f\|_{h^{p,q}_\at(X)}$.
\end{proposition}

Next we show $h^{*,p}(X)\subset h^{p,q}_\at(X)$. To this end, we need the following Calder\'{o}n--Zygmund
decomposition established in \cite[Proposition 4.4]{hhllyy18}.
\begin{proposition}\label{prop-ozdec}
Suppose $\Omega\subsetneqq X$ is a proper open subset such that $\mu(\Omega)\in(0,\fz)$ and $A\in[1,\fz)$.
For any $x\in\Omega$, let
$$
r(x):=\frac{d(x,\Omega^{\complement})}{2AA_0}\in(0,\fz).
$$
Then there exist $L_0\in\nn$ and a sequence $\{x_k\}_{k\in I}\subset \Omega$, where $I$ is a countable index
set, such that
\begin{enumerate}
\item $\{B(x_k,r_k/(5A_0^3))\}_{k\in I}$ is disjoint. Here and hereafter, $r_k:=r(x_k)$ for any $k\in I$;
\item $\bigcup_{k\in I} B(x_k,r_k)=\Omega$ and $B(x_k,Ar_k)\subset\Omega$;
\item for any $x\in\Omega$, $Ar_k\le d(x,\Omega^\complement)\le3AA_0^2r_k$ whenever $x\in B(x_k,r_k)$ and
$k\in I$;
\item for any $k\in I$, there exists $y_k\notin \Omega$ such that $d(x_k,y_k)<3AA_0r_k$;
\item for any given $k\in I$, the number of balls $B(x_j,Ar_j)$ that intersect $B(x_k,Ar_k)$ is at most
$L_0$; moreover, if $B(x_j,Ar_j)\cap B(x_k,Ar_k)\neq\emptyset$, then $(8A_0^2)^{-1}r_k\le r_j\le 8A_0^2r_k$;
\item if, in addition, $\Omega$ is bounded, then, for any $\sigma\in(0,\fz)$, the set
$\{k\in I:\ r_k>\sigma\}$ is finite.
\end{enumerate}
\end{proposition}

We also need the following decomposition of unity from \cite[Proposition 4.5]{hhllyy18}

\begin{proposition}\label{prop-chidec}
Let $\Omega\subsetneqq X$ be a proper open subset such that $\mu(\Omega)\in(0,\fz)$. Suppose that sequences
$\{x_k\}_{k\in I}$ and $\{r_k\}_{k\in I}$ are as in Proposition \ref{prop-ozdec} with $A:=16A_0^4$. Then
there exist non-negative functions $\{\phi_k\}_{k\in I}$ such that
\begin{enumerate}
\item for any $k\in I$, $0\le\phi_k\le 1$ and $\supp\phi_k\subset B(x_k,2A_0r_k)$;
\item $\sum_{k\in I} \phi_k=\mathbf{1}_\Omega$;
\item for any $k\in I$, $\phi_k\ge L_0^{-1}$ in $B(x_k,r_k)$, where $L_0$ is as in Proposition
\ref{prop-ozdec};
\item there exists a positive constant $C$ such that, for any $k\in I$,
$\|\phi_k\|_{\CG(x_k,r_k,\eta,\eta)}\le CV_{r_k}(x_k)$.
\end{enumerate}
\end{proposition}

Similarly to the global case, it is still \emph{unknown} whether or not the level set
$\{x\in X:\ f_0^*(x)>\lz\}$ with $\lz\in(0,\fz)$ is open. To overcome this difficulty, we use the same method
as that used in \cite{hhllyy18}. By the proof of \cite[Theorem 2]{ms79}, we know that there exist
$\thz\in(0,1)$ and a metric $d'$ such that $d'\sim d^\thz$. For any $x\in X$ and $r\in(0,\fz)$, define the
\emph{$d'$-ball} $B'(x,r):=\{y\in X:\ d'(x,y)<r\}$.
Moreover, for any $x\in X,\ \rho\in(0,\fz)$ and $\bz',\ \gz'\in(0,\fz)$,
define the \emph{space $G(x,\rho,\bz',\gz')$} of new test functions to be the set of all measurable functions
$f$ satisfying that there exists a positive constant $C$ such that
\begin{enumerate}
\item (the \emph{size condition}) for any $y\in X$,
$$
|f(y)|\le C\frac 1{\mu(B'(y,\rho+d'(x,y)))}\lf[\frac{\rho}{\rho+d'(x,y)}\r]^{\gz'};
$$
\item (the \emph{regularity condition}) for any $y,\ y'\in X$ satisfying $d(y,y')\le[\rho+d'(x,y)]/2$, then
$$
|f(y)-f(y')|\le C\lf[\frac{d'(y,y')}{\rho+d'(y,y')}\r]^{\bz'}\frac 1{\mu(B'(y,\rho+d'(x,y)))}
\lf[\frac{\rho}{\rho+d'(x,y)}\r]^{\gz'}.
$$
\end{enumerate}
Also, define
$$
\|f\|_{G(x,\rho,\bz',\gz')}:=\inf\{C\in(0,\fz):\ \textup{(i) and (ii) hold true}\}.
$$
Using the same argument as that used right just after \cite[Definition 4.6]{hhllyy18}, we find that
$\CG(x,r,\bz,\gz)=G(x,r^\thz,\bz/\thz,\gz/\thz)$ with equivalent norms, where the positive equivalence
constants are independent of $x$ and $r$. For any $\bz,\ \gz\in(0,\eta)$ and $f\in(\go{\bz,\gz})'$, define
the \emph{modified local grand maximal function $f_0^\star$} of $f$ by setting, for any $x\in X$,
\begin{equation}\label{4.2x}
f^\star_0(x):=\sup\lf\{|\langle f,\vz\rangle|:\ \vz\in\go{\bz,\gz}\textup{ with }
\|\vz\|_{G(x,r^\thz,\bz/\thz,\gz/\thz)}\le 1\textup{ for some } r\in(0,1]\r\}.
\end{equation}
Then $f^\star_0\sim f^*_0$ pointwisely on $X$. For any $\lz\in(0,\fz)$ and $j\in\zz$,
define
\begin{equation}\label{4.3x}
\Omega_\lz:=\{x\in X:\ f^\star_0(x)>\lz\}\qquad \textup{and}\qquad \Omega^j:=\Omega_{2^j}.
\end{equation}
For any $j\in\zz$ with $\Omega^j\subsetneqq X$, let $I_j$ (resp., $\{r^j_k\}_{k\in I_j}$) be defined as $I$
(resp., $\{r_k\}_{k\in I_j}$) in Proposition \ref{prop-ozdec} with $\Omega:=\Omega^j$ and $A:=16A_0^4$
therein. Correspondingly, let $\{\phi^j_k\}_{k\in I_j}$ be defined as $\{\phi_k\}_{k\in I}$ in Proposition
\ref{prop-chidec} with $\Omega:=\Omega^j$. Define
\begin{align}\label{eq-ij1}
I_{j,1}&:=\lf\{k\in I_j:\ r^j_k\le\lf(48A_0^5\r)^{-1}\r\}\\
I_{j,2}&:=I_j\setminus I_{j,1}=\lf\{k\in I_j:\ r^j_k>\lf(48A_0^5\r)^{-1}\r\}\noz\\
I_j^*&:=\lf\{k\in I_j:\ r^j_k\le(2A_0)^{-4}\r\}\noz
\end{align}
and
$$
I_{j,2}^*:=I_{j,2}\cap I_j^*=\lf\{k\in I_j:\ \lf(48A_0^5\r)^{-1}<r^j_k\le(2A_0)^{-4}\r\}.
$$
For any $j\in\zz$ and $k\in I_j$, define the \emph{distribution $b^j_k$} by setting, for any
$\vz\in\go{\bz,\gz}$,
\begin{equation}\label{eq-defbjk}
\lf\langle b^j_k,\vz\r\rangle:=\begin{cases}
\lf\langle f,\Phi^j_k(\vz)\r\rangle & \textup{if } k\in I_j^*,\\
\lf<f,\phi^j_k\vz\r> & \textup{if } k\in I_j\setminus I_j^*,
\end{cases}
\end{equation}
where, for any $k\in I_j^*$, $\vz\in\go{\bz,\gz}$ and $x\in X$,
$$
\Phi^j_k(\vz)(x):=\phi^j_k(x)\lf[\int_X\phi^j_k(z)\,d\mu(z)\r]^{-1}\int_X[\vz(x)-\vz(z)]\phi^j_k(z)\,d\mu(z).
$$
Then we have the following estimates.
\begin{proposition}\label{prop-bjk}
For any $j\in\zz$ such that $\Omega^j\subsetneqq X$ and $k\in I_j$, let $b^j_k$ be as in \eqref{eq-defbjk}.
Then there exists a positive constant $C$ such that, for any $j$ as above, $k\in I_j$ and $x\in X$,
\begin{align}\label{eq-estbjk}
\lf(b^j_k\r)^*_0(x)&\le C2^j\frac{\mu(B(x^j_k,r^j_k))}{\mu(B(x^j_k,r^j_k))+V(x^j_k,x)}
\lf[\frac{r^j_k}{r^j_k+d(x^j_k,x)}\r]^{\bz\wedge\gz}\mathbf{1}_{[B(x^j_k,16A_0^4r^j_k)]^\complement}(x)\\
&\quad+Cf^*_0(x)\mathbf{1}_{B(x^j_k,16A_0^4r^j_k)}(x).\noz
\end{align}
\end{proposition}
\begin{proof}
Let $j\in\zz$ be such that $\Omega^j\subsetneqq X$. When $k\in I_j^*$, we have $r^j_k\le(2A_0)^{-4}<1$.
Using this and a similar argument to that used in the proof of \cite[Proposition 4.7]{hhllyy18},
we directly obtain \eqref{eq-estbjk}, with the details omitted. For the remainder of the proof,
it suffices to consider the case $k\in I_j\setminus I_j^*$. In this case, $r^j_k>(2A_0)^{-4}$. Let
$x\in X$ and $\vz\in\go{\bz,\gz}$ be such that $\|\vz\|_{\CG(x,r,\bz,\gz)}\le 1$ for some $r\in(0,1]$. We
consider two cases.

{\it Case 1) $d(x^j_k,x)<16A_0^4r^j_k$.} In this case, we claim that
\begin{equation}\label{eq-cl1}
\lf\|\phi^j_k\vz\r\|_{\CG(x,r,\bz,\gz)}\ls 1.
\end{equation}
To show this, we first establish the size estimate. By the size conditions of $\phi^j_k$ and $\vz$, we
conclude that, for any $z\in X$,
$$
\lf|\phi^j_k(z)\vz(z)\r|\le\frac{1}{V_r(x)+V(x,z)}\lf[\frac{r}{r+d(x,z)}\r]^\gz,
$$
as desired.

Next we show the regularity condition. Let $z,\ z'\in X$. Because of the size condition of $\phi^j_k\vz$, we
may assume $d(z,z')\le(2A_0)^{-10}[r+d(x,z)]$; the other cases are obviously true. Moreover,
$\phi^j_k(z)\vz(z)-\phi^j_k(z')\vz(z')\neq 0$ implies
$d(z,x^j_k)<(2A_0)^2r^j_k$. If this is not the case, then $d(z,x^j_k)\ge (2A_0)^2r^j_k$.
Since $r\le 1<(2A_0)^4r^j_k$, it follows that
\begin{align*}
d(z,z')&<(2A_0)^{-10}\lf[(2A_0)^4r^j_k+d(x,z)\r]
\le (2A_0)^{-10}\lf[(2A_0)^4r^j_k+A_0d\lf(x,x^j_k\r)+A_0 d\lf(z,x^j_k\r)\r]\\
&<(2A_0)^{-5}d\lf(z,x^j_k\r)
\end{align*}
and hence $d(z',x^j_k)\ge\frac{1}{A_0}d(z,x^j_k)-d(z,z')>\frac{1}{2A_0}d(z,x^j_k)\ge 2A_0r^j_k$. Thus,
$\phi^j_k(z')\vz(z')=0$, which contradicts to $\phi^j_k(z)\vz(z)-\phi^j_k(z')\vz(z')\neq 0$. Therefore,
$d(z,x^j_k)<(2A_0)^2r^j_k$ and hence
$$
d(z,z')\le(2A_0)^{-10}\lf[r+A_0d\lf(x,x^j_k\r)+A_0d\lf(x^j_k,z\r)\r]<(2A_0)^{-1}r^j_k
\le(2A_0)^{-1}\lf[r^j_k+d\lf(x^j_k,z\r)\r].
$$
Moreover, $r+d(x,z)\ls r+d(x,x^j_k)+d(x^j_k,z)\ls r^j_k$. By these inequalities and both the size and the
regularity conditions of $\phi^j_k$ and $\vz$, we obtain
\begin{align*}
&\lf|\phi^j_k(z)\vz(z)-\phi^j_k(z')\vz(z')\r|\\
&\quad\le\phi^j_k(z')|\vz(z)-\vz(z')|+|\vz(z)|\lf|\phi^j_k(z)-\phi^j_k(z')\r|\\
&\quad\ls\lf[\frac{d(z,z')}{r+d(x,z)}\r]^\bz\frac{1}{V_r(x)+V(x,z)}\lf[\frac{r}{r+d(x,z)}\r]^\gz
+\lf[\frac{d(z,z')}{r_k^j}\r]^\bz\frac{1}{V_r(x)+V(x,z)}\lf[\frac{r}{r+d(x,z)}\r]^\gz\\
&\quad\sim\lf[\frac{d(z,z')}{r+d(x,z)}\r]^\bz\frac{1}{V_r(x)+V(x,z)}\lf[\frac{r}{r+d(x,z)}\r]^\gz.
\end{align*}
This implies the regularity condition of $\phi^j_k\vz$. Consequently, we obtain \eqref{eq-cl1} and hence
\eqref{eq-estbjk} when $x\in B(x^j_k,16A_0^4r^j_k)$ by the arbitrariness of $\vz$.

{\it Case 2) $d(x^j_k,x)\ge 16A_0^4r^j_k$.} In this case, it suffices to show that
\begin{equation}\label{eq-cl2}
\lf\|\phi^j_k\vz\r\|_{\CG(y^j_k,r,\bz,\gz)}\ls\frac{\mu(B(x^j_k,r^j_k))}{\mu(B(x^j_k,r^j_k))+V(x^j_k,x)}
\lf[\frac{r^j_k}{r^j_k+d(x^j_k,x)}\r]^\gz,
\end{equation}
where $y^j_k\notin \Omega^j$ satisfies $d(y^j_k,x^j_k)<48A_0^5r^j_k$, which implies that
$d(y^j_k,x^j_k)\sim r^j_k$ [for the existence of such a $y_k^j$, see Proposition \ref{prop-ozdec}(iv) with
$A:=(2A_0)^4$].

Indeed, if \eqref{eq-cl2} holds true, then
\begin{align*}
\lf|\lf<b^j_k,\vz\r>\r|&=\lf|\lf\langle f,\phi^j_k\vz\r\rangle\r|
\ls\frac{\mu(B(x^j_k,r^j_k))}{\mu(B(x^j_k,r^j_k))+V(x^j_k,x)}
\lf[\frac{r^j_k}{r^j_k+d(x^j_k,x)}\r]^\gz f^\star\lf(y^j_k\r)\\
&\ls 2^j\frac{\mu(B(x^j_k,r^j_k))}{\mu(B(x^j_k,r^j_k))+V(x^j_k,x)}
\lf[\frac{r^j_k}{r^j_k+d(x^j_k,x)}\r]^\gz.
\end{align*}
This, together with the arbitrariness of $\vz$, implies \eqref{eq-estbjk} directly.

We now prove \eqref{eq-cl2}. We first consider the size condition. Let $z\in X$. By the support of
$\phi^j_k$, we may suppose $d(z,x^j_k)\le 2A_0r^j_k$, which implies that $d(y^j_k,z)\ls r^j_k$. Moreover,
since $d(y^j_k,x^j_k)>16A_0^4r^j_k>2A_0d(z,x^j_k)$, it then follows that
$$
d\lf(y^j_k,z\r)\ge\frac 1{A_0}d\lf(y^j_k,x^j_k\r)-d\lf(z,x^j_k\r)>\frac{1}{2A_0}d\lf(y^j_k,x^j_k\r)
\sim r^j_k.
$$
Therefore, we have $r^j_k\sim r^j_k+d(y^j_k,z)$. On the other hand, similarly, since
$d(x,x^j_k)>2A_0d(z,x^j_k)$, we obtain $d(x,z)\sim d(x^j_k,x)\gtrsim r^j_k$. By these above inequalities,
$d(x^j_k,x)>r^j_k\gtrsim r$ and the size conditions of $\phi^j_k$ and $\vz$, we conclude that
\begin{align*}
\lf|\phi^j_k(z)\vz(z)\r|&\le\frac{1}{V_r(x)+V(x,z)}\lf[\frac{r}{r+d(x,z)}\r]^\gz
\ls\frac{1}{V_r(x^j_k)+V(x^j_k,x)}\lf[\frac{r}{r+d(x^j_k,x)}\r]^\gz\\
&\sim\frac{1}{\mu(B(x^j_k,r^j_k))+V(x^j_k,x)}\lf[\frac{r}{r^j_k+d(x^j_k,x)}\r]^\gz
\frac{\mu(B(x^j_k,r^j_k))}{\mu(B(y^j_k,r^j_k))+V(y^j_k,z)}\lf[\frac{r^j_k}{r^j_k+d(y^j_k,z)}\r]^\gz\\
&\ls\frac{\mu(B(x^j_k,r^j_k))}{\mu(B(x^j_k,r^j_k))+V(x^j_k,x)}\lf[\frac{r^j_k}{r^j_k+d(x^j_k,x)}\r]^\gz
\frac{1}{V_r(y^j_k)+V(y^j_k,z)}\lf[\frac{r}{r+d(y^j_k,z)}\r]^\gz,
\end{align*}
which implies the size condition of $\phi^j_k\vz$.

Next we consider the regularity condition of $\phi^j_k\vz$. Suppose $z,\ z'\in X$ with
$d(z,z')\le(2A_0)^{-1}[r+d(y^j_k,z)]$. Because of the size condition of $\phi^j_k\vz$, we may further assume
that $d(z,z')\le(2A_0)^{-10}[r+d(y^j_k,z)]$; the other cases are obviously true.
By the support of $\phi^j_k$, repeating the proof of Case 1), we find
that $\phi^j_k(z)\vz(z)-\phi^j_k(z')\vz(z')\neq 0$ further implies $d(z,x^j_k)<(2A_0)^2r^j_k$, and hence
$d(z,z')<(2A_0)^{-1}r^j_k\le(2A_0)^{-1}[r+d(r^j_k,z)]$ and $d(z,y^j_k)\ls r^j_k$. Moreover, since
$d(x^j_k,x)\ge(2A_0)^4r^j_k$, it then follows that $d(z,x^j_k)<(2A_0)^{-2}d(x,x^j_k)$ and hence
$d(z,x)\sim d(x^j_k,x)$. By these, both the size and the regularity conditions of $\phi^j_k$ and $\vz$, and
$r^j_k>r$, similarly to the proof of the size condition, we conclude that
\begin{align*}
&\lf|\phi^j_k(z)\vz(z)-\phi^j_k(z')\vz(z')\r|\\
&\quad\le\phi^j_k(z')|\vz(z)-\vz(z')|+|\vz(z)|\lf|\phi^j_k(z)-\phi^j_k(z')\r|\\
&\quad\ls\lf[\frac{d(z,z')}{r+d(x,z)}\r]^\bz\frac{1}{V_r(x)+V(x,z)}\lf[\frac{r}{r+d(x,z)}\r]^\gz
+\lf[\frac{d(z,z')}{r_k^j}\r]^\bz\frac{1}{V_r(x)+V(x,z)}\lf[\frac{r}{r+d(x,z)}\r]^\gz\\
&\quad\ls\frac{\mu(B(x^j_k,r^j_k))}{\mu(B(x^j_k,r^j_k))+V(x^j_k,x)}\lf[\frac{r^j_k}{r^j_k+d(x^j_k,x)}\r]^\gz
\lf[\frac{d(z,z')}{r+d(y^j_k,z)}\r]^\bz\frac{1}{V_r(y^j_k)+V(y^j_k,z)}\lf[\frac{r}{r+d(y^j_k,z)}\r]^\gz.
\end{align*}
Thus, we obtain the regularity condition of $\phi^j_k\vz$ and hence \eqref{eq-cl2} holds true. This finishes
the proof of Case 2).

Summarizing the above two cases, we obtain \eqref{eq-estbjk} and hence complete the proof of Proposition
\ref{prop-bjk}.
\end{proof}

\begin{proposition}\label{prop-bgj}
Let $p\in(\om,1]$ with $\omega$ and $\eta$, respectively, as in \eqref{eq-doub} and Definition \ref{def-iati}.
For any $j\in\zz$ such that $\Omega^j\subsetneqq X$ and $k\in I_j$, let $b^j_k$ be as in
\eqref{eq-defbjk} with $\bz,\ \gz\in(\omega[1/p-1],\eta)$. Then there exists a positive constant $C$ such
that, for any $j$ as above,
\begin{equation}\label{eq-bjp}
\int_X\sum_{k\in I_j}\lf[\lf(b^j_k\r)^*_0(x)\r]^p\,d\mu(x)\le C\lf\|f^*_0\mathbf{1}_{\Omega^j}\r\|_{L^p(X)}^p;
\end{equation}
moreover, there exists $b^j\in h^{*,p}(X)$ such that $b^j=\sum_{k\in I_j} b^j_k$ in $h^{*,p}(X)$ and, for any
$x\in X$,
\begin{equation}\label{eq-bj*}
\lf(b^j\r)^*_0(x)\le C2^j\sum_{k\in I_j}\frac{\mu(B(x^j_k,r^j_k))}{\mu(B(x^j_k,r^j_k))+V(x^j_k,x)}
\lf[\frac{r^j_k}{r^j_k+d(x^j_k,x)}\r]^{\bz\wedge\gz}+Cf^*_0(x)\mathbf{1}_{\Omega^j}(x);
\end{equation}
if $g^j:=f-b^j$ for any $j\in\zz$ such that $\Omega^j\subsetneqq X$, then, for any $x\in X$,
\begin{equation}\label{eq-gj*}
\lf(g^j\r)^*_0(x)\le C2^j\sum_{k\in I_j}\frac{\mu(B(x^j_k,r^j_k))}{\mu(B(x^j_k,r^j_k))+V(x^j_k,x)}
\lf[\frac{r^j_k}{r^j_k+d(x^j_k,x)}\r]^{\bz\wedge\gz}+Cf^*_0(x)\mathbf{1}_{(\Omega^j)^\complement}(x).
\end{equation}
\end{proposition}
\begin{proof}
Suppose $j\in\zz$ such that $\Omega^j\subsetneqq X$. We first prove \eqref{eq-bjp}. Indeed, by Propositions
\ref{prop-ozdec} and \ref{prop-bjk}, $\bz,\ \gz>\omega(1/p-1)$ and \cite[Lemma 4.8]{hhllyy18},
we conclude that
\begin{align*}
\int_X\sum_{k\in I_j}\lf[\lf(b^j_k\r)_0^*(x)\r]^p\,d\mu(x)&\ls 2^{jp}\int_X\sum_{k\in I_j}
\lf\{\frac{\mu(B(x^j_k,r^j_k))}{\mu(B(x^j_k,r^j_k))+V(x^j_k,x)}\lf[\frac{r^j_k}{r^j_k+d(x^j_k,x)}\r]
^{\bz\wedge\gz}\r\}^p\,d\mu(x)\\
&\quad+\int_{\bigcup_{k\in I_j} B(x^j_k,16A_0^4r^j_k)} [f^*_0(x)]^p\,d\mu(x)\\
&\ls 2^{jp}\mu(\Omega^j)+\lf\|f^*_0\mathbf 1_{\Omega^j}\r\|_{L^p(X)}^p
\ls\lf\|f^*_0\mathbf 1_{\Omega^j}\r\|_{L^p(X)}^p.
\end{align*}
This finishes the proof of \eqref{eq-bjp}.

Now we prove \eqref{eq-bj*}. Indeed, by \eqref{eq-bjp}, the dominated convergence theorem and the
completeness of $h^{*,p}(X)$, we find that there exists $b^j\in h^{*,p}(X)$ such that
$b^j=\sum_{k\in I_j}b^j_k$ in $h^{*,p}(X)$. By this, Proposition \ref{prop-bjk}, and (ii) and (v) of
Proposition \ref{prop-ozdec}, we conclude that, for any $x\in X$,
\begin{align*}
\lf(b^j\r)^*_0(x)&\le\sum_{k\in I_j}\lf(b^j_k\r)^*_0(x)\\
&\ls 2^j\sum_{k\in I_j}\frac{\mu(B(x^j_k,r^j_k))}{\mu(B(x^j_k,r^j_k))+V(x^j_k,x)}
\lf[\frac{r^j_k}{r^j_k+d(x^j_k,x)}\r]^{\bz\wedge\gz}+f^*_0(x)
\sum_{k\in I_j}\mathbf{1}_{B(x^j_k,16A_0^4r^j_k)}(x)\\
&\sim 2^j\sum_{k\in I_j}\frac{\mu(B(x^j_k,r^j_k))}{\mu(B(x^j_k,r^j_k))+V(x^j_k,x)}
\lf[\frac{r^j_k}{r^j_k+d(x^j_k,x)}\r]^{\bz\wedge\gz}+f^*_0(x)
\mathbf{1}_{\Omega^j}(x),
\end{align*}
which implies \eqref{eq-bj*}.

Finally, we prove \eqref{eq-gj*}. When $x\in(\Omega^j)^\complement$, then, by \eqref{eq-bj*}, we know that
$$
\lf(g^j\r)^*_0(x)\le f^*_0(x)+\lf(b^j\r)^*_0(x)\ls 2^j\sum_{k\in I_j}
\frac{\mu(B(x^j_k,r^j_k))}{\mu(B(x^j_k,r^j_k))+V(x^j_k,x)}
\lf[\frac{r^j_k}{r^j_k+d(x^j_k,x)}\r]^{\bz\wedge\gz}+f^*_0(x).
$$
Thus, \eqref{eq-gj*} holds true in this case.

Now suppose that $x\in\Omega^j$. Then, by Proposition \ref{prop-ozdec}(ii), we may find $k_0\in I_j$ such
that $x\in B(x^j_{k_0},r^j_{k_0})$ and then fix such a $k_0$. Let
$$
\CJ:=\lf\{n\in I_j:\ B\lf(x^j_n,16A_0^4r^j_n\r)\cap B\lf(x^j_{k_0},16A_0^4r^j_{k_0}\r)\neq\emptyset\r\}.
$$
By Proposition \ref{prop-ozdec}(v), we have $\#\CJ\ls 1$ and $(8A_0^2)^{-1}r^j_n\le r^j_{k_0}\le 8A_0^2r^j_n$
whenever $n\in\CJ$.

For any $x\in X$ and $\vz\in\go{\bz,\gz}$ with $\|\vz\|_{\CG(x,r,\bz,\gz)}\le 1$ for some $r\in(0,1]$, we
write
$$
\lf<g^j,\vz\r>=\langle f,\vz\rangle-\sum_{n\in I_j}\lf< b^j_n,\vz\r>
=\lf(\langle f,\vz\rangle-\sum_{n\in\CJ}\lf< b^j_n,\vz\r>\r)+\sum_{n\notin\CJ}\lf<b^j_n,\vz\r>
=:\RZ_1+\RZ_2.
$$
For the term $\RZ_2$, since $n\notin\CJ$, it then follows that $x\notin B(x^j_n,16A_0^4r^j_n)$. Therefore, by
\eqref{eq-estbjk}, we find that
\begin{align*}
|\RZ_2|&\ls 2^j\sum_{n\notin\CJ}\frac{\mu(B(x^j_n,r^j_n))}{\mu(B(x^j_n,r^j_n))+V(x^j_n,x)}
\lf[\frac{r^j_n}{r^j_n+d(x^j_n,x)}\r]^{\bz\wedge\gz}\\
&\ls 2^j\sum_{k\in I_j}\frac{\mu(B(x^j_k,r^j_k))}{\mu(B(x^j_k,r^j_k))+V(x^j_k,x)}
\lf[\frac{r^j_k}{r^j_k+d(x^j_k,x)}\r]^{\bz\wedge\gz},
\end{align*}
as desired.

It remains to estimate $\RZ_1$. To this end, we consider two cases.

{\it Case 1) $r^j_{k_0}\le(2A_0)^{-4}(8A_0)^{-2}$.} In this case, for any $n\in\CJ$, $r^j_n\le (2A_0)^{-4}$.
Thus, the definition of $b^j_k$ is the same as that in \cite[(4.3)]{hhllyy18}. Then, using a similar argument
to that used in the estimation of \cite[(4.6)]{hhllyy18}, we obtain, when $n\in\CJ$,
\begin{equation*}
|\RZ_1|\ls 2^j\sum_{k\in I_j}\frac{\mu(B(x^j_k,r^j_k))}{\mu(B(x^j_k,r^j_k))+V(x^j_k,x)}
\lf[\frac{r^j_k}{r^j_k+d(x^j_k,x)}\r]^{\bz\wedge\gz},
\end{equation*}
which is the desired estimate.

{\it Case 2) $r^j_{k_0}>(2A_0)^{-4}(8A_0)^{-2}$.} In this case, we let $\CJ_1:=\CJ\cap I_j^*$ and
$\CJ_2:=\CJ\setminus I_j^*$. We may further assume that $r\le r^j_{k_0}$ because, when $r\in(r^j_{k_0},1]$,
there exists $\wz r\in(0,r^j_{k_0}]$ such that $\|\vz\|_{\CG(x,\wz r,\bz,\gz)}\ls 1$, which can be reduced
to the case $r\le r^j_{k_0}$. Thus, we write
\begin{align*}
\RZ_1&=\langle f,\vz\rangle-\sum_{n\in\CJ_1}\lf< b^j_n,\vz\r>-\sum_{n\in\CJ_2}\lf<b^j_n,\vz\r>\\
&=\langle f,\vz\rangle-\sum_{n\in\CJ_1}\lf<f,\phi^j_n\vz\r>+\sum_{n\in\CJ_1}\lf<f,\wz\vz_n\r>
-\sum_{n\in\CJ_2}\lf<f,\phi^j_n\vz\r>
=\lf<f,\wz\vz\r>+\sum_{n\in\CJ_1}\lf<f,\wz\vz_n\r>,
\end{align*}
where $\wz\vz:=\vz(1-\sum_{n\in\CJ}\phi^j_n)$ and, for any $n\in\CJ_1$,
$$
\wz{\vz}_n:=\phi^j_n\lf[\int_X\phi^j_n(z)\,d\mu(z)\r]^{-1}\int_X\vz(z)\phi^j_n(z)\,d\mu(z).
$$
By the definition of $\wz\vz_n$, Proposition \ref{prop-chidec}(i) and \eqref{eq-doub}, we find that, for any
$n\in\CJ_1$,
$$
\int_X\phi^j_n(z)\,d\mu(z)\sim \mu\lf(B\lf(x_n^j,r^j_n\r)\r)\quad\textup{and}\quad
\int_X\lf|\vz(z)\phi^j_n(z)\r|\,d\mu(z)\ls 1,
$$
which implies $\|\wz\vz_n\|_{\CG(y^j_n,r^j_n,\bz,\gz)}\sim\|\wz\vz_n\|_{\CG(x^j_n,r^j_n,\bz,\gz)}\ls 1$.
Here and hereafter, for any $j\in\zz$ and $k\in I_j$, $y^j_k$ denotes the point such that
$y^j_k\notin\Omega^j$ satisfies $d(y^j_k,x^j_k)\le 48A_0^5r^j_k$ [see Proposition \ref{prop-ozdec}(iv)].
Moreover, by \cite[(4.7)]{hhllyy18}, we have
$\|\wz\vz\|_{\CG(y^j_{k_0},r^j_{k_0},\bz,\gz)}\ls 1$. Therefore, by this,
$r^j_n\le(2A_0)^{-4}$ whenever $n\in\CJ_1$, Proposition \ref{prop-bjk}, $\#\CJ_1\le\#\CJ\ls 1$ and the fact
$x\in B(x^j_{k_0},r^j_{k_0})$, we conclude that
$$
|\RZ_1|\ls f^*_0\lf(y^j_{k_0}\r)+\sum_{n\in\CJ_1} f^*_0\lf(y^j_n\r)\ls 2^j
\ls 2^j\sum_{k\in I_j}\frac{\mu(B(x^j_k,r^j_k))}{\mu(B(x^j_k,r^j_k))+V(x^j_k,x)}
\lf[\frac{r^j_k}{r^j_k+d(x^j_k,x)}\r]^{\bz\wedge\gz}.
$$
This, together with the estimate of $\RZ_2$ and the arbitrariness of $\vz$, further implies \eqref{eq-gj*}
in the case when $x\in\Omega^j$, which completes the proof of \eqref{eq-gj*} and hence of Proposition
\ref{prop-bgj}.
\end{proof}

By Proposition \ref{prop-bgj}, we have the following lemma.

\begin{lemma}\label{lem-dense}
Let $p\in(\om,1]$, $\bz,\ \gz\in(\omega(1/p-1),\eta)$ with $\omega$ and $\eta$, respectively, as in
\eqref{eq-doub} and Definition \ref{def-iati}, and $q\in[1,\fz)$. If regarding $h^{*,p}(X)$ as a
subspace of $(\go{\bz,\gz})'$, then $L^q(X)\cap h^{*,p}(X)$ is dense in $h^{*,p}(X)$.
\end{lemma}

\begin{proof}
Suppose $f\in h^{*,p}(X)$ and, for any $j\in\zz$, $\Omega^j$ is as in \eqref{4.3x}. Since $f^*_0\in L^p(X)$, it
then follows that there exists $j_0\in\zz$ such that, for any $j\in\{j_0,j_0+1,\ldots\}$,
$\Omega^j\subsetneqq X$. Therefore, for any $j\in\{j_0,j_0+1,\ldots\}$, let $b^j$ and $g^j$ be as in
Proposition \ref{prop-bgj}. By $p\in(\om,1]$, \eqref{eq-bjp} and the dominated convergence theorem, we
find that, for any $j\in\{j_0,j_0+1,\ldots\}$,
\begin{align*}
\lf\|\lf(b^j\r)^*_0\r\|_{L^p(X)}^p&=\int_X\lf[\lf(\sum_{k\in I_j}b^j_k\r)^*_0(x)\r]^p\,d\mu(x)
\le\int_X\lf[\sum_{k\in I_j}\lf(b^j_k\r)^*_0(x)\r]^p\,d\mu(x)\\
&\le\int_X\sum_{k\in I_j}\lf[\lf(b^j_k\r)^*_0(x)\r]^p\,d\mu(x)
\ls\lf\|f^*_0\mathbf 1_{\Omega^j}\r\|_{L^p(X)}^p\to 0
\end{align*}
as $j\to\fz$. This shows that $g^j\in h^{*,p}(X)$ and $f-g^j\to 0$ in $h^{*,p}(X)$ as $j\to\fz$.

We still need to prove that $g^j\in L^q(X)$ for any given $q\in[1,\fz]$. Indeed, by \eqref{eq-gj*}, the
Minkowski inequality, (ii) and (v) of Proposition \ref{prop-ozdec} and \cite[Lemma 4.6]{hhllyy18}, we conclude
that
\begin{align*}
\lf\|\lf(g^j\r)^*_0\r\|_{L^q(X)}&\ls 2^j\lf[\int_X\lf\{\sum_{k\in I_j}
\frac{\mu(B(x^j_k,r^j_k))}{\mu(B(x^j_k,r^j_k))+V(x^j_k,x)}
\lf[\frac{r^j_k}{r^j_k+d(x^j_k,x)}\r]^{\bz\wedge\gz}\r\}^q\,d\mu(x)\r]^{1/q}\\
&\qquad+\lf\{\int_{(\Omega^j)^\complement}\lf[f^*_0(x)\r]^q\,d\mu(x)\r\}^{1/q}\\
&\ls 2^j\mu\lf(\bigcup_{k\in I_j}B\lf(x^j_k,r^j_k\r)\r)+2^{j(q-p)/q}
\lf\{\int_{X}\lf[f^*_0(x)\r]^p\,d\mu(x)\r\}^{1/q}\\
&\ls 2^j\mu\lf(\Omega^j\r)+2^{j(q-p)/q}\|f\|_{h^{*,p}(X)}^{p/q}<\fz.
\end{align*}
Thus, $(g^j)^*_0\in L^q(X)$, which implies that $g^j\in L^q(X)$. Combining these above conclusions, we then
complete the proof of Lemma \ref{lem-dense}.
\end{proof}

In what follows, we suppose $f\in L^2(X)\cap h^{*,p}(X)$. For any $j\in\zz$ with $\Omega^j\subsetneqq X$ and
$k\in I_j$, let
\begin{equation}\label{eq-defmb}
m^j_k:=\frac{1}{\|\phi^j_k\|_{L^1(X)}}\int_X f(x)\phi^j_k(x)\,d\mu(x) \quad \textup{and} \quad
b^j_k:=\begin{cases}
\lf(f-m^j_k\r)\phi^j_k &\textup{if } k\in I_j^*,\\
f\phi^j_k & \textup{if } k\in I_j\setminus I_j^*.
\end{cases}
\end{equation}
Then we have the following properties.
\begin{lemma}\label{lem-funbgj}
For any $j\in\zz$ such that $\Omega^j\subsetneqq X$ and $k\in I_j$, let $m^j_k$ and $b^j_k$ be as in
\eqref{eq-defmb}. Then
\begin{enumerate}
\item there exists a positive constant $C$, independent of $j$ and $k\in I_j^*$, such that $|m^j_k|\le C2^j$;
\item $b^j_k$ induces the same distribution as the one defined in \eqref{eq-defbjk};
\item $\sum_{k\in I_j} b^j_k$ converges to some function $b^j$ in $L^2(X)$, which induces a distribution that
coincides with $b^j$ as in Proposition \ref{prop-bgj};
\item let $g^j:=f-b^j$. Then $g^j=f\mathbf{1}_{(\Omega^j)^\complement}+\sum_{k\in I_j^*}m^j_k\phi^j_k$.
Moreover, there exists a positive constant $C$, independent of $j$, such that, for almost every $x\in X$,
$|g^j(x)|\le C2^j$.
\end{enumerate}
\end{lemma}

\begin{proof}
Let all the notions be as in Lemma \ref{lem-funbgj}. We first show (i). Indeed, for any $k\in I_j^*$, by
Proposition \ref{prop-ozdec}(iii), we find that there exists $y^j_k\notin\Omega^j$ such that
$d(y^j_k,x^j_k)\ls r^j_k$. By this and Proposition \ref{prop-chidec}, we know that
$\|\phi^j_k\|_{\CG(y^j_k,r^j_k,\bz,\eta)}\sim\|\phi^j_k\|_{\CG(x^j_k,r^j_k,\bz,\eta)}\ls V_{\dz^k}(x^j_k)$.
From this, Proposition \ref{prop-chidec}(i) and $r^j_k\le 1$, we further deduce that
$$
\lf|m^j_k\r|\sim\frac 1{V_{\dz^k}(x^j_k)}\lf|\lf<f,\phi^j_k\r>\r|\ls f_0^\star\lf(y^j_k\r)\ls 2^j.
$$
This finishes the proof of (i).

Now we prove (ii). When $k\in I_j$, by the Fubini theorem, we conclude that, for any $\vz\in\go{\bz,\gz}$,
\begin{align*}
&\int_X b^j_k(x)\vz(x)\,d\mu(x)\\
&\quad=\int_X \lf[f(x)-\frac{1}{\|\phi^j_k\|_{L^1(X)}}\int_X
f(y)\phi^j_k(y)\,d\mu(y)\r]\phi^j_k(x)\vz(x)\,d\mu(x)\\
&\quad=\int_X f(y)\phi^j_k(y)\vz(y)\,d\mu(y)-\int_X f(y)\phi^j_k(y)\lf[\int_X\phi^j_k(z)\,d\mu(z)\r]^{-1}
\int_X\phi_k^j(x)\vz(x)\,d\mu(x)\,d\mu(y)\\
&\quad=\int_X f(y)\lf\{\phi^j_k(y)\lf[\int_X\phi^j_k(z)\,d\mu(z)\r]^{-1}\int_X[\vz(y)-\vz(x)]
\phi^j_k(x)\,d\mu(x)\r\}\,d\mu(y)\\
&\quad=\int_X b^j_k(x)\Phi^j_k(x)\,d\mu(x),
\end{align*}
which coincides with \eqref{eq-defbjk}. When $k\in I_j\setminus I_j^*$, the proof is even more direct and we
omit the details here. This finishes the proof of (ii).

Now we show (iii). For any $x\in X$, by Proposition \ref{prop-chidec}(ii), and (i) of this lemma, we find that
\begin{equation}\label{eq-sumbj}
\sum_{k\in I_j^*}\lf|f(x)-m^j_k\r|\phi^j_k(x)+\sum_{k\in I_j\setminus I_j^*}|f(x)|\phi^j_k(x)
\ls [2^j+|f(x)|]\sum_{k\in I_j}\phi^j_k(x)\sim [2^j+|f(x)|]\mathbf 1_{\Omega^j}(x).
\end{equation}
From this, it follows that
\begin{equation}\label{eq-bj}
b^j:=\sum_{k\in I_j} b^j_k,
\end{equation}
where the summation converges almost everywhere. Moreover, since $f\in L^2(X)$, it then follows that
$(2^j+|f|)\mathbf 1_{\Omega^j}\in L^2(X)$ with $\Omega^j$ as in \eqref{4.3x}. By \eqref{eq-sumbj} again,
together with the dominated convergence theorem and the completeness of $L^2(X)$, we further find that
\eqref{eq-bj} converges in $L^2(X)$, and hence in $(\go{\bz,\gz})'$. This finishes the proof of (iii).

Finally we prove (iv). For any $x\in X$, by Proposition \ref{prop-chidec}(ii), we have
$$
g^j(x)=f(x)-\sum_{k\in I_j} f(x)\phi^j_k(x)+\sum_{k\in I_j^*} m^j_k\phi^j_k(x)
=f(x)\mathbf 1_{(\Omega^j)^\complement}(x)+\sum_{k\in I_j^*} m^j_k\phi^j_k(x).
$$
Moreover, by the Lebesgue differential theorem, we find that $|f(x)|\ls f^\star_0(x)$ for almost every
$x\in X$. From this, (i), Proposition \ref{prop-chidec}(ii) and the above identity, we deduce that, for any
$x\in X$,
$$
\lf|g^j(x)\r|\ls f^\star_0(x)\mathbf 1_{(\Omega^j)^\complement}(x)+2^j\mathbf 1_{\Omega^j}(x)\ls 2^j,
$$
which completes the proof of (iv) and hence of Lemma \ref{lem-funbgj}.
\end{proof}

For any $j\in\zz$, $k\in I_j$ and $l\in I_{j+1}^*$, let
\begin{equation}\label{eq-defl}
L^{j+1}_{k,l}:=\frac{1}{\|\phi^{j+1}_l\|_{L^1(X)}}\int_X\lf[f(\xi)-m^{j+1}_l\r]\phi^j_k(\xi)\phi^{j+1}_l(\xi)
\,d\mu(\xi).
\end{equation}
Then we have the following lemma.

\begin{lemma}\label{lem-psi}
For any $j\in\zz$ such that $\Omega^j\subsetneqq X$, $k\in I_{j}$ and $l\in I_{j+1}^*$, let $L^{j+1}_{k,l}$
be as in \eqref{eq-defl}. Then
\begin{enumerate}
\item there exists a positive constant $C$, independent of $j$, $k$ and $l$, such that
$$
\sup_{x\in X}\lf| L^{j+1}_{k,l}\phi^{j+1}_l(x)\r|\le C2^j;
$$
\item $\sum_{k\in I_j}\sum_{l\in I_{j+1}^*} L^{j+1}_{k,l}\phi^{j+1}_l=0$ both in $(\go{\bz,\gz})'$ and
almost everywhere.
\end{enumerate}
\end{lemma}

This lemma can be proved by using a similar argument to that used in the proof of
\cite[Lemma 4.12]{hhllyy18}, with $I_{j+1}$ replaced by
$I_{j+1}^*$ and $f^*$ by $f^*_0$, because $r^{j+1}_k\le 1$ whenever $k\in I_{j+1}^*$; we omit the details.

When we establish an atomic decomposition of any element in $h^{*,p}(X)$, we need the next useful lemma
which can be found in the proof of \cite[Lemma 4.12(i)]{hhllyy18}.

\begin{lemma}\label{lem-pre}
Let $j\in\zz$ be such that $\Omega^j\subsetneqq X$, $k\in I_j$ and $l\in I_{j+1}$. Then
$B(x^{j+1}_l,2A_0r^{j+1}_l)\cap B(x^{j}_k,2A_0r^{j}_k)\neq\emptyset$ implies
$r^{j+1}_l\le 3A_0 r^j_k$ and $B(x^{j+1}_r,2A_0r^{j+1}_l)\subseteqq B(x^j_k,16A_0^4r^j_k)$.
\end{lemma}

Finally, we have the following atomic decomposition.
\begin{proposition}\label{prop-at}
Suppose that $p\in(\om,1]$ and $\bz,\ \gz\in(\omega(1/p-1),\eta)$ with $\omega$ and $\eta$, respectively, as
in \eqref{eq-doub} and Definition \ref{def-iati}. Then there exists a positive constant $C$ such that,
for any $f\in(\go{\bz,\gz})'$ belonging to $h^{*,p}(X)$, there exist local $(p,\fz)$-atoms
$\{a_j\}_{j=1}^\fz$ and $\{\lz_j\}_{j=1}^\fz\subset\cc$ such that $f=\sum_{j=1}\lz_j a_j$ in
$(\go{\bz,\gz})'$ and $\sum_{j=1}^\fz|\lz_j|^p\le C\|f^*_0\|_{L^p(X)}^p$.
\end{proposition}
\begin{proof}
By Lemma \ref{lem-dense}, we first suppose that $f\in L^2(X)\cap h^{*,p}(X)$. We may further assume that,
for any $x\in X$, $|f(x)|\ls f^*_0(x)$, where the implicit positive constant is independent of $f$ and $x$.
We use the same notation as in Lemmas \ref{lem-funbgj} and \ref{lem-psi}, and we consider two cases.

When $\mu(X)=\fz$, for any $j\in\zz$, we have $\Omega^j\subsetneqq X$ and let
$h^j:=g^{j+1}-g^j=b^j-b^{j+1}$. Thus, for any $m\in\nn$, $f-\sum_{|j|\le m} h^j=g^{-m}+b^{m+1}$. On one hand,
by Lemma \ref{lem-funbgj}(iii), we have $\|g^{-m}\|_{L^\fz(X)}\ls 2^{-j}$. On the other hand, by
\eqref{eq-bjp}, we have
$\|(b^j)^*_0\|_{L^p(X)}\ls \|f^*_0\mathbf{1}_{(\Omega^j)^\complement}\|_{L^p(X)}\to 0$ as $m\to\fz$.
Therefore, $f=\sum_{j=-\fz}^\fz h^j$ in $(\go{\bz,\gz})'$.

For any $j\in\zz$, using Proposition \ref{prop-chidec}(ii), we can write
\begin{align*}
h^j&=b^j-b^{j+1}=\sum_{k\in I_j} b^j_k-\sum_{l\in I_{j+1}} b^{j+1}_l
=\sum_{k\in I_j}\lf(b^j_k-\sum_{l\in I_{j+1}}b^{j+1}_l\phi^j_k\r)\\
&=\sum_{k\in I_{j,1}}\lf(b^j_k-\sum_{l\in I_{j+1}}b^{j+1}_l\phi^j_k\r)+
\sum_{k\in I_{j,2}}\lf(b^j_k-\sum_{l\in I_{j+1}}b^{j+1}_l\phi^j_k\r),
\end{align*}
where $I_{j,1}$ and $I_{j,2}$ are defined as in \eqref{eq-ij1}.
By Lemma \ref{lem-pre}, we find that $b^{j+1}_l\phi^j_k\neq 0$
implies that $r^{j+1}_l\le 3A_0r^j_k$, which, combined with $k\in I_{j,1}$, implies $l\in I_{j+1}^*$. By this,
(iii) and (iv) of Proposition \ref{lem-psi} and \eqref{eq-defmb}, we can write
\begin{align*}
h^j&=\sum_{k\in I_{j,1}}\lf(b^j_k-\sum_{l\in I_{j+1}^*}b^{j+1}_l\phi^j_k\r)+
\sum_{k\in I_{j,2}^*}\lf(b^j_k-\sum_{l\in I_{j+1}}b^{j+1}_l\r)\phi^j_k+
\sum_{k\in I_{j,2}\setminus I_{j,2}^*}\lf(b^j_k-\sum_{l\in I_{j+1}}b^{j+1}_l\r)\phi^j_k\\
&=\sum_{k\in I_{j,1}}\lf[b^j_k-\sum_{l\in I_{j+1}^*}\lf(b^{j+1}_l\phi^j_k-L^{j+1}_{k,l}\phi^{j+1}_l\r)\r]
+\sum_{k\in I_{j,2}}\sum_{l\in I_{j+1}^*} L^{j+1}_{k,l}\phi^{j+1}_l
+\sum_{k\in I_{j,2}^*}\lf(g^{j+1}-m^j_k\r)\phi^j_k\\
&\quad+\sum_{k\in I_{j,2}\setminus I_{j,2}^*}g^{j+1}\phi^j_k\\
&=\sum_{k\in I_{j,1}}\lf[b^j_k-\sum_{l\in I_{j+1}^*}\lf(b^{j+1}_l\phi^j_k-L^{j+1}_{k,l}\phi^{j+1}_l\r)\r]
+\sum_{k\in I_{j,2}^*}\lf[\lf(g^{j+1}-m^j_k\r)\phi^j_k+\sum_{l\in I_{j+1}^*} L^{j+1}_{k,l}\phi^{j+1}_l\r]\\
&\quad+\sum_{k\in I_{j,2}\setminus I_{j,2}^*}
\lf(g^{j+1}\phi^j_k+\sum_{l\in I_{j+1}^*} L^{j+1}_{k,l}\phi^{j+1}_l\r).
\end{align*}
Let
\begin{equation*}
h^j_k:=\begin{cases}
\displaystyle b^j_k-\sum_{l\in I_{j+1}^*}\lf(b^{j+1}_l\phi^j_k-L^{j+1}_{k,l}\phi^{j+1}_l\r)
 & \textup{when } k\in I_{j,1},\\
\displaystyle\lf(g^{j+1}-m^j_k\r)\phi^j_k+\sum_{l\in I_{j+1}^*} L^{j+1}_{k,l}\phi^{j+1}_l
 & \textup{when } k\in I_{j,2}^*,\\
\displaystyle g^{j+1}\phi^j_k+\sum_{l\in I_{j+1}^*} L^{j+1}_{k,l}\phi^{j+1}_l
& \textup{when } k\in I_{j,2}\setminus I_{j,2}^*.
\end{cases}
\end{equation*}
From Lemma \ref{lem-pre} and Proposition \ref{prop-chidec}(i), we deduce that, for any $j\in\zz$ and
$k\in I_j$, $\supp h^j_k\subset B(x^j_k,16A_0^4r^j_k)$. For any $j\in\zz$ and $k\in I_j$, let
$B^j_k:=B(x^j_k,48A_0^5r^j_k)$. We claim that, for any $j\in\zz$ and $k\in I_j$,
$h^j_k/[2^{jp}\mu(B^j_k)]^{1/p}$ is a harmless constant multiple of a local $(p,\fz)$-atom supported on
$B^j_k$.

Indeed, when $k\in I_{j,1}$, then $16A_0^4r^j_k\le(3A_0)^{-1}<1$. By Lemma \ref{lem-funbgj}(i),
(i) and (ii) of Proposition \ref{prop-chidec}, Lemma \ref{lem-psi}(i), and (ii) and (v) of Proposition
\ref{prop-ozdec}, we conclude that
\begin{align*}
\sum_{l\in I_{j+1}^*}\lf(\lf|b^{j+1}_l\r|\phi^j_k+\lf|L^{j+1}_{k,l}\phi^{j+1}_l\r|\r)
&\ls\sum_{l\in I_{j+1}^*}\lf|m^{j+1}_l\r|\phi^{j+1}_l\phi^j_k
+\sum_{l\in I_{j+1}^*}|f|\phi^{j+1}_l\phi^j_k+2^j
\sum_{l\in I_{j+1}^*}\mathbf 1_{B(x^{j+1}_l,2A_0r^{j+1}_l)}\\
&\ls 2^j\phi^j_k+|f|\phi^j_k+2^j\mathbf 1_{\Omega^j}\in L^1(X).
\end{align*}
From this and the dominated convergence theorem, we easily deduce that $\int_X h^j_k(x)\,d\mu(x)=0$. Finally,
we show the size condition of $h^j_k$. Indeed, by Lemma \ref{lem-pre}, we find that $\phi^j_k\phi^{j+1}_l\neq
0$ implies $r^{j+1}_l\le 3A_0r^j_k$, which, together with $k\in I_{j,1}$, further implies $l\in I_{j+1}^*$.
Thus, we have
$$
\phi^j_k\mathbf 1_{\Omega^j}=\phi^j_k\sum_{l\in I_{j+1}}\phi^{j+1}_l
=\phi^j_k\sum_{l\in I_{j+1}^*}\phi^{j+1}_l.
$$
By this and Proposition \ref{prop-chidec}(ii), we can write
\begin{align*}
h^j_k&=\lf(f-m^j_k\r)\phi^j_k-\sum_{l\in I_{j+1}^*}\lf(f-m^{j+1}_l\r)\phi^{j+1}_l\phi^j_k
+\sum_{l\in I_{j+1}^*}L^{j+1}_{k,l}\phi^{j+1}_l\\
&=f\phi^j_k\mathbf 1_{(\Omega^j)^\complement}-m^j_k\phi^j_k
+\phi^j_k\sum_{l\in I_{j+1}^*}m^{j+1}_l\phi^{j+1}_l+\sum_{l\in I_{j+1}^*}L^{j+1}_{k,l}\phi^{j+1}_l.
\end{align*}
Then, from Lemma \ref{lem-psi}, (i) and (iv) of Proposition \ref{prop-chidec}, and (ii) and (v) of
Proposition \ref{prop-ozdec}, we further deduce that $\|h^j_k\|_{L^\fz(X)}\ls 2^{j}$. Combining these, we
conclude that $h^j_k/[2^{jp}\mu(B^j_k)]^{1/p}$ is a harmless constant multiple of a local $(p,\fz)$-atom
supported on $B^j_k$.

Now we suppose $k\in I_{j,2}$. Obviously, we have $\supp h^j_k\subset B(x^j_k,16A_0^4r^j_k)\subset
B^j_k$. Moreover, noticing that $48A_0^5r^j_k>1$ when $k\in I_{j,2}$, to show the desired conclusion in this
case, it still needs to prove that $\|h^j_k\|_{L^\fz(X)}\ls 2^j$. Indeed, for any
$k\in I_{j,2}$, by Lemma \ref{lem-funbgj}(iii) and Proposition \ref{prop-chidec}(i), we have
$\|g^{j+1}\phi^j_k\|_{L^\fz(X)}\le\|g^{j+1}\|_{L^\fz(X)}\ls 2^j$. In addition, by Lemma \ref{lem-psi}(i),
Propositions \ref{prop-ozdec} and \ref{prop-chidec}(v), we further conclude that, for any $k\in I_{j,2}$,
$$
\lf\|\sum_{l\in I_{j+1}^*} L^{j+1}_{k,l}\phi^{j+1}_l\r\|_{L^\fz(X)}\ls 2^j\lf\|\sum_{l\in I_{j+1}^*}
\mathbf{1}_{B(x^{j+1}_l,2A_0r^{j+1}_l)}\r\|_{L^\fz(X)}\ls 2^j.
$$
Moreover, when $k\in I_{j,2}^*\subset I_j^*$, by Lemma \ref{lem-funbgj}(i), we obtain
$\|m^j_k\phi^j_k\|_{L^\fz(X)}\le|m^j_k|\ls 2^j$. Combining the above three inequalities, we finally find
that $\|h^j_k\|_{L^\fz(X)}\ls 2^j$ whenever $k\in I_{j,2}$, which completes the proof of the above claim.

To summarize, there exists a positive constant
$\wz C$ such that, for any $j\in\zz$ and $k\in I_j$, $\|h^j_k\|_{L^\fz(X)}\le\wz C2^j$. Then, for any
$j\in\zz$ and $k\in I_j$, if letting $\lz^j_k:=\wz C2^j[\mu(B^j_k)]^{1/p}$ and $a^j_k:=h^j_k/\lz^j_k$,
then $a^j_k$ is a local $(p,\fz)$-atom and $f=\sum_{j\in\zz}\sum_{k\in I_j}\lz^j_ka^j_k$ in $(\go{\bz,\gz})'$
and almost everywhere. Moreover, by (i) and (ii) of Proposition \ref{prop-ozdec}, we further conclude that
\begin{align*}
\sum_{j\in\zz}\sum_{k\in I_j}\lf|\lz^j_k\r|^p&\ls\sum_{j=-\fz}^\fz\sum_{k\in I_j} 2^{jp}\mu\lf(B^j_k\r)
\ls\sum_{j=-\fz}^\fz 2^{jp}\sum_{k\in I_j}\mu\lf(B\lf(x^j_k,r^j_k/\lf(5A_0^3\r)\r)\r)\\
&\ls\sum_{j=-\fz}^\fz 2^{jp}\mu\lf(\Omega^j\r)\ls\lf\|f^*_0\r\|_{L^p(X)}^p.
\end{align*}
Thus, Proposition \ref{prop-at} holds true when $f\in L^2(X)\cap h^{*,p}(X)$ in the case $\mu(X)=\fz$.

When $\mu(X)<\fz$, then, without loss of generality, we may assume $\mu(X)=1$. Moreover, by
\cite[Lemma 5.1]{ny97}, we conclude that $R_0:=\diam X<\fz$. Let $j_0\in\zz$ be the
minimal integer $j\in\zz$ such that $2^{j}\ge\|f^\star_0\|_{L^p(X)}$ with $f^\star_0$ as in \eqref{4.2x}.
Then, for any $j\in\{j_0,j_0+1,\ldots\}$,
we have $\Omega^j\subsetneqq X$; otherwise, we have $\Omega^j=X$ and hence
$$
\lf\|f^\star_0\r\|_{L^p(X)}^p>2^{jp}\mu(X)\ge 2^{j_0p}\ge\lf\|f^\star_0\r\|_{L^p(X)}^p,
$$
which is a contradiction. Thus, for any $j\ge j_0$, $\Omega^j\subsetneqq X$. For any
$j\in\{j_0,j_0+1,\ldots\}$, let $h^j:=g^{j+1}-g^j=b^j-b^{j+1}$. Then, for any $N\in\{j_0,j_0+1,\ldots\}$,
$$
f-\sum_{j=j_0}^N h^j=f-g^{N+1}-g^{j_0}=b^{N+1}+g^{j_0}\to g^{j_0}
$$
as $N\to\fz$, which converges both in $(\go{\bz,\gz})'$ and almost everywhere. Therefore,
$f=\sum_{j=j_0}^\fz h^j+g^{j_0}$. For the terms $\{h^j\}_{j=j_0}^\fz$, we may use a similar argument to that
used in the proof of the case $\mu(X)=\fz$ and we omit the details.
For the term $g^{j_0}$, by Lemma \ref{lem-funbgj}(iv), we have
$\|g^{j_0}\|_{L^\fz(X)}\le c2^{j_0}$, where $c$ is a positive constant independent of $f$ and $j_0$. Fix
$x_1\in X$ and let $\lz^{j_0}:=c2^{j_0}[\mu(B(x_1,R_0+1))]^{1/p}=c2^{j_0}$ and $a^{j_0}:=g^{j_0}/\lz^{j_0}$.
Then $a^{j_0}$ is a local $(p,\fz)$-atom supported on $B(x_1,R_0+1)$. Moreover, $|\lz^{j_0}|^p\sim
2^{j_0p}\sim\|f\|_{h^{*,p}(X)}$, which completes the proof when $f\in L^2(X)\cap h^{*,p}(X)$.

When $f\in h^{*,p}(X)$, we may use a density discussion to obtain a locally $(p,\fz)$-atomic decomposition
of $f$ (see \cite[pp.\ 301--302]{ms79b} for more details) and we omit the details. This finishes the proof
of Proposition \ref{prop-at}.
\end{proof}

Combining Propositions \ref{prop-atd} and \ref{prop-at}, we immediately obtain the following atomic
characterization of $h^{*,p}(X)$ for any given $p\in(\om,1]$. We omit the details.

\begin{theorem}\label{thm-at*}
Let $p\in(\om,1]$, $q\in(p,\fz]\cap[1,\fz]$ and $\bz,\ \gz\in(\omega(1/p-1),\eta)$,
with $\omega$ and $\eta$, respectively, as in \eqref{eq-doub} and Definition \ref{def-iati}. As subspaces of
$(\go{\bz,\gz})'$, $h^{*,p}(X)=h^{p,q}_\at(X)$ with equivalent (quasi-)norms.
\end{theorem}


\section{Littlewood--Paley function characterizations of atomic local Hardy  spaces} \label{s-lp}

In this section, we consider the Littlewood--Paley function characterizations of local Hardy spaces
$h^p(X)$. Let $\bz,\ \gz\in(0,\eta)$, $\thz\in(0,\fz)$ and $f\in(\go{\bz,\gz})'$. The
\emph{inhomogeneous Lusin area function $\CS_{\thz,0}(f)$ of $f$ with aperture $\thz$} is defined by setting,
for any $x\in X$,
\begin{equation}\label{eq-la}
\CS_{\thz,0}(f)(x):=\lf[\sum_{k=0}^\fz\int_{B(x,\thz\dz^k)}|Q_kf(y)|^2\,\frac{d\mu(y)}{V_{\thz\dz^k}(x)}\r]
^{1/2}.
\end{equation}
For $\thz=1$, we denote $\CS_{1,0}(f)$ simply by $\CS_0(f)$. The \emph{inhomogeneous Littlewood--Paley
$g$-function $g_0(f)$} is defined by setting, for any $x\in X$,
\begin{equation}\label{eq-g}
g_0(f)(x):=\lf[\sum_{k=0}^N\sum_{\az\in\CA_k}\sum_{m=1}^{N(k,\az)}m_{Q_\az^{k,m}}\lf(|Q_kf|^2\r)
\mathbf{1}_{Q_\az^{k,m}}(x)+\sum_{k=N+1}^\fz|Q_kf(x)|^2\r]^{1/2},
\end{equation}
where $N$ and $\{Q_\az^{k,m}:\ k\in\{0,\ldots,N\},\ \az\in\CA_k,\ m\in\{1,\ldots,N(k,\az)\}\}$ are the same
as in Theorem \ref{thm-idrf} and, for any $E\subset X$ with $\mu(E)\in(0,\fz)$ and non-negative function
$h$,
$$
m_E(h):=\frac{1}{\mu(E)}\int_E h(y)\,d\mu(y).
$$
For any $\lz\in(0,\fz)$, the \emph{inhomogeneous Littlewood--Paley $g_\lz^*$-functions $g_{\lz,0}^*(f)$} is
defined by setting, for any $x\in X$,
\begin{equation}\label{eq-gl}
g_{\lz,0}^*(f)(x):=\lf\{\sum_{k=0}^\fz\int_X|Q_kf(y)|^2\lf[\frac{\dz^k}{\dz^k+d(x,y)}\r]^\lz\,
\frac{d\mu(y)}{V_{\dz^k}(x)+V_{\dz^k}(y)}\r\}^{1/2}.
\end{equation}
For any $p\in(\om,1]$ and $\bz,\ \gz\in(\omega(1/p-1),\eta)$ with $\omega$ and $\eta$, respectively, as in
\eqref{eq-doub} and Definition \ref{def-iati}, the \emph{local Hardy space $h^p(X)$} is defined, via the
Lusin area function, by setting
$$
h^p(X):=\lf\{f\in\lf(\go{\bz,\gz}\r)':\ \|f\|_{h^p(X)}:=\|\CS_0(f)\|_{L^p(X)}<\fz\r\}.
$$

First we prove that $h^p(X)$ is independent of the choice of the $\exp$-IATI. If $\mathcal E:=
\{E_k\}_{k\in\zz}$ and $\CQ:=\{Q_k\}_{k\in\zz}$ are two $\exp$-IATIs, then we denote, respectively, by
$\CS_{0,\mathcal E}$ and $\CS_{0,\CQ}$ the inhomogeneous Lusin area functions via ${\mathcal E}$ and $\CQ$.
We have the following conclusions.
\begin{theorem}\label{thm-la}
Let ${\mathcal E}:=\{E_k\}_{k\in\zz}$ and $\CQ:=\{Q_k\}_{k\in\zz}$ be two {\rm $\exp$-IATIs}. Suppose that
$p\in(\om,1]$ and $\bz,\ \gz\in(\omega(1/p-1),\eta)$ with $\omega$ and $\eta$, respectively, as in
\eqref{eq-doub} and Definition \ref{def-iati}. Then there exists a positive constant $C$ such that,
for any $f\in(\go{\bz,\gz})'$,
\begin{equation*}
C^{-1}\|\CS_{0,\CQ}(f)\|_{L^p(X)}\le\|\CS_{0,\mathcal E}(f)\|_{L^p(X)}\le C\|\CS_{0,\CQ}(f)\|_{L^p(X)}.
\end{equation*}
\end{theorem}

\begin{proof}
By symmetry, to complete the proof of this theorem, it suffices to show that, for any $f\in(\go{\bz,\gz})'$,
$$
\|\CS_{0,\mathcal E}(f)\|_{L^p(X)}\ls\|\CS_{0,\CQ}(f)\|_{L^p(X)}.
$$
By Theorem \ref{thm-idrf}, we find that, for any $l\in\zz_+$ and $x\in X$,
\begin{align*}
E_lf(x)&=\langle f,E_l(x,\cdot)\rangle\\
&=\sum_{k=0}^N\sum_{\az\in\CA_k}\sum_{m=1}^{N(k,\az)}\int_{Q_\az^{k,m}}\wz{Q}_k^*E_l(x,y)\,d\mu(y)
Q_{\az,1}^{k,m}(f)\\
&\quad+\sum_{k=N+1}^\fz\sum_{\az\in\CA_k}\sum_{m=1}^{N(k,\az)}\mu\lf(Q_\az^{k,m}\r)
\wz{Q}_k^*E_l\lf(x,y_\az^{k,m}\r)Q_kf\lf(y_\az^{k,m}\r).
\end{align*}
Suppose $\xi\in X$, $l\in\zz_+$ and $x\in B(\xi,\dz^l)$. Notice that, for any $k\in\zz_+$, $\az\in\CA_k$,
$m\in\{1,\ldots,N(k,\az)\}$ and $y\in Q_\az^{k,m}$, we can write
$$
\wz Q_k^*E_l(x,y)=\int_X\wz Q_k(z,x)E_l(z,y)\,d\mu(y)=\wz Q_k^*(E_l(\cdot,y))(x).
$$
Thus, by Lemma \ref{lem-est1}, we find that, for any fixed $\eta'\in(0,\bz\wedge\gz)$,
\begin{align}\label{eq-5-1}
\lf|Q_k^*E_l(x,y)\r|&\ls\dz^{|k-l|\eta'}\frac{1}{V_{\dz^{k\wedge l}}(x)+V(x,y)}
\lf[\frac{\dz^{k\wedge l}}{\dz^{k\wedge l}+d(x,y)}\r]^\gz\\
&\sim\dz^{|k-l|\eta'}\frac{1}{V_{\dz^{k\wedge l}}(\xi)+V(\xi,y)}
\lf[\frac{\dz^{k\wedge l}}{\dz^{k\wedge l}+d(\xi,y)}\r]^\gz.\noz
\end{align}
When $k\in\{0,\ldots,N\}$, by the H\"{o}lder inequality, we conclude that, for any $\az\in\CA_k$ and
$m\in\{1,\ldots,N(k,\az)\}$,
\begin{align*}
\lf|Q_{\az,1}^{k,m}(f)\r|&\le\frac{1}{\mu(Q_\az^{k,m})}\int_{Q_\az^{k,m}}|Q_kf(u)|\,d\mu(u)
\le\lf[\frac{1}{\mu(Q_\az^{k,m})}\int_{Q_\az^{k,m}}|Q_kf(u)|^2\,d\mu(u)\r]^{1/2}\\
&\ls\lf[\frac{1}{V_{\dz^k}(u)}\int_{B(z,\dz^k)}|Q_kf(u)|^2\,d\mu(u)\r]^{1/2}\sim m_k(f)(z),
\end{align*}
which further implies that $|Q_{\az,1}^{k,m}(f)|\ls\inf_{z\in Q_{\az}^{k,m}} m_k(f)(z)$. Here and hereafter,
for any $k\in\zz_+$, $f\in(\go{\bz,\gz})'$ and $z\in X$, let
$$
m_k(f)(z):=\lf[\int_{B(z,\dz^k)}|Q_kf(u)|^2\,\frac{d\mu(u)}{V_{\dz^k}(u)}\r]^{1/2}.
$$
On another hand, for any $k\in\{N+1,N+2,\ldots\}$, $\az\in\CA_k$ and $m\in\{1,\ldots,N(k,\az)\}$, choose
$y_\az^{k,m}$ such that
$$
\lf|Q_kf\lf(y_\az^{k,m}\r)\r|\le\inf_{z\in Q_{\az}^{k,m}}|Q_kf(z)|+\frac{1}{\mu(Q_\az^{k,m})}
\int_{Q_\az^{k,m}} |Q_kf(y)|\,d\mu(y)\le\frac 2{\mu(Q_\az^{k,m})}\int_{Q_\az^{k,m}} |Q_kf(y)|\,d\mu(y).
$$
By this, the H\"{o}lder inequality and $\dz\le(2A_0)^{-10}$, we conclude that, for any
$z\in Q_\az^{k,m}$,
$$
\lf|Q_kf\lf(y_\az^{k,m}\r)\r|\ls\lf[\frac 1{\mu(Q_\az^{k,m})}\int_{Q_\az^{k,m}} |Q_kf(y)|^2\,d\mu(y)\r]^{1/2}
\ls m_k(f)(z).
$$
From these inequalities and Lemma \ref{lem-max}, we deduce that, for any fixed $r\in(0,p)$ and
$\eta'\in(0,\bz\wedge\gz)$, any $l\in\zz_+$, $\xi\in X$ and $x\in B(\xi,\dz^l)$,
\begin{align*}
|E_lf(x)|&\ls\sum_{k=0}^\fz\dz^{|k-l|[\eta'-\omega(1/r-1)]}\lf\{\CM\lf(\sum_{\az\in\CA_k}
\sum_{m=1}^{N(k,\az)}\inf_{z\in Q_\az^{k,m}}[m_k(f)(z)]^r\mathbf{1}_{Q_\az^{k,m}}\r)(\xi)\r\}^{1/r},
\end{align*}
where $\CM$ is as in \eqref{2.1x}.
Thus, choosing $\eta'$ and $r$ such that $\omega/(\omega+\eta')<r<p$, we know that, for any $\xi\in X$,
\begin{align*}
\lf[\CS_{0,\mathcal E}(f)(\xi)\r]^2&=\sum_{l=0}^\fz\int_{B(\xi,\dz^l)}|E_lf(x)|^2
\,\frac{dy}{V_{\dz^l}(\xi)}\\
&\ls\sum_{l=0}^\fz\sum_{k=0}^\fz\dz^{|k-l|[\eta'-\omega(1/r-1)]}
\lf\{\CM\lf(\sum_{\az\in\CA_k}\sum_{m=1}^{N(k,\az)}\inf_{z\in Q_\az^{k,m}}[m_k(f)(z)]^r
\mathbf{1}_{Q_\az^{k,m}}\r)(\xi)\r\}^{2/r}\\
&\ls\sum_{k=0}^\fz\lf\{\CM\lf(\sum_{\az\in\CA_k}\sum_{m=1}^{N(k,\az)}\inf_{z\in Q_\az^{k,m}}[m_k(f)(z)]^r
\mathbf{1}_{Q_\az^{k,m}}\r)(\xi)\r\}^{2/r}\\
&\ls\sum_{k=0}^\fz\lf\{\CM\lf([m_k(f)]^r\r)(\xi)\r\}^{2/r}.
\end{align*}
By this, $r<p$ and the Fefferman--Stein vector-valued maximal inequality (see \cite[(1.13)]{gly09}),
we further conclude that
\begin{equation*}
\|\CS_{0,\mathcal E}(f)\|_{L^p(X)}\ls\lf\|\lf(\sum_{k=0}^\fz\lf\{\CM\lf([m_k(f)]^r\r)\r\}
^{2/r}\r)^{r/2}\r\|_{L^{p/r}(X)}^{1/r}
\ls\lf\|\lf\{\sum_{k=0}^\fz[m_k(f)]^2\r\}^{1/2}\r\|_{L^p(X)}
\sim\|\CS_{0,\CQ}(f)\|_{L^p(X)}.
\end{equation*}
This finishes the proof of Theorem \ref{thm-la}.
\end{proof}

The next proposition shows that $h^{p,2}_\at(X)\subset h^p(X)$.

\begin{proposition}\label{prop-angle}
Let $p\in(\om,1]$, $\bz,\ \gz\in(\omega(1/p-1),\eta)$ with $\omega$ and $\eta$, respectively, as in
\eqref{eq-doub} and Definition \ref{def-iati} and $\{Q_k\}_{k=0}^\fz$ be an {\rm $\exp$-IATI}. Let
$\thz\in(0,\infty)$ and $\CS_{\thz,0}$ be as in \eqref{eq-la}. Then there exists a positive constant $C$,
independent of $\thz$, such that, for any $f\in(\go{\bz,\gz})'$ belonging to $h^{p,2}_\at(X)$,
\begin{equation}\label{eq-x2}
\|\CS_{\thz,0}(f)\|_{L^p(X)}\le C\max\lf\{\thz^{-\omega/2},\thz^{\omega/p}\r\}\|f\|_{h^{p,2}_\at(X)}.
\end{equation}
In particular, for any $q\in[1,\fz]\cap(p,\fz]$, $h^{p,q}_\at(X)=h^{p,2}_\at(X)\subset h^p(X)$.
\end{proposition}

\begin{proof}
Note that, for any $\thz\in(0,1]$ and $f\in(\go{\bz,\gz})'$, we have
$\CS_{\thz,0}(f)\ls\thz^{-\omega/2}\CS_0(f)$, which further implies that \eqref{eq-x5} holds true when
$\thz\in(0,1]$. Thus, to complete the proof of this proposition, we still need to prove that \eqref{eq-x5}
holds true when $\thz\in(1,\fz)$.
To this end , it suffices to prove that, for any $\thz\in(0,\fz)$ and any local $(p,2)$-atom $a$,
\begin{equation}\label{eq-angle}
\|\CS_{\thz,0}(a)\|_{L^p(X)}\ls\max\lf\{\thz^{-\omega/2},\thz^{\omega/p}\r\}.
\end{equation}
Indeed, suppose $f\in h^{p,2}_\at(X)$. Then there exist $\{\lz_j\}_{j=1}^\fz\subset\cc$ and local
$(p,2)$-atoms $\{a_j\}_{j=1}^\fz$ such that $\sum_{j=1}^\fz|\lz_j|^p<\fz$ and $f=\sum_{j=1}^\fz\lz_ja_j$ in
$(\go{\bz,\gz})'$. By the Minkowski integral inequality, for any $x\in X$, we have
$$
\CS_{\thz,0}(f)(x)=\CS_{\thz,0}\lf(\sum_{j=1}^\fz\lz_ja_j\r)(x)\le\sum_{j=1}^\fz|\lz_j|\CS_{\thz,0}(a_j)(x).
$$
From this, $p\in(\om,1]$ and \eqref{eq-angle}, we further deduce that
$$
\lf\|\CS_{\thz,0}(f)\r\|_{L^p(X)}^p\le\sum_{j=1}^\fz|\lz_j|^p\lf\|\CS_{\thz,0}(a_j)\r\|_{L^p(X)}^p
\ls\max\lf\{\thz^{-p\omega/2},\thz^{\omega}\r\}\sum_{j=1}^\fz|\lz_j|^p.
$$
Taking the infimum over all atomic decomposition of $f$, we obtain \eqref{eq-x2}. This finishes the proof
of Proposition \ref{prop-angle} under the assumption \eqref{eq-angle}.

Now we prove \eqref{eq-angle}. To this end, we first show that $\CS_{\thz,0}$ is bounded on $L^2(X)$ with its
operator norm independent of $\thz$. Indeed, for any $\vz\in L^2(X)$, by the Fubini theorem, we obtain
\begin{align*}
\lf\|\CS_{\thz,0}(\vz)\r\|_{L^2(X)}&=\int_X\sum_{k=0}^\fz\int_{B(x,\thz\dz^k)}|Q_k\vz(y)|^2
\,\frac{d\mu(y)}{V_{\thz\dz^k}(x)}\,d\mu(x)=\int_X\sum_{k=0}^\fz|Q_k\vz(y)|^2\,d\mu(y)\\
&=\lf\|g_0'(\vz)\r\|_{L^2(X)}^2,
\end{align*}
here and hereafter, for any $\vz\in L^2(X)$, $g_0'(\vz):=(\sum_{k=0}^\fz|Q_k\vz|^2)^{1/2}$.

Now we prove that, for any $\vz\in L^2(X)$,
\begin{equation}\label{eq-x6}
\lf\|g_0'(\vz)\r\|_{L^2(X)}\ls\|\vz\|_{L^2(X)}.
\end{equation}
Without loss
of generality, we may assume that $\{Q_k\}_{k=0}^\fz$ and $\vz$ are real. Otherwise we may split them
into real and imagine parts and estimate them separately. We write
$$
\lf\|g_0'(\vz)\r\|_{L^2(X)}^2=\sum_{k=0}^\fz\int_X \lf[Q_k\vz(x)\r]^2\,d\mu(x)
=\sum_{k=0}^\fz\lf<Q_k^*Q_k\vz,\vz\r>.
$$
Observe that $\{Q_k^*\}_{k=0}^\fz$ satisfies (i), (ii) and (iii) of Theorem \ref{thm-icrf} with $N=0$ and
$\bz=\eta=\gz$. Thus, by Lemma \ref{lem-est1}, we find that, for any fixed $\eta'\in(0,\eta)$, any
$k_1,\ k_2\in\zz_+$ and $x,\ y\in X$,
\begin{align*}
\lf|Q_{k_1}Q_{k_2}^*(x,y)\r|&=\lf|\int_X Q_{k_1}(x,z)Q_{k_2}(y,z)\,d\mu(y)\r|\\
&\ls\dz^{|k_1-k_2|\eta'}\frac{1}{V_{\dz^{k_1\wedge k_2}}(x)+V(x,y)}
\lf[\frac{\dz^{k_1\wedge k_2}}{\dz^{k_1\wedge k_2}+d(x,y)}\r]^\eta.
\end{align*}
From this and the size condition of $\{Q_k\}_{k=0}^\fz$, we deduce that, for any
$k_1,\ k_2\in\zz_+$,
$$
\lf\|\lf(Q_{k_1}^*Q_{k_1}\r)\lf(Q_{k_2}^*Q_{k_2}\r)\r\|_{L^2(X)\to L^2(X)}
\ls\lf\|Q_{k_1}Q_{k_2}^*\r\|_{L^2(X)\to L^2(X)}\dz^{|k_1-k_2|\eta'}.
$$
By this, the fact that $Q_{k}^*Q_{k}$ is self-adjoint and the Cotlar--Stein lemma
(see \cite[pp.\ 279--280]{stein93} or \cite[Lemma 4.4]{hlyy18}), we conclude that
$$
\lf\|g_0'(\vz)\r\|_{L^2(X)}^2=\lf<\sum_{k=0}^\fz Q_k^*Q_k\vz,\vz\r>\ls\|\vz\|_{L^2(X)}^2,
$$
which implies \eqref{eq-x6} and further the boundedness of $\CS_{\thz,0}$ on $L^2(X)$.

Now we continue to prove \eqref{eq-angle}. Suppose that $a$ is a local $(p,2)$-atom supported on the ball
$B:=B(x_0,r_0)$ for some $x_0\in X$ and $r_0\in(0,\fz)$. When $r_0\in(0,1]$, using the boundedness of
$\CS_{\thz,0}$ on $L^2(X)$ and a similar argument to that used in the estimation of \cite[(5.10)]{hhllyy18},
we conclude that \eqref{eq-angle} holds true.

It suffices to show \eqref{eq-angle} in the case $r_0\in(1,\fz)$. By the Fubini theorem and \eqref{eq-x6}, we
find that
$$
\|\CS_{\thz,0}(a)\|_{L^2(X)}= \lf\| \lf(\sum_{k=0}^\fz |Q_k a|^2\r)^{1/2}\r\|_{L^2(X)}
=\lf\|g_0'(a)\r\|_{L^2(X)}\ls \|a\|_{L^2(X)}\ls [\mu(B)]^{1/2-1/p},
$$
which, together with the H\"older inequality, further implies that
\begin{equation}\label{eq-Sa1}
\int_{B(x_0,4A_0^2\thz r_0)}[\CS_{\thz,0}(a)(x)]^p\,d\mu(x)\le\|\CS_{\thz,0}(a)\|_{L^2(X)}^p
\lf[\mu\lf(B\lf(x_0,4A_0^2\thz r_0\r)\r)\r]^{1-p/2}\ls\thz^{\omega(1-p/2)}.
\end{equation}

Now we consider the case $x\notin B(x_0,4A_0^2\thz r_0)$. Fix $k\in\zz_+$ and $y\in B(x,\thz\dz^k)$. When
$u\in B(x_0,r_0)$, we have $d(u,x_0)\le (2A_0)^{-1}d(x_0,y)$, which implies that $d(u,y)\sim d(x_0,y)
\sim d(x_0,x)$. By this, the size condition of $Q_k$ and the H\"older inequality, we have
\begin{align}\label{eq-qka}
|Q_ka(y)|&\le\int_B |Q_k(y,u)a(u)|\,d\mu(u)\ls\int_X\frac{1}{V_{\dz^k}(y)+V(y,u)}
\lf[\frac{\dz^k}{\dz^k+d(y,u)}\r]^\gz|a(u)|\,d\mu(u)\\
&\ls[\mu(B)]^{1-1/p}\frac{1}{V_{\dz^k}(x_0)+V(x_0,y)}\lf[\frac{\dz^k}{\dz^k+d(x_0,y)}\r]^\gz.\noz
\end{align}
We then separate $k\in\zz_+$ into two parts. If $\dz^k<(4A_0\thz)^{-1}d(x_0,x)$, then $d(x_0,y)\sim d(x_0,x)$.
In this case, we choose $\gz'\in(\omega(1/p-1),\gz)$. Therefore, by \eqref{eq-qka} and the fact
$r_0>1\ge\dz^k$, we conclude that
\begin{align*}
&\sum_{\dz^k<(4A_0\thz)^{-1}d(x_0,x)}\int_{B(x,\thz\dz^k)}|Q_ka(y)|^2\,\frac{d\mu(y)}{V_{\thz\dz^k}(x)}\\
&\quad\ls[\mu(B)]^{2(1-1/p)}\lf[\frac 1{V(x_0,x)}\r]^2
\sum_{\dz^k<(4A_0\thz)^{-1}d(x_0,x)}\lf[\frac{\dz^k}{d(x_0,x)}\r]^{2\gz}\\
&\quad\ls[\mu(B)]^{2(1-1/p)}\lf[\frac 1{V(x_0,x)}\r]^2\lf[\frac{r_0}{d(x_0,x)}\r]^{2\gz'}
\sum_{\dz^k<(4A_0^2\thz)^{-1}d(x_0,x)}\lf[\frac{\dz^k}{d(x_0,x)}\r]^{2(\gz-\gz')}\\
&\quad\ls[\mu(B)]^{2(1-1/p)}\lf[\frac 1{V(x_0,x)}\r]^2\lf[\frac{r_0}{d(x_0,x)}\r]^{2\gz'}.
\end{align*}
When $\dz^k\ge(4A_0^2\thz)^{-1}d(x_0,x)$, by \eqref{eq-doub}, we have
$V(x_0,x)\ls\mu(B(x_0,\thz\dz^k))\ls\thz^\omega V_{\dz^k}(x_0)$. Therefore,
\begin{align*}
\sum_{\dz^k\ge(4A_0\thz)^{-1}d(x_0,x)}\int_{B(x,\thz\dz^k)}|Q_ka(y)|^2\,\frac{d\mu(y)}{V_{\thz\dz^k}(x)}
&\ls[\mu(B)]^{2(1-1/p)}\sum_{\dz^k<(4A_0\thz)^{-1}d(x_0,x)}
\lf[\frac 1{V_{\dz^k}(x_0)}\r]^2\lf(\frac{r_0}{\dz^k}\r)^\gz\\
&\ls\thz^{2\omega}[\mu(B)]^{2(1-1/p)}\lf[\frac 1{V(x_0,x)}\r]^2
\sum_{\dz^k<(4A_0\thz)^{-1}d(x_0,x)}\lf(\frac{r_0}{\dz^k}\r)^\gz\\
&\ls\thz^{2(\omega+\gz)}[\mu(B)]^{2(1-1/p)}\lf[\frac 1{V(x_0,x)}\r]^2\lf[\frac{r_0}{d(x_0,x)}\r]^{2\gz}.
\end{align*}
Combining the above two estimates, we conclude that
\begin{equation}\label{eq-angle2}
\CS_{\thz,0}(a)(x)\ls\thz^{\omega+\eta}[\mu(B)]^{1-1/p}\lf[\frac{r_0}{d(x_0,x)}\r]^{\gz'}\frac{1}{V(x_0,x)}.
\end{equation}
Thus, using $\bz,\ \gz\in(\omega(1/p-1),\eta)$ and choosing $\gz'\in(\omega(1/p-1),\bz\wedge\gz)$, we obtain
\begin{align}\label{5.x2}
&\int_{[B(x_0,4A_0^2\thz r_0)]^\complement}[\CS_{\thz,0}(a)(x)]^p\,d\mu(x)\\
&\quad\ls\thz^{(\omega+\eta)p}[\mu(B)]^{p-1}\int_{[B(x_0,4A_0^2\thz r_0)]^\complement}
\lf[\frac{r_0}{d(x_0,x)}\r]^{p\gz'}\lf[\frac 1{V(x_0,x)}\r]^p\,d\mu(x)\noz\\
&\quad\ls\thz^{p\omega}[\mu(B)]^{p-1}\sum_{j=2}^\fz 2^{-jp\gz'}\int_{(2A_0)^j\thz r_0\le
d(x_0,x)<(2A_0)^{j+1}\thz r_0}
\lf[\frac 1{\mu(B(x_0,(2A_0)^j\thz r_0))}\r]^p\,d\mu(x)\noz\\
&\quad\ls\thz^\omega\sum_{j=2}^\fz 2^{-j[p\gz'-(1-p)\omega]}\ls \thz^\omega.\noz
\end{align}
This, combined with \eqref{eq-Sa1}, implies \eqref{eq-angle} in this case. Summarizing the above two cases,
we then complete the proof of \eqref{eq-angle} and hence of Proposition \ref{prop-angle}.
\end{proof}

Next we prove $h^p(X)\subset h^{p,2}_\at(X)$. To this end, we introduce the notion of \emph{local
molecules}.
\begin{definition}\label{def-mol}
Suppose that $p\in(0,1]$, $q\in(p,\fz]\cap[1,\fz]$ and $\vec{\ez}:=\{\ez_m\}_{m=1}^\fz\subset [0,\fz)$
satisfies
\begin{equation}\label{eq-epcon}
\sum_{m=1}^\fz m(\ez_m)^p<\fz.
\end{equation}
Let $B:=B(x_0,r_0)$ for some $x_0\in X$ and $r_0\in(0,\fz)$. A function $M\in L^q(X)$ is called a
\emph{local $(p,q,\vec{\ez})$-molecule centered at $B$} if $M$ satisfies the following conditions:
\begin{enumerate}
\item $\|M\mathbf{1}_B\|_{L^q(X)}\le [\mu(B)]^{1/q-1/p}$;
\item for any $m\in\nn$, $\|M\mathbf{1}_{B(x_0,\dz^{-m}r_0)\setminus B(x_0,\dz^{-m+1}r_0)}\|_{L^q(X)}\le
\ez_m[\mu(B(x_0,\dz^{-m}r_0))]^{1/q-1/p}$;
\item $\int_X M(x)\,d\mu(x)=0$ if $r_0\in(0,1]$.
\end{enumerate}
\end{definition}

Observe that, when $\mu(X)=\fz$, in Definition \ref{def-mol}, if we remove the restriction $r_0\in(0,1]$ in
(iii) of Definition \ref{def-mol}, then $M$ is called a \emph{$(p,q,\vec{\ez})$-molecule centered at $B$}
(see \cite[Definition 5.4]{hhllyy18}).

\begin{theorem}\label{thm-m=a}
Suppose that $p\in(\om,1]$, $q\in(p,\fz]\cap[1,\fz]$, $\bz,\ \gz\in(\omega(1/p-1),\eta)$ with $\omega$ and
$\eta$, respectively, as in \eqref{eq-doub} and Definition \ref{def-iati}, and
$\ez:=\{\ez_m\}_{m=1}^\fz$ satisfies \eqref{eq-epcon}. Then $f\in h^{p,q}_\at(X)$ if and only if there exist
local $(p,q,\vec{\ez})$-molecules $\{M_j\}_{j=1}^\fz$ and $\{\lz_j\}_{j=1}^\fz\subset\cc$ such that
$\sum_{j=1}^\fz|\lz_j|^p<\fz$ and
\begin{equation}\label{eq-mdec}
f=\sum_{j=1}^\fz\lz_jM_j \quad \textit{in}\quad \lf(\go{\bz,\gz}\r)'.
\end{equation}
Moreover, there exists a constant $C\in[1,\fz)$, independent of $f$, such that
$$
C^{-1}\|f\|_{h^{p,q}_\at(X)}\le\inf\lf\{\lf(\sum_{j=1}^\fz|\lz_j|^p\r)^{1/p}\r\}\le C\|f\|_{h^{p,q}_\at(X)},
$$
where the infimum is taken over all the molecular decompositions of $f$ as in \eqref{eq-mdec}.
\end{theorem}

\begin{proof}
It is obvious that a local $(p,q)$-atom is also a local $(p,q,\vec{\ez})$-molecule. By Theorem \ref{thm-at*},
to show this theorem, it still needs to prove that, for any local $(p,q,\vec{\ez})$-molecule $M$ centered at
a ball $B:=B(x_0,r_0)$ for some $x_0\in X$ and $r\in(0,\fz)$, there exist
$\{\lz_j\}_{j=0}^\fz\subset\cc$ and local $(p,q)$-atoms $\{a_j\}_{j=0}^\fz$ such that
$\sum_{j=0}^\fz|\lz_j|^p\ls 1$ and
\begin{equation}\label{eq-mat}
M=\sum_{j=0}^\fz\lz_ja_j
\end{equation}
in $(\go{\bz,\gz})'$. Indeed, if $r\in(0,1]$, then
$M$ is also a $(p,q,\vec{\ez})$-molecule. Using a similar argument to that used in the proof of
\cite[Theorem 3.3]{lcfy18}, we find that \eqref{eq-mat} holds true for some $\{\lz_j\}_{j=0}^\fz\subset\cc$
satisfying $\sum_{j=0}^{\fz}|\lz_j|^p\ls 1$ and $(p,q)$-atoms [hence also local $(p,q)$-atoms]
$\{a_j\}_{j=0}^\fz$. This implies that \eqref{eq-mat} holds true in this case.

We still need to prove \eqref{eq-mat} in the case $r\in(1,\fz)$. Indeed, let $a_0:=M\mathbf{1}_{B(x_0,r_0)}$
and $a_j:=M\mathbf{1}_{B(x_0,\dz^{-j}r_0)\setminus B(x_0,\dz^{-j+1}r_0)}/\ez_j$ for any $j\in\nn$. Then, by
the Lebesgue dominated theorem, we find that $M=M_0+\sum_{j=1}^\fz\ez_j M_j$ in $L^q(X)$ and hence, due to
the fact $q\in[1,\fz]$ and the H\"older inequality, also in $(\go{\bz,\gz})'$. Moreover, by (i) and (ii) of
Definition \ref{def-mol} with $r_0\in(1,\fz)$, we know that, for any $j\in\zz_+$, $M_j$ is a local
$(p,q)$-atom. Further, by \eqref{eq-epcon}, we have $1+\sum_{j=1}^\fz\ez_j^p\ls 1$. Thus, \eqref{eq-mat}
also holds true in this case. Combining the above two cases, we then completes the proof of Theorem
\ref{thm-m=a}.
\end{proof}

Using Theorem \ref{thm-m=a}, we obtain the following atomic decomposition of any element in $h^p(X)$.

\begin{proposition}\label{prop-ldec}
Suppose that $p\in (\om,1]$ and $\bz,\ \gz\in(\omega(1/p-1),\eta)$ with $\omega$ and $\eta$, respectively,
as in \eqref{eq-doub} and Definition \ref{def-iati}. Then there exists a positive constant
$C$ such that, for any $f\in(\go{\bz,\gz})'$ belonging to $h^p(X)$, $f\in h^{p,2}_\at(X)$ and
$\|f\|_{h^{p,2}_\at(X)}\le C\|f\|_{h^p(X)}$.
\end{proposition}
\begin{proof}
Suppose that $f\in(\go{\bz,\gz})'$ belongs to $h^p(X)$. Let $D_0:=P_0$ and, for any $k\in\nn$, $D_k:=Q_k$,
with $\{P_0,\{Q_k\}_{k=1}^\fz\}$ as in \cite{ah13}. Then, for any $k\in\zz_+$, $D_k^*=D_k$. Moreover, by the
orthonormality of $\{D_k\}_{k=0}^\fz$, we conclude that, for any $k,\ k'\in\zz_+$,
\begin{equation}\label{eq-orth}
D_kD_{k'}=\begin{cases}
D_k & \textup{if } k=k',\\
0 & \textup{if }k\neq k'.
\end{cases}
\end{equation}
Besides, by \cite{ah13}, we find that $\{D_k\}_{k=0}^\fz$ is an $\exp$-IATI. Then, using \eqref{eq-orth} and
a similar argument to that used in the proof of Theorem \ref{thm-icrf}, we conclude that
\begin{equation}\label{eq-icrf2}
f=\sum_{k=0}^\fz D_k^2f=\sum_{k=0}^\fz D_kD_k f
\end{equation}
in $(\go{\bz,\gz})'$. Moreover, by this and Theorem \ref{thm-la}, we may assume $\{Q_k\}_{k=0}^\fz
:=\{D_k\}_{k=0}^\fz$ appearing in \eqref{eq-la}

Let $\mathcal D:=\{Q_\az^l:\ l\in\nn,\
\az\in\CA_l\}$ with $\CA_l$ as in Lemma \ref{cube}. For any $k\in\zz$, we define
$\Omega_k:=\{x\in X:\ \CS_0(f)(x)>2^k\}$ and
$$
{\mathcal D}_k:=\lf\{Q\in{\mathcal D}:\ \mu(Q\cap\Omega_k)>\frac 12\mu(Q)\;\textup{and}\;
\mu(Q\cap\Omega_{k+1})\le\frac 12\mu(Q)\r\}.
$$
It is easy to see that, for any $Q\in{\mathcal D}$, there exists a unique $k\in\zz$ such that
$Q\in{\mathcal D}_k$. A dyadic cube $Q\in{\mathcal D}_k$ is called a \emph{maximal cube in ${\mathcal D}_k$}
if $Q'\in{\mathcal D}$ and $Q'\supsetneqq Q$, then
$Q'\notin{\mathcal D}_k$. Denote the set of all maximal cubes in ${\mathcal D}_k$ at level $j\in\nn$ by
$\{Q_{\tau,k}^j\}_{\tau\in I_{j,k}}$, where $I_{j,k}\subset \CA_j$ may be empty.
The center of $Q_{\tau,k}^j$ is denoted by $z_{\tau,k}^j$. Then
${\mathcal D}=\bigcup_{k\in\zz}\bigcup_{j\in\nn}\bigcup_{\tau\in I_{j,k}}
\{Q\in{\mathcal D}_k:\ Q\subset Q_{\tau,k}^j \}$.

From now on, we write $D_Q:=D_{l+1}$ for some $l\in\zz_+$ and $\az\in\CA_l$. Now we show that
\begin{equation}\label{eq-fmdec}
f=\sum_{k=-\fz}^\fz\sum_{j=1}^{N+1}\sum_{\tau\in I_{j,k}}\sum_{Q\in S_{\tau,k}^{j,1}}
\lz_{\tau,k}^{j,1} b_{\tau,k}^{j,1}
+\sum_{k=-\fz}^\fz\sum_{j=1}^{\fz}\sum_{\tau\in I_{j,k}}\sum_{Q\in S_{\tau,k}^{j,2}}
\lz_{\tau,k}^{j,2} b_{\tau,k}^{j,2}
\end{equation}
in $(\go{\bz,\gz})'$. Here and hereafter, for any $k\in\zz$, $j\in\nn$ and $\tau\in I_{j,k}$, let
$$
S_{\tau,k}^{j,1}:=\lf\{Q\in\CD_k:\ Q=Q_\az^l,\ l\in\{1,\ldots,N+1\},\ \az\in\CA_l,\ Q_\az^l\subset
Q_{\tau,k}^j\r\}
$$
(observe that, when $j\in\{N+2,N+3,\ldots\}$, $S_{\tau,k}^{j,1}=\emptyset$ because $ Q_\az^l\subset Q_{\tau,k}^j$
implies that $j\le l$), and
$$
S_{\tau,k}^{j,2}:=\lf\{Q\in\CD_k:\ Q=Q_\az^l,\ l\in\{N+2,N+3,\ldots\},\ \az\in\CA_l,\ Q_\az^l\subset
Q_{\tau,k}^j\r\};
$$
moreover, for any $k\in\zz$, $j\in\nn$, $\tau\in I_{j,k}$ and $i\in\{1,2\}$, let
$$
\lz_{\tau,k}^{j,i}:=\lf[\mu\lf(Q_{\tau,k}^j\r)\r]^{1/p-1/2}\lf[\sum_{Q\in S_{\tau,k}^{j,i}}
\int_Q|D_Qf(y)|^2\,d\mu(y)\r]^{1/2}
$$
and
\begin{equation}\label{eq-defbtji}
b_{\tau,k}^{j,i}(\cdot):=\frac 1{\lz_{\tau,k}^{j,i}}\sum_{Q\in S_{\tau,k}^{j,i}}
\int_{Q}D_Q(\cdot,y)D_Qf(y)\,d\mu(y).
\end{equation}

We first show that
\begin{equation}\label{eq-suml}
\sum_{k=-\fz}^\fz\sum_{j=1}^\fz\sum_{\tau\in I_{j,k}}\lf[\lf(\lz_{\tau,k}^{j,1}\r)^p
+\lf(\lz_{\tau,k}^{j,2}\r)^p\r]\ls\|\CS_0(f)\|_{L^p(X)}^p,
\end{equation}
which is equivalent to
\begin{equation*}
\sum_{k=-\fz}^\fz\sum_{j=1}^\fz\sum_{\tau\in I_{j,k}}\lf(\lz_{\tau,k}^{j,1}
+\lz_{\tau,k}^{j,2}\r)^p\ls\|\CS_0(f)\|_{L^p(X)}^p.
\end{equation*}
Indeed, for any $k\in\zz$, $j\in\nn$ and $\tau\in I_{j,k}$, we have
$$
\lz_{\tau,k}^{j,1}+\lz_{\tau,k}^{j,2}\sim\lf[\lf(\lz_{\tau,k}^{j,1}\r)^2+\lf(\lz_{\tau,k}^{j,2}\r)^2\r]^{1/2}
\sim\lf[\mu\lf(Q_{\tau,k}^j\r)\r]^{1/p-1/2}\lf[\sum_{Q\in\CD_k,\ Q\subset Q_{\tau,k}^j}
\int_Q|D_Qf(y)|^2\,d\mu(y)\r]^{1/2}=:\lz_{\tau,k}^j.
$$

Suppose $k\in\zz$, $j\in\nn$, $\tau\in I_{j,k}$ and $Q\in{\mathcal D}_k$ such that $Q\subset Q_{\tau,k}^j$.
We may further assume that $Q:=Q_\az^{l+1}$ for some $l\in\zz_+$ and $\az\in\CA_{l+1}$.
Since $\dz<(2A_0)^{-10}$, it then follows that $2A_0C^\natural \dz<1$ so that
$Q=Q_\az^{l+1}\subset B(y,\dz^l)$ for any $y\in Q$. By this and the fact
$\mu(Q\cap\Omega_{k+1})\le\frac 12\mu(Q)$, we conclude that
$$
\mu\lf(B\lf(y,\dz^l\r)\cap\lf[Q_{\tau,k}^j\setminus\Omega_{k+1}\r]\r)
\ge \mu\lf(B\lf(y,\dz^l\r)\cap[Q\setminus\Omega_{k+1}]\r) =
\mu(Q\setminus\Omega_{k+1}) \ge \frac 12\mu(Q)\sim  V_{\dz^l}(y).
$$
From this, we further deduce that
\begin{align*}
&\sum_{Q\in{\mathcal D}_k,\ Q\subset Q_{\tau,k}^j}\int_{Q}|D_Qf(y)|^2\,d\mu(y)\\
&\quad\ls \sum_{l=j-1}^\fz\;\sum_{\az\in\CA_{l+1},\;{\mathcal D}_k\ni Q_\az^{l+1}\subset Q_{\tau,k}^j}\;
\int_{Q_\az^{l+1}}\frac{\mu(B(y,\dz^l)\cap(Q_{\tau,k}^j\setminus\Omega_{k+1}))}{V_{\dz^l}(y)}
|E_lf(y)|^2\,d\mu(y)\\
&\quad\ls \sum_{l=j-1}^\fz\int_{Q_{\tau,k}^j}\frac{\mu(B(y,\dz^l)\cap(Q_{\tau,k}^j\setminus\Omega_{k+1}))}{V_{\dz^l}(y)}
|D_lf(y)|^2\,d\mu(y)\\
&\quad\sim\int_X\sum_{l=j-1}^\fz\int_{B(y,\dz^l)\cap(Q_{\tau,k}^j\setminus \Omega_{k+1})}
|D_lf(y)|^2\,\frac{d\mu(x)}{V_{\dz^l}(y)}\,d\mu(y)\\
&\quad\ls\int_{Q_{\tau,k}^j\setminus\Omega_{k+1}}[\CS_0(f)(x)]^2\,d\mu(x)\ls 2^{2k}\mu\lf(Q_{\tau,k}^{j}\r).
\end{align*}
Thus, by this and the fact $\mu(Q_{\tau,k}^j)<2 \mu(Q_{\tau,k}^j\cap\Omega_k)$, we conclude that
\begin{align*}
\sum_{k=-\fz}^{\fz}\sum_{j=1}^\fz\sum_{\tau\in I_{j,k}}\lf(\lz_{\tau,k}^j\r)^p
&\ls\sum_{k=-\fz}^\fz 2^{kp}\sum_{j=1}^\fz\sum_{\tau\in I_{j,k}}\mu\lf(Q_{\tau,k}^j\r)\\
&\ls\sum_{k=-\fz}^\fz 2^{kp}\sum_{j=1}^\fz\sum_{\tau\in I_{j,k}}\mu\lf(Q_{\tau,k}^j\cap\Omega_k\r)
\ls\sum_{k=-\fz}^\fz 2^{kp}\mu\lf(\Omega_k\r)\sim\|\CS_0(f)\|_{L^p(X)}^p.
\end{align*}
This shows \eqref{eq-suml}.

Next, we prove that, for any $k\in\zz$, $j\in\nn$, $\tau\in I_{j,k}$ and $i\in\{1,2\}$, $b_{\tau,k}^{j,i}$ is
a harmless constant multiple of a local $(p,2,\vec{\ez})$-molecule, where $\vec{\ez}$ satisfying
\eqref{eq-epcon} will be determined later. We first consider the case $i=1$. Let
$B_{\tau,k}^{j,1}:=B(z_{\tau,k}^j,\dz^{-1})$.
Since $\dz\in(0,1)$ is sufficiently small, it then follows that we only need to consider (i) and (ii) of
Definition \ref{def-mol}. Indeed, for (i), choose $h\in L^2(X)$. Then, by the H\"{o}lder inequality, we have
\begin{align}\label{eq-Y11}
\lf|\lf<b_{\tau,k}^{j,1},h\r>\r|&=\frac 1{\lz_{\tau,k}^{j,1}}\lf|\int_X\sum_{Q\in S_{\tau,k}^{j,1}}
\int_Q D_Q(x,y)D_Qf(y)\,d\mu(y)h(x)\,d\mu(x)\r|\\
&\le\frac 1{\lz_{\tau,k}^{j,1}}\lf[\sum_{Q\in S_{\tau,k}^{j,1}}\int_Q|D_Qf(y)|^2\,d\mu(y)\r]^{1/2}\noz\\
&\quad\times\lf[\sum_{Q\in S_{\tau,k}^{j,1}}
\int_Q\lf|\int_X D_Q(x,y)h(x)\,d\mu(x)\r|^2\,d\mu(y)\r]^{1/2}\noz\\
&\le\lf[\mu\lf(Q_{\tau,k}^j\r)\r]^{1/2-1/p}\lf[\sum_{l=1}^{N+1}
\sum_{\az\in\CA_l,\ \CD_k\ni Q_\az^l\subset Q_{\tau,k}^j}\int_{Q_\az^l}\lf|D_{l-1}^*h(y)\r|^2
\,d\mu(y)\r]^{1/2}\noz\\
&\le\lf[\mu\lf(Q_{\tau,k}^j\r)\r]^{1/2-1/p}\lf\|\wz{g_0}(h)\r\|_{L^2(X)},\noz
\end{align}
here and hereafter, for any $\vz\in L^2(X)$ and $x\in X$,
$$
\wz{g_0}(\vz)(x):=\lf[\sum_{k=0}^\fz\lf|D_k^*\vz(x)\r|^2\r]^{1/2}.
$$
Using a similar argument to that used in the estimation of \eqref{eq-x6}, we find that $\wz{g_0}$ is bounded
on $L^2(X)$. By this and \eqref{eq-Y11}, we conclude that
\begin{equation}\label{eq-b1}
\lf|\lf\langle b_{\tau,k}^{j,1},h\r\rangle\r|\ls\lf[\mu\lf(Q_{\tau,k}^j\r)\r]^{1/2-1/p}
\lf\|\wz g_0(h)\r\|_{L^2(X)}\ls\lf[\mu\lf(Q_{\tau,k}^j\r)\r]^{1/2-1/p}\|h\|_{L^2(X)}.
\end{equation}
Thus, from the arbitrariness of $h$, we deduce that
$$
\lf\|b_{\tau,k}^{j,1}\r\|_{L^2(X)}\ls\lf[\mu\lf(Q_{\tau,j}^{k}\r)\r]^{1/2-1/p}
\sim\lf[\mu\lf(B_{\tau,k}^{j,1}\r)\r]^{1/2-1/p}
$$
because of Lemma \ref{cube}(v) and the fact $\dz^k\sim 1$ for any $k\in\{1,\ldots,N\}$. Then
$b_{\tau,k}^{j,1}$ satisfies Definition \ref{def-mol}(i) as desired.

Now we show that $b_{\tau,k}^{j,1}$ satisfies Definition \ref{def-mol}(ii).
Choose $\gz'\in(\omega(1/p-1),\gz)$ and
let $\vec{\ez}:=\{\ez_m\}_{m=1}^\fz:=\{\dz^{[\gz'-\omega(1-1/p)]m}\}_{m=1}^\fz$. It then follows that
$\vec{\ez}$ satisfies \eqref{eq-epcon}. Moreover, using a similar argument to that used in the proof of
\cite[Lemma 5.8]{hhllyy18}, we find that, for any $m\in\nn$,
$$
\lf\|b_{\tau,k}^{j,1}\mathbf{1}_{(\dz^{-m}B_{\tau,k}^{j,1})\setminus(\dz^{-m+1}B_{\tau,k}^{j,1})}\r\|_{L^2(X)}
\ls\dz^{m[\gz'-\omega(1/p-1)]}\lf[\mu\lf(\dz^{-m}B_{\tau,k}^{j,1}\r)\r]^{1/2-1/p},
$$
as desired. Therefore, we conclude that $b_{\tau,k}^{j,1}$ is a harmless constant multiple of a local
$(p,2,\vec{\ez})$-molecule centered at $B_{\tau,k}^{j,1}$.

Now we consider the case $i=2$. In this case, for any $k\in\zz$, $j\in\nn$ and
$\tau\in I_{j,k}$, let $B_{\tau,k}^{j,2}:=B(z_{\tau,k}^j,\dz^{j-1})$.
When $Q\in S_{\tau,j}^{k,2}$, for any $y\in X$, we have $\int_X D_Q(x,y)\,d\mu(x)=0$. This proves that
$b_{\tau,j}^{k,2}$ satisfies Definition \ref{def-mol}(iii) because $\dz^{j-1}\in(0,1]$ when $j\in\nn$.
To show that $b_{\tau,j}^{k,2}$ satisfies (i) and (ii) of Definition \ref{def-mol}, we may use a similar
argument to that used in the estimation of $b_{\tau,k}^{j,1}$, with $b_{\tau,k}^{j,1}$ and
$\sum_{Q\in S_{\tau,k}^{j,1}}$ replaced, respectively, by $b_{\tau,k}^{j,2}$ and
$\sum_{Q\in S_{\tau,k}^{j,2}}$. Therefore, we conclude that $b_{\tau,k}^{j,2}$ is a harmless constant
multiple of a local $(p,2,\vec{\ez})$-molecule centered at $B_{\tau,k}^{j,2}$.

To summarize, we conclude that, for any $k\in\zz$, $j\in\nn$, $\tau\in I_{j,k}$ and $i\in\{1,2\}$,
$b_{\tau,k}^{j,i}$ is a harmless constant multiple of a local $(p,2,\vec{\ez})$-molecule for some
$\vec{\ez}$ satisfying \eqref{eq-epcon}. By this, \eqref{eq-suml} and Theorems \ref{thm-m=a} and
\ref{thm-at*}, we conclude that there exists $\wz f\in h^{*,p}(X)\subset(\go{\bz,\gz})'$ such that
\begin{equation}\label{eq-tfmdec}
\wz f=\sum_{k=-\fz}^\fz\sum_{j=1}^{N+1}\sum_{\tau\in I_{j,k}}\sum_{Q\in S_{\tau,k}^{j,1}}
\lz_{\tau,k}^{j,1} b_{\tau,k}^{j,1}
+\sum_{k=-\fz}^\fz\sum_{j=1}^{\fz}\sum_{\tau\in I_{j,k}}\sum_{Q\in S_{\tau,k}^{j,2}}
\lz_{\tau,k}^{j,2} b_{\tau,k}^{j,2}
\end{equation}
in $(\go{\bz,\gz})'$.

It remains to show $\wz f=f$ in $(\go{\bz,\gz})'$. To this end, we claim that, for any $k'\in\zz_+$ and
$x\in X$,
\begin{equation}\label{eq-tf=f}
D_{k'}D_{k'}\wz f(x)=D_{k'}D_{k'}f(x).
\end{equation}
Assuming this for the moment, by \eqref{eq-icrf2}, we have
$$
\wz f=\sum_{k'=0}^\fz D_{k'}D_{k'}\wz f=\sum_{k'=0}^\fz D_{k'}D_{k'}f=f.
$$
Thus, we find that \eqref{eq-fmdec} holds true. This, together with \eqref{eq-suml} and Theorem \ref{thm-m=a},
further implies that $f\in h^{p,2}_\at(X)$ and $\|f\|_{h^{p,2}_\at(X)}\ls\|f\|_{h^p(X)}$, which completes the
proof of Proposition \ref{prop-ldec}.

It remains to show \eqref{eq-tf=f}. Indeed, fix $k'\in\zz_+$ and $z\in X$. Noting that \eqref{eq-tfmdec}
converges in $(\go{\bz,\gz})'$, by the fact that $D_{k'}(x,\cdot)\in\CG(\eta,\eta)\subset\go{\bz,\gz}$, we
have
\begin{align*}
D_{k'}\wz f(z)&=\lf<\wz f,D_{k'}(z,\cdot)\r>\\
&=\sum_{k=-\fz}^\fz\sum_{j=1}^{N+1}\sum_{\tau\in I_{j,k}}\sum_{Q\in S_{\tau,k}^{j,1}}
\lf<\lz_{\tau,k}^{j,1}b_{\tau,k}^{j,1},D_{k'}(z,\cdot)\r>
+\sum_{k=-\fz}^\fz\sum_{j=1}^{\fz}\sum_{\tau\in I_{j,k}}\sum_{Q\in S_{\tau,k}^{j,2}}
\lf<\lz_{\tau,k}^{j,2}b_{\tau,k}^{j,2},D_{k'}(z,\cdot)\r>\\
&=:\sum_{k=-\fz}^\fz\sum_{j=1}^{N+1}\sum_{\tau\in I_{j,k}}\sum_{Q\in S_{\tau,k}^{j,1}}\RJ_{\tau,k}^{j,1}
+\sum_{k=-\fz}^\fz\sum_{j=1}^{\fz}\sum_{\tau\in I_{j,k}}\sum_{Q\in S_{\tau,k}^{j,2}}\RJ_{\tau,k}^{j,2}.
\end{align*}
We first deal with $\RJ_{\tau,k}^{j,1}$. Indeed, using a similar argument to that used in the estimation of
\eqref{eq-Y11}, we conclude that \eqref{eq-defbtji} converges in $L^2(X)$. By this, the definition of
$\{D_{k'}\}_{k'=0}^\fz$, the Fubini theorem and \eqref{eq-orth}, we have
\begin{align*}
\RJ_{\tau,k}^{j,1}&=\sum_{Q\in S_{\tau,k}^{j,1}}\lf<\int_{Q}D_Q(\cdot,y)D_Qf(y)\,d\mu(y),D_{k'}(z,\cdot)\r>\\
&=\sum_{l=1}^{N+1}\sum_{\{\az\in\CA_l:\ Q_\az^l\subset Q_{\tau,k}^{j}\}}
\lf<\int_{Q_\az^l}D_{l-1}(\cdot,y)D_{l-1}f(y)\,d\mu(y),D_{k'}(x,\cdot)\r>\\
&=\sum_{l=1}^{N+1}\sum_{\{\az\in\CA_l:\ Q_\az^l\subset Q_{\tau,k}^{j}\}}\int_{Q_\az^l}
\lf<D_{l-1}(\cdot,y),D_{k'}(x,\cdot)\r>D_{l-1}f(y)\,d\mu(y)\\
&=\begin{cases}
0 &\text{if } k'>N,\\
\displaystyle \sum_{\{\az\in\CA_{k'+1}:\ Q_\az^{k'+1}\subset Q_{\tau,{k}}^{j}\}}\int_{Q_\az^{k'+1}}
D_{k'}(z,y)D_{k'}f(y)\,d\mu(y) & \textup{if } k'\le N.
\end{cases}
\end{align*}
Similarly,
$$
\RJ_{\tau,k}^{j,2}=\begin{cases}
0 &\text{if } k'\le N,\\
\displaystyle \sum_{\{\az\in\CA_{k'+1}:\ Q_\az^{k'+1}\subset Q_{\tau,{k}}^{j}\}}\int_{Q_\az^{k'+1}}
D_{k'}(z,y)D_{k'}f(y)\,d\mu(y) & \textup{if } k'\ge N+1.
\end{cases}
$$
By the three identity above, we conclude that
\begin{equation}\label{eq-dk'}
D_{k'}\wz f(z)=\sum_{k=-\fz}^\fz\sum_{j=1}^\fz
\sum_{\tau\in I_{j,k}}\sum_{\{\az\in\CA_{k'+1}:\ Q_\az^{k'+1}\subset Q_{\tau,{k}}^{j}\}}\int_{Q_\az^{k'+1}}
D_{k'}(z,y)D_{k'}f(y)\,d\mu(y).
\end{equation}
Observe that, for any $\az\in\CA_{k'+1}$, there exists uniquely $k\in\zz$, $j\in\nn$ and $\tau\in I_{j,k}$
such that $Q_{\az}^{k'+1}\subset Q_{\tau,k}^j$. By this, Lemma \ref{cube}(iii) and the size condition of
$D_{k'}$, we find that
\begin{align*}
&\sum_{k=-\fz}^\fz\sum_{j=1}^\fz
\sum_{\tau\in I_{j,k}}\sum_{\{\az\in\CA_{k'+1}:\ Q_\az^{k'+1}\subset Q_{\tau,{k}}^{j}\}}\int_{Q_\az^{k'+1}}
|D_{k'}(z,y)D_{k'}f(y)|\,d\mu(y)\\
&\quad=\sum_{\az'\in\CA_{k'+1}}\int_{Q_\az^{k'+1}}|D_{k'}(z,y)D_{k'}f(y)|\,d\mu(y)
=\int_X |D_{k'}(z,y)D_{k'}f(y)|\,d\mu(y)<\fz,
\end{align*}
which, together with \eqref{eq-dk'}, further implies that
$$
D_{k'}\wz f(z)=\int_X D_{k'}(z,y)D_{k'}f(y)\,d\mu(y).
$$
From this, the size condition of $D_{k'}$, the Fubini theorem and \eqref{eq-orth}, it then
follows that, for any $x\in X$,
\begin{align*}
D_{k'}D_{k'}\wz f(x)&=\int_X D_{k'}(x,z)D_{k'}f(z)\,d\mu(z)=
\int_X\int_X D_{k'}(x,z)D_{k'}(z,y)D_{k'}f(y)\,d\mu(y)\,d\mu(z)\\
&=\int_X\int_X D_{k'}(x,z)D_{k'}(z,y)\,d\mu(z) D_{k'}f(y)\,d\mu(y)
=\int_X D_{k'}(x,y)D_{k'}f(y)\,d\mu(y)\\
&=D_{k'}D_{k'}f(x).
\end{align*}
This finishes the proof \eqref{eq-tf=f} and hence of Proposition \ref{prop-ldec}.
\end{proof}

Combining this with Theorem \ref{thm-m=a} and Proposition \ref{prop-angle}, we have the following
characterization of the atomic Hardy spaces $h^{p,q}_\at(X)$ and we omit the details.

\begin{theorem}\label{thm-h=}
Let $p\in(\om,1]$, $\bz,\ \gz\in(\omega(1/p-1),\eta)$ with $\omega$ and $\eta$, respectively, as in
\eqref{eq-doub} and Definition \ref{def-iati}, and $q\in(p,\fz]\cap[1,\fz]$. Then,
as subspaces of $(\go{\bz,\gz})'$, $h^{p,q}_\at(X)=h^p(X)$ with equivalent (quasi-)norms.
\end{theorem}

Finally, we characterize $h^p(X)$ via two other inhomogeneous Littlewood--Paley functions.

\begin{theorem}\label{thm-g}
Let $p\in(\om,1]$ and $\bz,\ \gz\in(\omega(1/p-1),\eta)$ with $\omega$ and $\eta$, respectively, as in
\eqref{eq-doub} and Definition \ref{def-iati}. Assume that $\thz\in(0,\infty)$ and $\lz\in(2\omega/p,\infty)$.
Then, for any $f\in(\go{\bz,\gz})'$, it holds true that
\begin{align}\label{eq-x5}
\|g_{\lz,0}^*(f)\|_{L^p(X)}\sim \|f\|_{h^p(X)} \sim \|g_0(f)\|_{L^p(X)},
\end{align}
provided that either one in \eqref{eq-x5} is finite, where the implicit positive equivalence constants in
\eqref{eq-x5} are independent of $f$.
\end{theorem}

\begin{proof}
To show the first equivalence of \eqref{eq-x5}, we may use a similar argument to that used in the proof of
\cite[Theorem 5.12]{hhllyy18}; we omit the details here.

It remains to prove the second equivalence of \eqref{eq-x5}. Indeed, we first suppose that $a$ is a local
$(p,2)$-atom supported on a ball $B:=B(x_0,r_0)$ for some $x_0\in X$ and $r_0\in(0,\fz)$. Observe that, by
the Fubini theorem and \eqref{eq-x6}, we have
\begin{align*}
\|g_0(a)\|_{L^2(X)}^2&=\int_X\lf[\sum_{k=0}^N\sum_{\az\in\CA_k}\sum_{m=1}^{N(k,\az)}
\int_{Q_\az^{k,m}}|Q_ka(y)|^2\,\frac{d\mu(y)}{\mu(Q_\az^{k,m})}
\mathbf{1}_{Q_\az^{k,m}}(x)+\sum_{k=N+1}^\fz|Q_ka(x)|^2\r]\,d\mu(x)\\
&=\int_X\sum_{k=0}^\fz |Q_ka(x)|^2\,d\mu(x)=\lf\|g_0'(f)\r\|_{L^2(X)}\ls\|a\|_{L^2(X)}^2
\ls[\mu(B)]^{1/2-1/p},
\end{align*}
where $g_0'$ is as in the proof of Proposition \ref{prop-angle}.
From this and the H\"{o}lder inequality, we deduce that
\begin{equation}\label{eq-g01}
\|g_0(a)\mathbf 1_{B(x_0,8A_0^3C^\natural r_0)}\|_{L^p(X)}^p\ls\|g_0(a)\|_{L^2(X)}^p[\mu(B)]^{1-p/2}\ls 1,
\end{equation}
here and hereafter, $C^\natural$ is as in Lemma \ref{cube}.

Now we estimate $\|g_0(a)\mathbf 1_{[B(x_0,8A_0^3C^\natural r_0)]^\complement}\|_{L^p(X)}$.
Suppose $k\in\zz$ and $z\in Q_\az^k$ for some $\az\in\CA_k$ with $\CA_k$ as in Lemma \ref{cube}.
By Lemma \ref{cube}, we find that
$Q_\az^k\subset B(z_\az^k,C^\natural\dz^k)\subset B(z,2A_0C^\natural\dz^k)$. Using this and a similar argument
to that used in the estimation of \eqref{eq-angle2} with $\thz:=2A_0C^\natural$
thereby, we conclude that, for any fixed $\gz'\in(0,\bz\wedge\gz)$ and any
$x\notin B(x_0,8A_0^3C^\natural\dz^k)$,
$$
g_0(a)(x)\ls[\mu(B)]^{1-1/p}\lf[\frac{r_0}{d(x_0,x)}\r]^{\gz'}\frac{1}{V(x_0,x)}.
$$
Therefore, we have
\begin{align*}
&\int_{[B(x_0,8A_0^3C^\natural r_0)]^\complement}[g_0(a)(x)]^p\,d\mu(x)\\
&\quad\ls[\mu(B)]^{p-1}\int_{[B(x_0,8A_0^3C^\natural r_0)]^\complement}
\lf[\frac{r_0}{d(x_0,x)}\r]^{p\gz'}\lf[\frac 1{V(x_0,x)}\r]^p\,d\mu(x)\noz\\
&\quad\ls[\mu(B)]^{p-1}\sum_{j=2}^\fz 2^{-jp\gz'}
\int_{(2A_0)^{j+1}C^\natural r_0\le d(x_0,x)<(2A_0)^{j+2}C^\natural r_0}
\lf[\frac 1{\mu(B(x_0,r_0))}\r]^p\,d\mu(x)\noz\\
&\quad\ls\sum_{j=2}^\fz 2^{-j[p\gz'-(1-p)\omega]}\ls 1.\noz
\end{align*}
From this and \eqref{eq-g01}, we deduce that
$\|g_0(a)\|_{L^p(X)}\ls 1$. This, combined with the Minkowski inequality and Theorem \ref{thm-h=}, further
implies that, for any $f\in(\go{\bz,\gz})'$,
$\|g_0(f)\|_{L^p(X)}\ls\|f\|_{h^{p,2}_\at(X)}\sim\|f\|_{h^p(X)}$.

We still need to prove the inverse inequality of this inequality. By Theorem \ref{thm-idrf}, we find that,
for any $f\in(\go{\bz,\gz})'$, $l\in\zz_+$, $x\in X$ and $z\in B(x,\dz^l)$.
\begin{align*}
Q_lf(z)&=\langle f,Q_l(z,\cdot)\rangle\\
&=\sum_{k=0}^N\sum_{\az\in\CA_k}\sum_{m=1}^{N(k,\az)}\int_{Q_\az^{k,m}}\wz{Q}_k^*Q_l(x,y)\,d\mu(y)
Q_{\az,1}^{k,m}(f)\\
&\quad+\sum_{k=N+1}^\fz\sum_{\az\in\CA_k}\sum_{m=1}^{N(k,\az)}\mu\lf(Q_\az^{k,m}\r)
\wz{Q}_k^*Q_l\lf(x,y_\az^{k,m}\r)Q_kf\lf(y_\az^{k,m}\r),
\end{align*}
where $y_\az^{k,m}$ is an arbitrary point in $Q_\az^{k,m}$.
When $k\in\{0,\ldots,N\}$, by the H\"{o}lder inequality, we conclude that, for any $\az\in\CA_k$ and
$m\in\{1,\ldots,N(k,\az)\}$,
\begin{align*}
\lf|Q_{\az,1}^{k,m}(f)\r|&\le\frac 1{\mu(Q_\az^{k,m})}\int_{Q_\az^{k,m}}|Q_kf(u)|\,d\mu(u)\\
&\le\lf[\frac 1{\mu(Q_\az^{k,m})}\int_{Q_\az^{k,m}}|Q_kf(u)|^2\,d\mu(u)\r]^{1/2}
=\lf[m_{Q_\az^{k,m}}\lf(|Q_kf|^2\r)\r]^{1/2}.
\end{align*}
Moreover, by \eqref{eq-5-1} with $E_l:=Q_l$ thereby, we find that, for any fixed $\bz'\in(0,\bz\wedge\gz)$,
any $k,\ l\in\zz_+$, $\az\in\CA_k$, $m\in\{1,\ldots,N(k,\az)\}$, $x\in X$, $y\in Q_\az^{k,m}$ and
$z\in B(x,\dz^l)$,
\begin{equation*}
\lf|\wz{Q}_k^*Q_l\lf(z,y\r)\r|\ls\dz^{|k-l|\bz'}\frac 1{V_{\dz^{k\wedge l}}(x)+V(x,z_\az^{k,m})}
\lf[\frac{\dz^{k\wedge l}}{\dz^{k\wedge l}+d(x,z_\az^{k,m})}\r]^\gz.
\end{equation*}
Thus, we conclude that, for any $l\in\zz_+$ and $z\in X$,
\begin{align*}
|Q_lf(z)|&\ls\sum_{k=0}^N\dz^{|k-l|\bz'}\sum_{\az\in\CA_k}\sum_{m=1}^{N(k,\az)}\mu\lf(Q_\az^{k,m}\r)
\frac 1{V_{\dz^{k\wedge l}}(x)+V(x,z_\az^{k,m})}
\lf[\frac{\dz^{k\wedge l}}{\dz^{k\wedge l}+d(x,z_\az^{k,m})}\r]^\gz
\lf[m_{Q_\az^{k,m}}\lf(|Q_kf|^2\r)\r]^{\frac 12}\\
&\quad+\sum_{k=N+1}^\fz\dz^{|k-l|\bz'}\sum_{\az\in\CA_k}\sum_{m=1}^{N(k,\az)}\mu\lf(Q_\az^{k,m}\r)
\frac 1{V_{\dz^{k\wedge l}}(x)+V(x,z_\az^{k,m})}
\lf[\frac{\dz^{k\wedge l}}{\dz^{k\wedge l}+d(x,z_\az^{k,m})}\r]^\gz\lf|Q_kf\lf(y_\az^{k,m}\r)\r|,
\end{align*}
Due to the arbitrariness of $y_\az^{k,m}\in Q_\az^{k,m}$, the above inequality also holds true with
$|Q_kf(y_\az^{k,m})|$ replaced by $\inf_{v\in Q_{\az}^{k,m}}|Q_kf(v)|$. By this and Lemma \ref{lem-max}, we
find that, for any fixed $r\in(\omega/(\omega+\bz'),p)$ and any $l\in\zz_+$, $x\in X$ and $z\in B(x,\dz^l)$,
\begin{align*}
|Q_lf(z)|&\ls\sum_{k=0}^N\dz^{|k-l|[\bz'-\omega(1/r-1)]}\lf\{\CM\lf(\sum_{\az\in\CA_k}\sum_{m=1}^{N(k,\az)}
\lf[m_{Q_\az^{k,m}}\lf(|Q_kf|^2\r)\r]^{r/2}\mathbf{1}_{Q_\az^{k,m}}\r)(x)\r\}^{1/r}\\
&\quad+\sum_{k=N+1}^\fz\dz^{|k-l|[\bz'-\omega(1/r-1)]}\lf\{\CM\lf(\sum_{\az\in\CA_k}\sum_{m=1}^{N(k,\az)}
\inf_{v\in Q_\az^{k,m}}|Q_kf(v)|^r\mathbf{1}_{Q_\az^{k,m}}\r)(x)\r\}^{1/r}.
\end{align*}
This, together with the H\"{o}lder inequality, further implies that for any $l\in\zz_+$ and $x\in X$,
\begin{align*}
&\int_{B(x,\dz^l)}|Q_lf(z)|^2\,\frac{d\mu(z)}{V_{\dz^l}(x)}\\
&\quad\ls\sum_{k=0}^N\dz^{|k-l|[\bz'-\omega(1/r-1)]}\lf\{\CM\lf(\sum_{\az\in\CA_k}\sum_{m=1}^{N(k,\az)}
\lf[m_{Q_\az^{k,m}}\lf(|Q_kf|^2\r)\r]^{r/2}\mathbf{1}_{Q_\az^{k,m}}\r)(x)\r\}^{2/r}\\
&\qquad+\sum_{k=N+1}^\fz\dz^{|k-l|[\bz'-\omega(1/r-1)]}\lf\{\CM\lf(\sum_{\az\in\CA_k}\sum_{m=1}^{N(k,\az)}
\inf_{v\in Q_\az^{k,m}}|Q_kf(v)|^r\mathbf{1}_{Q_\az^{k,m}}\r)(x)\r\}^{2/r}.
\end{align*}
By this and the Fefferman--Stein vector-valued maximal inequality (see \cite[(1.13)]{gly09}), we conclude
that, for any $f\in(\go{\bz,\gz})'$,
\begin{align*}
\|\CS_0(f)\|_{L^p(X)}&=\lf\|\lf[\sum_{l=0}^\fz\int_{B(\cdot,\dz^l)}|Q_lf(z)|^2
\,\frac{d\mu(z)}{V_{\dz^l}(\cdot)}\r]^{1/2}\r\|_{L^p(X)}\\
&\ls\lf\|\lf[\sum_{k=0}^N\lf\{\CM\lf(\sum_{\az\in\CA_k}\sum_{m=1}^{N(k,\az)}
\lf[m_{Q_\az^{k,m}}\lf(|Q_kf|^2\r)\r]^{r/2}\mathbf{1}_{Q_\az^{k,m}}\r)\r\}^{2/r}\r]^{r/2}\r\|_{L^{p/r}(X)}^{1/r}\\
&\quad+\lf\|\lf\{\sum_{k=N+1}^\fz\lf[\CM\lf(\sum_{\az\in\CA_k}\sum_{m=1}^{N(k,\az)}
\inf_{v\in Q_\az^{k,m}}|Q_kf(v)|^r\mathbf{1}_{Q_\az^{k,m}}\r)(x)\r]^{2/r}\r\}^{r/2}\r\|_{L^{p/r}(X)}^{1/r}\\
&\ls\lf\|\lf[\sum_{k=0}^N\sum_{\az\in\CA_k}\sum_{m=1}^{N(k,\az)}
m_{Q_\az^{k,m}}\lf(|Q_kf|^2\r)\mathbf{1}_{Q_\az^{k,m}}\r]^{1/2}\r\|_{L^{p}(X)}+
\lf\|\lf(\sum_{k=N+1}^\fz|Q_kf|^2\r)^{1/2}\r\|_{L^{p}(X)}\\
&\ls\|g_0(f)\|_{L^p(X)}.
\end{align*}
This finishes the proof of $\|f\|_{h^p(X)}\ls\|g_0(f)\|_{L^p(X)}$ and hence of Theorem \ref{thm-g}.
\end{proof}

\section{Relationship between $H^p_\cw(X)$ and $h^p(X)$}\label{s-hh}

By the arguments used in Sections \ref{s-max} through \ref{s-lp}, we conclude that, when $p\in(\om,1]$,
the local Hardy spaces $h^{+,p}(X)$, $h^p_\thz(X)$ with $\thz\in(0,\fz)$, $h^{*,p}(X)$, $h^{p,q}_\at(X)$,
$h^p(X)$ are the same space in the sense of equivalent (quasi-)norms. From now on, we simply use $h^p(X)$ to
denote either one of them if no confusion.

In this section, we discuss the relationship between $H^p_\cw(X)$, introduced by Coifman and Weiss in
\cite{cw77}, and $h^p(X)$. In Section \ref{ss-infz}, we find that $H^p_\cw(X)$ and $h^p(X)$ are essentially
different spaces when $\mu(X)=\fz$; while in Section \ref{ss-fz}, we show that $H^p_\cw(X)=h^p(X)$ when
$p\in(\om,1]$ and $\mu(X)<\fz$.

\subsection{The case $\mu(X)=\fz$}\label{ss-infz}

In this section, we always assume that $\mu(X)=\fz$. First we recall the definition of global
atomic Hardy spaces $H^{p,q}_\cw(X)$.

\begin{definition}\label{def-atom}
Let $p\in(\om,1]$, $\bz,\ \gz\in(\omega(1/p-1),\eta)$ with $\omega$ and $\eta$, respectively, as in
\eqref{eq-doub} and Definition \ref{def-iati} and $q\in(p,\fz]\cap[1,\fz]$. The space
$H^{p,q}_\cw(X)$ is defined to be the set of all $f\in(\go{\bz,\gz})'$ satisfying that there
exist $(p,q)$-atoms $\{a_j\}_{j=1}^\fz$ and $\{\lz_j\}_{j=1}^\fz\subset\cc$ such that
$\sum_{j=1}^\fz|\lz_j|^p<\fz$ and
$$
f=\sum_{j=1}^\fz\lz_ja_j \quad \textup{in} \quad (\go{\bz,\gz})'.
$$
For any $f\in H^{p,q}_\cw(X)$, define
$$
\|f\|_{H^{p,q}_\cw(X)}:=\inf\lf\{\lf(\sum_{j=1}^\fz|\lz_j|^p\r)^{1/p}\r\},
$$
where the infimum is taken over all atomic decompositions of $f$ as above.
\end{definition}

In \cite{cw77}, Coifman and Weiss proved that the space $H^{p,q}_\cw(X)$ is independent of $q$. Thus, we
can denote $H^{p,q}_\cw(X)$ simply by $H^p_\cw(X)$. We then have the following conclusion.

\begin{proposition}\label{prop-hh}
Let $p\in(\om,1]$ and $\bz,\ \gz\in(\omega(1/p-1),\eta)$ with $\omega$ and $\eta$, respectively, as in
\eqref{eq-doub} and Definition \ref{def-iati}. Then the following statements hold true:
\begin{enumerate}
\item $\int_X f(x)\,d\mu(x)=0$ if $f\in H^p_\cw(X)\cap L^2(X)$;
\item for any fixed $x_0\in X$, $\CG(x_0,1,\bz,\gz)\subset h^p(X)$;
\item $H^p_\cw(X)\subsetneqq h^p(X)$.
\end{enumerate}
\end{proposition}
\begin{proof}
We first prove (i). Indeed, suppose $f\in L^2(X)\cap H^p_\cw(X)$. Then, from \cite[Theorem 4.16]{hhllyy18} and
the H\"{o}lder inequality, we deduce that $f\in H^1_\cw(X)$. Thus, by the definition of $H^1_\cw(X)$, we know
that there exist $(p,\fz)$-atoms $\{a_j\}_{j=1}^\fz$ and $\{\lz_j\}_{j=1}^\fz\subset\cc$ satisfying
$\sum_{j=1}^\fz|\lz_j|\ls\|f\|_{H^1_\cw(X)}$ such that $f=\sum_{j=1}^\fz \lz_ja_j$ in $L^1(X)$. Therefore,
from the cancellation of $\{a_j\}_{j=1}^\fz$, we deduce that $\int_X f(x)\,d\mu(x)=0$ directly. This finishes
the proof of (i).

For (ii), suppose $f\in\CG(x_0,1,\bz,\gz)$ for some fixed $x_0\in X$. We show that $f$ is a harmless
constant multiple of a local $(p,\fz,\vec{\ez})$-molecule centered at $B(x_0,2)$ for some
$\vec{\ez}:=\{\ez_m\}_{m=1}^\fz\subset [0,\fz)$ satisfying \eqref{eq-epcon}. Without loss of generality, we
may further assume $\|f\|_{\CG(x_0,1,\bz,\gz)}\le 1$. Indeed, by the size condition of $f$, we have,
for any $x\in X$,
$$
|f(x)|\le\frac{1}{V_1(x_0)+V(x_0,x)}\lf[\frac{1}{1+d(x_0,x)}\r]^\gz\ls \lf[V_2(x_0)\r]^{-1/p},
$$
which further implies that $\|f\|_{L^\fz(X)}\ls[V_2(x_0)]^{-1/p}$, This shows that $f$, modulo a harmless
constant multiple, satisfies Definition \ref{def-mol}(i).

Now we show that $f$, modulo a harmless constant multiple, satisfies Definition \ref{def-mol}(ii). Fix
$m\in\nn$. Again, by the size condition of
$f$, we conclude that, for any $x\in B(x_0,2\dz^{-m})\setminus B(x_0,2\dz^{-m+1})$,
\begin{align*}
|f(x)|&\le\frac{1}{V_1(x_0)+V(x_0,x)}\lf[\frac{1}{1+d(x_0,x)}\r]^\gz\ls\frac{1}{V_{2\dz^{-m}}(x_0)}
\dz^{m\gz}\sim\lf[\frac{1}{V_{2\dz^{-m}}(x_0)}\r]^{1/p}\lf[\frac{V_{\dz^{-m}}(x_0)}{V_1(x_0)}\r]^{1/p-1}
\dz^{m\gz}\\
&\ls[\mu(B(x_0,2\dz^{-m}))]^{-1/p}\dz^{m[\gz-\omega(1/p-1)]}.
\end{align*}
Therefore, choosing $\ez_m:=\dz^{m[\gz-\omega(1/p-1)]}$ for any $m\in\nn$, we find that
$\vec{\ez}:=\{\ez_m\}_{m=1}^\fz$ satisfies \eqref{eq-epcon} and
$$
\lf\|f\mathbf{1}_{B(x_0,2\dz^{-m})\setminus B(x_0,2\dz^{-m+1})}\r\|_{L^\fz(X)}
\ls\ez_m[\mu(B(x_0,2\dz^{-m}))]^{-1/p}.
$$
Thus, we show that $f$, modulo a harmless constant multiple, satisfies Definition \ref{def-mol}(ii).
Consequently, $f$ is a harmless constant multiple of a local $(p,\fz,\vec{\ez})$-molecule, which, combined
with Theorem \ref{thm-m=a}, further implies $f\in h^p(X)$.

Finally, we prove (iii). By Theorem \ref{thm-at*} and the definition of $H^p_\cw(X)$, we may easily
obtain $H^p_\cw(X)\subset h^p(X)$. To show $H^p_\cw(X)\subsetneqq h^p(X)$, we choose a function
$f\in\CG(\bz,\gz)$ such that $\int_X f(x)\,d\mu(x)\neq 0$ (for the existence of such an $f$, see, for instance,
\cite[Corollary 4.2]{ah13}). On one hand, by (ii), we find that $f\in h^p(X)$. On the other hand, if
$f\in H^p_\cw(X)$, then, by the fact $\CG(\bz,\gz)\subset L^2(X)$ and (i), we find that
$\int_X f(x)\,d\mu(x)=0$, which contradicts to $\int_X f(x)\,d\mu(x)\neq 0$. Thus, $f\notin H^p_\cw(X)$.
This shows $H^p_\cw(X)\subsetneqq h^p(X)$, which then completes the proof of (iii) and hence of
Proposition \ref{prop-hh}.
\end{proof}

We also have the following relationship between $h^p(X)$ and $H^p_\cw(X)$.

\begin{proposition}\label{prop-pa}
Let $p\in(\om,1]$ and $\bz,\ \gz\in(\omega(1/p-1),\eta)$ with $\omega$ and $\eta$, respectively, as in
\eqref{eq-doub} and Definition \ref{def-iati}. Suppose that $P:\ X\times X\to\cc$ has the following property:
\begin{enumerate}
\item there exists a positive constant $C_0$ such that, for any $x\in X$, $P(x,\cdot)\in\go{\bz,\gz}$ and
$$
\|P(x,\cdot)\|_{\CG(x,1,\bz,\gz)}\le C_0;
$$
\item for any $x\in X$, $\int_X P(x,y)\,d\mu(y)=1=\int_X P(y,x)\,d\mu(y)$.
\end{enumerate}
Then there exists a positive constant $C$ such that, for any $f\in h^p(X)$, $f-Pf\in H^p_\cw(X)$ and
$\|f-Pf\|_{H^p_\cw(X)}\le C\|f\|_{h^p(X)}$.
\end{proposition}
\begin{proof}
By Theorem \ref{thm-h=} and \cite[Propositioin 5.5]{hhllyy18}, we only need to show that, for any local
$(p,2)$-atom $a$ supported on $B:=B(x_0,r_0)$ for some $x_0\in X$ and $r_0\in(0,\fz)$, $a-Pa$ is a
harmless constant multiple of a $(p,2,\vec{\ez})$-molecule centered at $4A_0^2 B$ for some
$\vec{\ez}:=\{\ez_m\}_{m=1}^\fz\subset[0,\fz)$ satisfying \eqref{eq-epcon}. From (ii), it easily follows that
$\int_X [a(x)-Pa(x)]\,d\mu(x)=0$. By (i), we conclude that
$$
|P(x,y)|\ls\frac{1}{V_1(x)+V(x,y)}\lf[\frac 1{1+d(x,y)}\r]^\gz.
$$
Therefore, by the boundedness of $\CM$ on $L^2(X)$ (see, for instance, \cite[(3.6)]{cw77}), we have
$$
\|a-Pa\|_{L^2(X)}\ls\|a\|_{L^2(X)}+\|\CM(a)\|_{L^2(X)}\sim\|a\|_{L^2(X)}\ls[\mu(B)]^{1/2-1/p}
\sim\lf[\mu\lf(4A_0^2B\r)\r]^{1/2-1/p},
$$
which shows that $a-Pa$, modulo a harmless constant multiple, satisfies Definition \ref{def-mol}(ii).

Now we show that $a-Pa$, modulo a harmless constant multiple, also satisfies Definition \ref{def-mol}(iii).
We consider two case. When $r_0\in(0,1]$, then we have
$\int_X a(x)\,d\mu(x)=0$. By this and (i), we conclude that, for any $x\notin B(x_0,4A_0^2r_0)$,
\begin{align*}
|Pa(x)|&=\lf|\int_X P(x,z)a(z)\,d\mu(u)\r|\le
\int_{B}|P(x,z)-P(x,x_0)||a(z)|\,d\mu(u)\\
&\ls\int_{B}\lf[\frac{d(x_0,z)}{1+d(x_0,x)}\r]^\eta\frac 1{V_{1}(x_0)+V(x_0,x)}
\lf[\frac{1}{1+d(x_0,y)}\r]^\gz|a(u)|\,d\mu(u)\\
&\ls[\mu(B)]^{1-1/p}\lf[\frac{r_0}{1+d(x_0,x)}\r]^\eta\frac 1{V_{1}(x_0)+V(x_0,x)}
\lf[\frac{1}{1+d(x_0,x)}\r]^\gz\\
&\ls\lf[\frac{1}{\mu(B)}\r]^{1/p-1}\lf[\frac{r_0}{d(x_0,x)}\r]^\eta\frac 1{V(x_0,x)}.
\end{align*}
When $r_0\in(1,\fz)$, then, using a similar argument to that used in the estimation of \eqref{eq-qka}, we
conclude that, for any $x\notin B(x_0,4A_0^2r_0)$,
\begin{equation*}
|Pa(x)|\ls[\mu(B)]^{1-1/p}\frac{1}{V_1(x_0)+V(x_0,x)}\lf[\frac{1}{1+d(x_0,x)}\r]^\gz\ls
\lf[\frac 1{\mu(B)}\r]^{1/p-1}\frac 1{V(x_0,x)}\lf[\frac{r_0}{d(x_0,x)}\r]^\gz.
\end{equation*}
Altogether, we have, for any $x\notin B(x_0,4A_0^2r_0)$,
\begin{equation}\label{eq-pa}
|Pa(x)|\ls\lf[\frac 1{\mu(B)}\r]^{1/p-1}\frac 1{V(x_0,x)}\lf[\frac{r_0}{d(x_0,x)}\r]^\gz.
\end{equation}

Fix $m\in\nn$. For any $x\in[B(x_0,4A_0^2\dz^{-m}r_0)\setminus B(x_0,4A_0^2\dz^{-m+1}r_0)]=:R_m$, by
\eqref{eq-pa}, we conclude that
\begin{align*}
|Pa(x)|&\ls\dz^{m\gz}\lf[\frac 1{\mu(B)}\r]^{1/p-1}\frac 1{\mu(B(x_0,\dz^{-m}r_0))}\\
&\sim\dz^{m\gz}\lf[\mu\lf(B\lf(x_0,4A_0^2\dz^{-m}r_0\r)\r)\r]^{-1/p}
\lf[\frac{\mu(B(x_0,\dz^{-m}r_0))}{\mu(B(x_0,r_0))}\r]^{1/p-1}\\
&\ls\dz^{m[\gz-\omega(1/p-1)]}\lf[\mu\lf(B\lf(x_0,4A_0^2\dz^{-m}r_0\r)\r)\r]^{-1/p}.
\end{align*}
This, together with the H\"{o}lder inequality, further implies that
$$
\lf\|(a-Pa)\mathbf{1}_{R_m}\r\|_{L^2(X)}=\lf\|Pa\mathbf{1}_{R_m}\r\|_{L^2(X)}\ls\dz^{m[\gz-\omega(1/p-1)]}
\lf[\mu\lf(B\lf(x_0,4A_0^2\dz^{-m}r_0\r)\r)\r]^{1/2-1/p}.
$$
Then, choosing $\ez_m:=\dz^{m[\gz-\omega(1/p-1)]}$, we find that $\vec{\ez}:=\{\ez_m\}_{m=1}^\fz$
satisfies \eqref{eq-epcon} and
$$
\lf\|(a-Pa)\mathbf{1}_{R_m}\r\|_{L^2(X)}\ls\ez_m\lf[\mu\lf(B\lf(x_0,4A_0^2\dz^{-m}r_0\r)\r)\r]^{1/2-1/p}.
$$
Therefore, $a-Pa$ is a harmless multiple constant of a $(p,2,\vec{\ez})$-molecule. This finishes the proof
of Proposition \ref{prop-pa}.
\end{proof}

\subsection{The case $\mu(X)<\fz$}\label{ss-fz}

In this section, we consider the case $\mu(X)<\fz$. To this end, we first recall the definition of
\emph{atoms} in the case $\mu(X)<\fz$.

\begin{definition}\label{def-atomb}
Let $p\in(0,1]$ and $q\in(p,\fz]\cap[1,\fz]$. A function $a\in L^q(X)$ is called a \emph{$(p,q)$-atom}
if $a=[\mu(X)]^{-1/p}$ or there exists a ball $B:=B(x_0,r_0)$ for some $x_0\in X$ and $r_0\in(0,\fz)$ such
that $a$ satisfies $\textup{(i)}_1$ and $\textup{(i)}_2$ of Definition \ref{def-lat} and
$\int_X a(x)\,d\mu(x)=0$.
\end{definition}

For any $p\in(\om,1]$, $\bz,\ \gz\in(\omega(1/p-1),\eta)$ and $q\in(p,\fz]\cap[1,\fz]$, the atomic Hardy
space $H^{p,q}_\cw(X)$ is defined in the same way as the case $\mu(X)=\fz$. The space $H^{p,q}_\cw(X)$
is also independent of the choice of $q$ (see \cite{cw77}). We have the following conclusion.

\begin{proposition}\label{prop-h=H}
Let $p\in(\om,1]$, $\bz,\ \gz\in(\omega(1/p-1),\eta)$ and $q\in(p,\fz]\cap[1,\fz]$. Then, as subspaces of
$(\go{\bz,\gz})'$, $h^{p,q}_\at(X)=H^{p}_\cw(X)$ with equivalent norms.
\end{proposition}

\begin{proof}
Fix $p\in(\om,1]$, $\bz,\ \gz\in(\omega(1/p-1),\eta)$ and $q\in(p,\fz]\cap[1,\fz]$. To prove this
proposition, it suffices to show the following two statements:
\begin{enumerate}
\item a $(p,q)$-atom is also a local $(p,q)$-atom;
\item for any local $(p,q)$-atom $a$, $\|a\|_{H^{p,q}_\cw(X)}\ls 1$.
\end{enumerate}

We first prove (i). Indeed, suppose that $a$ is a $(p,q)$-atom. If $a$ satisfies $\textup{(i)}_1$ and
$\textup{(i)}_2$ of Definition \ref{def-lat} and $\int_X a(x)\,d\mu(x)=0$ for some $B:=B(x_0,r_0)$ with
$x_0\in X$ and $r_0\in(0,\fz)$, then, by Definition \ref{def-lat}, we easily know that $a$ is also
a local $(p,q)$-atom. Now suppose $a=[\mu(X)]^{-1/p}$. Note that $\mu(X)<\fz$ implies that $\diam X<\fz$
(see, for instance, \cite[Lemma 5.1]{ny97}). Fix $x_0\in X$ and let $R:=\diam X+1$. Then
$R\in(1,\fz)$, $B:=B(x_0,R)=X$ and
$\|a\|_{L^q(B)}=[\mu(X)]^{1/q-1/p}=[\mu(B)]^{1/q-1/p}$. Thus, $a$ is a local $(p,q)$-atom. This implies (i).

Next we show (ii). Suppose that $a$ is a local $(p,q)$-atom supported on $B:=B(x_0,r_0)$ for some $x_0\in X$
and $r_0\in(0,\fz)$. If $r_0\in (0,1]$, then, by Definition \ref{def-atomb}, we find that $a$ is exactly a
$(p,q)$-atom, which further implies that $\|a\|_{H^{p,q}_\cw(X)}\ls 1$. If $r_0\in(1,\fz)$, then, since
$\diam X<\fz$, it follows that $X=B(x_0,\diam X+1)$ and
\begin{equation}\label{eq-xb}
\mu(B)\sim\mu(B(x_0,\diam X+1))\sim \mu(X).
\end{equation}
Let $a_1:=a-m_X(a)$ and $a_2:=m_X(a)$. Then $a=a_1+a_2$. By the H\"{o}lder inequality and \eqref{eq-xb},
we have
\begin{equation}\label{eq-mxa}
|m_X(a)|=\lf|\frac 1{\mu(X)}\int_X a(x)\,d\mu(x)\r|\le\lf[\frac 1{\mu(X)}\int_X |a(x)|^q\,d\mu(x)\r]^{1/q}
\le [\mu(X)]^{-1/p},
\end{equation}
which implies that $a_2$ is a harmless constant multiple of a $(p,q)$-atom. For $a_1$, it is obvious that
$\int_X a_1(x)\,d\mu(x)=0$. Moreover, by the size condition of $a$, \eqref{eq-xb} and \eqref{eq-mxa}, we
further obtain $\|a_1\|_{L^q(X)}\ls[\mu(X)]^{1/q-1/p}$. Therefore, $a_1$ is a harmless constant multiple of
a $(p,q)$-atom supported on $X=B(x_0,\diam X+1)$. Therefore, $\|a\|_{H^{p,q}_\at(X)}\ls 1$. This finishes the
proof of (ii) and hence of Proposition \ref{prop-h=H}.
\end{proof}

\section{Finite atomic characterizations and dual spaces\\ of local Hardy 
spaces}\label{s-dual}

In this section, we first obtain the finite atomic characterizations of $h^p(X)$ for any given $p\in(\om,1]$
with $\omega$ and $\eta$, respectively, as in \eqref{eq-doub} and Definition \ref{def-iati}. As an
application, we give the dual space of $h^p(X)$.

\subsection{Finite atomic characterizations of local Hardy spaces}\label{ss-fin}

In this section, we concern about the finite atomic characterizations of $h^p(X)$. For any $p\in(\om,1]$
and $q\in(p,\fz]\cap[1,\fz]$ with $\omega$ and $\eta$, respectively, as in \eqref{eq-doub} and Definition
\ref{def-iati}, the \emph{finite atomic local Hardy space $h^{p,q}_\fin(X)$} is defined to be the set of all
$f\in L^q(X)$ satisfying that there exist $N\in\nn$,
local $(p,q)$-atoms $\{a_j\}_{j=1}^N$ and $\{\lz_j\}_{j=1}^N\subset\cc$ such that
\begin{equation}\label{eq-fdec}
f=\sum_{j=1}^N\lz_j a_j.
\end{equation}
For any $f\in h^{p,q}_\fin(X)$, its (quasi-)norm in $h^{p,q}_\fin(X)$ is defined by setting
$$
\|f\|_{h^{p,q}_\fin(X)}:=\inf\lf\{\lf(\sum_{j=1}^N|\lz_j|^p\r)^{1/p}\r\},
$$
where the infimum is taken over all the decompositions of $f$ as in \eqref{eq-fdec}.

Obviously, $h^{p,q}_\fin(X)$ is a dense subspace of $h^{p,q}_\at(X)$ and, for any $f\in h^{p,q}_\fin(X)$,
$\|f\|_{h^{p,q}_\at(X)}\le\|f\|_{h^{p,q}_\fin(X)}$. We now prove that the inverse inequality also holds
true to some certain extend. To this end, we denote by $\UC(X)$ the space of all absolutely continuous
functions on $X$. Then we have the following conclusion.

\begin{proposition}\label{prop-fin}
Suppose $p\in(\om,1]$ with $\omega$ and $\eta$, respectively, as in \eqref{eq-doub} and Definition
\ref{def-iati}. Then the following statements hold true:
\begin{enumerate}
\item if $q\in(p,\fz)\cap[1,\fz)$, then  $\|\cdot\|_{h^{p,q}_\fin(X)}$ and $\|\cdot\|_{h^{p,q}_\at(X)}$
are equivalent (quasi)-norms on $h^{p,q}_\fin(X)$;
\item  $\|\cdot\|_{h^{p,\fz}_\fin(X)}$ and $\|\cdot\|_{h^{p,\fz}_\at(X)}$ are equivalent (quasi)-norms on
$h^{p,q}_\fin(X)\cap\UC(X)$;
\item $h^{p,\fz}_\fin(X)\cap\UC(X)$ is a dense subspace of $h^{p,\fz}_\at(X)$.
\end{enumerate}
\end{proposition}

\begin{proof}
Since this proof is quite similar to that of \cite[Theorem 7.1]{hhllyy18}, we only present the key points of
this proof here. Suppose $f\in h^{p,q}_\fin(X)\subset h^{p,q}_\at(X)$ with $\|f\|_{h^{*,p}(X)}=1$. By Theorem
\ref{thm-at*}, to show (i) and (ii), it suffices to show that $\|f\|_{h^{p,q}_\fin(X)}\ls 1$. By the proof of
Proposition \ref{prop-at}, we find that
$$
f=\sum_{j\in\zz}\sum_{k\in I_j}\lz^j_k a^j_k=\sum_{j\in\zz}\sum_{k\in I_j} h^j_k=\sum_{j\in\zz}h^j
$$
both in $(\go{\bz,\gz})'$ and almost everywhere. Here and hereafter, for any $j\in\zz$ and $k\in I_j$,
$h^j$, $h^j_k$, $\lz^j_k$ and $a^j_k$ are as in the proof of Proposition \ref{prop-at}. Since
$f\in h^{p,q}_\fin(X)$, it follows that, for fixed $x_1\in X$, there exists $R\in[1,\fz)$ such that
$\supp f\subset B(x_1,R)$. We claim that there exists a positive constant $\wz c$ such that, for any
$x\notin B(x_1,16A_0^4R)$,
\begin{equation}\label{eq-f0}
f^\star_0(x)\le\wz c[\mu(B(x_1,R))]^{-1/p},
\end{equation}
where $f_0^\star$ is as in \eqref{4.2x}.
Assuming this for the moment, we now prove (i) and (ii) of this proposition. Let $j'$ be the
maximal integer such that $2^j\le\wz c[\mu(B(x_1,R))]^{-1/p}$. Define
\begin{equation}\label{eq-defhl}
h:=\sum_{j\le j'}\sum_{k\in I_j}\lz^j_ka^j_k\quad\textup{and}\quad
\ell:=\sum_{j>j'}\sum_{k\in I_j}\lz^j_ka^j_k.
\end{equation}
Then $f=h+l$. In what follows, for the sake of convenience, we elide the fact whether or not $I_j$ is finite
and simply write the summation $\sum_{k\in I_j}$ in \eqref{eq-defhl} into $\sum_{k=1}^\fz$. If $j>j'$, then
$\Omega^j\subset B(x_1,16A_0^4R)$. By the definition of $a^j_k$, we have
$\supp\ell\subset B(x_1,16A_0^4R)$. From this, $h=f-\ell$ and $\supp f\subset B(x_1,R)$, we deduce that
$\supp h\subset B(x_1,16A_0^4R)$. Moreover, by Proposition \ref{prop-ozdec}(v) and the proof of Proposition
\ref{prop-at}, we have
$$
\|h\|_{L^\fz(X)}\ls\sum_{j\le j'}\lf\|\sum_{k=1}^\fz h^j_k\r\|_{L^\fz(X)}\ls\sum_{j\le j'} 2^j
\ls[\mu(B(x_1,R))]^{-1/p}.
$$
This, combined with $\eqref{eq-doub}$ and $R\ge 1$, shows that $h$ is a harmless multiple
constant of a local $(p,\fz)$-atom, and hence also a local $(p,q)$-atom with $q\in(p,\fz)\cap[1,\fz)$.

To deal with $\ell$, we split into two cases. When $q\in(p,\fz)\cap[1,\fz)$, for any $N=(N_1,N_2)\in\nn^2$,
let
$$
\ell_N=\sum_{j=j'+1}^{N_1}\sum_{k=1}^{N_2}\lz^j_ka^j_k=\sum_{j=j'+1}^{N_1}\sum_{k=1}^{N_2}h^j_k.
$$
Then $\ell_N$ is a finite combination of local $(p,q)$-atoms and
$\sum_{j=j'+1}^{N_1}\sum_{k=1}^{N_2}|\lz^j_k|^p\ls 1$. Moreover, $\supp\ell_N\subset B(x_1,16A_0^4R)$,
which, together with the support of $\ell$ implies that $\supp(\ell-\ell_N)\subset B(x_1,16A_0^4R)$.
Using a similar argument to that used in the proof of \cite[Theorem 7.1(i)]{hhllyy18} with $f^\star$
thereby replaced by $f^\star_0$ in \eqref{4.2x}, we conclude
that, for any $\ez\in(0,\fz)$, there exists $N\in\nn^2$ such that $\|\ell-\ell_N\|_{L^q(X)}<\ez$.
Combining the above conclusions, we have $\|f\|_{h^{p,q}_\fin(X)}\ls 1$.
This proves (i) under the assumption \eqref{eq-f0}.

Now we consider the case $q=\fz$. In this case, $f\in\UC(X)\cap h^{p,\fz}_\fin(X)$. Notice that
$\|f^\star_0\|_{L^\fz(X)}\ls\|\CM(f)\|_{L^\fz(X)}\le c_0\|f\|_{L^\fz(X)}$, where $c_0$ is a positive constant
independent of $f$. Then let $j''>j'$ be the largest integer such that $2^j\le c_0\|f\|_{L^\fz(X)}$
with $j'$ as in \eqref{eq-defhl}.
Then $\ell=\sum_{j'<j\le j''}\sum_{k=1}^\fz h^j_k$ and, as in the proof of (i), $h$ is a harmless constant
multiple of a local $(p,\fz)$-atom. Since $f\in\UC(X)$, it follows that, for any
$\ez\in(0,\fz)$, there exists $\sigma\in(0,(2A_0)^{-2})$ such that $|f(x)-f(y)|\le\ez$
whenever $d(x,y)\le\sigma$. Split $\ell=\ell^\sigma_1+\ell^\sigma_2$ with
$$
\ell^\sigma_1:=\sum_{(j,k)\in G_1} h^j_k=\sum_{(j,k)\in G_1}\lz^j_k a^j_k\qquad \textup{and}\qquad
\ell^\sigma_2:=\sum_{(j,k)\in G_2} h^j_k,
$$
where
$$
G_1:=\{(j,k):\ 12A_0^3r^j_k\ge\sigma,\ j'<j\le j''\}\quad\textup{and}\quad
G_2:=\{(j,k):\ 12A_0^3r^j_k<\sigma,\ j'<j\le j''\}.
$$

For the term $\ell^\sigma_1$, notice that, for any $j'<j\le j''$,
$\Omega^j$ is bounded. Thus, by Proposition \ref{prop-ozdec}(vi), we find that $G_1$ is a finite set, which
further implies that $\ell^\sigma_1$ is a finite linear combination of local $(p,\fz)$-atoms with
$\sum_{(j,k)\in G_1}|\lz^j_k|^p\ls 1$.

For the term $\ell^\sigma_2$, notice that, for any $(j,k)\in G_2$, $r^k_j<(12A_0)^{-3}\sigma<(48A_0^5)^{-1}$.
This shows that $r^j_k\in I_{j,1}$ and the definition of $h^j_k$ coincides with that in
\cite[(4.12)]{hhllyy18}. Therefore, using a similar argument to that used in the proof of
\cite[Theorem 7.1(ii)]{hhllyy18}, we find that
$\supp\ell^\sigma_2\subset B(x_1,16A_0^4R)$ and $\|\ell^\sigma_2\|_{L^\fz(X)}\ls\ez$. Thus,
$\|f\|_{h^{p,\fz}_\fin(X)}\ls 1$, which completes the proof of (ii) under the assumption \eqref{eq-f0}.

Now we prove \eqref{eq-f0}. Let $x\notin B(x_1,16A_0^4R)$ and $\vz\in\go{\bz,\gz}$ satisfy
$\|\vz\|_{\CG(x,r,\bz,\gz)}\le 1$ for some $r\in(0,1]$. Then $r<(4A_0)^2d(x_1,x)/3$ because $R\ge 1$.
By \cite[Corollary 4.2]{ah13}, we know that there exists a function $\xi$ such that
$\mathbf{1}_{B(x_1,(2A_0)^{-4}d(x_1,x))}\le\xi\le\mathbf{1}_{B(x_1,(2A_0)^{-3}d(x_1,x))}$
and
$$
\|\xi\|_{\dot{C}^\eta(X)}:=\sup_{y\neq z}\frac{|\xi(y)-\xi(z)|}{[d(y,z)]^\eta}
\ls [d(x_1,x)]^{-\eta}.
$$
Since $\supp f\subset B(x_1,R)$, it follows that
$f\xi=f$. Let $\wz{\vz}:=\vz\xi$. By the estimation of \cite[(7.5)]{hhllyy18}, we conclude that, for any
$y\in B(x,d(x,x_1))$,
\begin{equation*}
\lf\|\wz{\vz}\r\|_{\CG(y,r,\bz,\gz)}\ls 1.
\end{equation*}
From this, it follows that, for any $y\in B(x,d(x,x_1))$,
$$
|\langle f,\vz\rangle|=\lf|\int_X f(z)\vz(z)\,d\mu(z)\r|=\lf|\int_X f(z)\xi(z)\vz(z)\,d\mu(z)\r|
=\lf|\langle f,\wz{\vz}\rangle\r|\ls f^*_0(y),
$$
which further implies that
\begin{equation*}
|\langle f,\vz\rangle|\ls\lf\{\frac 1{\mu(B(x,d(x,x_1)))}\int_{B(x,d(x,x_1))}\lf[f^*_0(y)\r]^p
\,d\mu(y)\r\}^{1/p}\ls[\mu(B(x_1,R))]^{-1/p}.
\end{equation*}
This proves \eqref{eq-f0} and hence (i) and (ii) of Proposition \ref{prop-fin}.

The proof of Proposition (iii) is similar to that of \cite[Theorem 7.1(iii)]{hhllyy18}; we omit the details.
This finishes the proof of Proposition \ref{prop-fin}.
\end{proof}

\subsection{Dual spaces of local Hardy spaces}\label{ss-dual}

As an application of Proposition \ref{prop-fin}, in this section, we consider the dual space of the
local Hardy space $h^p(X)$ for any given $p\in(\om,1]$. Indeed, it is known that the dual space of $h^1(X)$
is the \emph{space $\bmo(X)$} (the space of all functions with local bounded mean oscillations), which has been
shown in \cite{dy12}. Thus, it suffices to consider the case $p\in(\om,1)$. Let $q\in[1,\fz]$. Denote by
$L^q_\CB(X)$ the set of all measurable functions $f$ such that $f\mathbf{1}_B\in L^q(X)$ for any ball
$B\subset X$. For any $\az\in(0,\fz)$, $q\in(1,\fz)$, ball $B:=B(x_B,r_B)$ for some $x_B\in X$ and
$r_B\in(0,\fz)$ and $f\in L^q_\CB(X)$, define
$$
\FM_{\az,q}^B(f):=\begin{cases}
\displaystyle [\mu(B)]^{-\az-1/q}\lf[\int_B |f(x)-m_B(f)|^q\,d\mu(x)\r]^{1/q}
& \textup{if }r_B\in(0,1],\\
\displaystyle [\mu(B)]^{-\az-1/q}\lf[\int_B |f(x)|^q\,d\mu(x)\r]^{1/q} & \textup{if }r_B\in(1,\fz).
\end{cases}
$$
Here and hereafter, for any ball $B$ and measurable function $f\in L^1(B)$, let
$$
m_B(f):=\frac{1}{\mu(B)}\int_B f(y)\,d\mu(y).
$$
The \emph{local Campanato space $c_{\az,q}(X)$} is defined by setting
$$
c_{\az,q}(X):=\lf\{f\in L^q_\CB(X):\ \|f\|_{c_{\az,q}(X)}:=\sup_{B\ \textup{ball}} \FM_{\az,q}^B(f)<\fz\r\},
$$
where the supremum is taken over all balls $B\subset X$. Moreover, for any $f\in L^\fz_\CB(X)$ and
$\az\in(0,1)$, define
$$
\FN_{\az}^B(f):=\begin{cases}
\displaystyle \sup_{x,y\in B}\frac{|f(x)-f(y)|}{[\mu(B)]^{\az}} & \textup{if }r_B\in(0,1],\\
\displaystyle \frac{\|f\|_{L^\fz(B)}}{[\mu(B)]^{\az}}& \textup{if }r_B\in(1,\fz).
\end{cases}
$$
The \emph{local Lipschitz space $\ell_\az(X)$} is defined by setting
$$
\ell_\az(X):=\lf\{f\in L^\fz_\CB(X):\ \|f\|_{\ell_\az(X)}:=\sup_{B\ \textup{ball}} \FN_{\az}^B(f)<\fz\r\},
$$
where the supremum is taken over all balls $B\subset X$.

\begin{remark}
Suppose $X=\rn$. Then $L^q_\CB(\rn)=L^q_\loc(\rn)$ because any bounded closed subset of $\rn$ is compact.
Moreover, for any given $p\in(n/(n+1),1)$, $\ell_{1/p-1}(\rn)=\lip_{n(1/p-1)}(\rn)$ (see Proposition
\ref{prop-l=lip} below), where, for any $\az\in(0,1)$, the \emph{space $\lip_\az(\rn)$} is defined to be the
set of all $f\in L^\fz(\rn)$ such that
$$
\|f\|_{\lip_\az(\rn)}:=\|f\|_{L^\fz(\rn)}+\sup_{x\neq y}\frac{|f(x)-f(y)|}{|x-y|^\az}<\fz,
$$
which is the dual space of $h^p(\rn)$ (see \cite[Theorem 5]{goldberg79} and \cite[(7), p.\ 5]{tri92}). By
this, we know that the definition of $\ell_\az(X)$ is reasonable.
\end{remark}

\begin{proposition}\label{prop-l=lip}
For any $p\in(n/(n+1),1)$, $\ell_{1/p-1}(\rn)=\lip_{n(1/p-1)}(\rn)$ with equivalent norms.
\end{proposition}
\begin{proof}
We first suppose $f\in\lip_{n(1/p-1)}(\rn)$ and $B:=B(x,r)$ for some $x\in\rn$ and
$r\in(0,\fz)$. If $r\in(1,\fz)$, then
$\FN^B_{1/p-1}(f)\le\|f\|_{L^\fz(X)}/|B|^{1/p-1}\ls\|f\|_{\lip_{n(1/p-1)}(\rn)}$. If $r\in(0,1]$, then, for
any $x,\ y\in B$, $|x-y|<2r$ and hence $|f(x)-f(y)|\le |x-y|^{n(1/p-1)}\ls|B|^{1/p-1}$. Thus,
$\FN^B_{1/p-1}(f)\ls\|f\|_{\lip_{n(1/p-1)}(\rn)}$, which further implies that $f\in\ell_{1/p-1}(\rn)$ and
$\|f\|_{\ell_{1/p-1}(\rn)}\ls\|f\|_{\lip_{n(1/p-1)}(\rn)}$.

Conversely, suppose $f\in\ell_{1/p-1}(\rn)$ and $x,\ y\in X$. If $|x-y|<1$, then, for any
$\ez\in (0,1-|x-y|)$, $x,\ y\in B(x,|x-y|+\ez)$ and
\begin{equation*}
|f(x)-f(y)|\le\FN_{1/p-1}^{B(x,|x-y|+\ez)}(f)|B(x,|x-y|+\ez)|^{1/p-1}\ls\|f\|_{\ell_{1/p-1}(\rn)}
(|x-y|+\ez)^{n(1/p-1)}.
\end{equation*}
Taking $\ez\to 0^+$, we obtain
\begin{equation}\label{eq-con}
|f(x)-f(y)|\ls \|f\|_{\ell_{1/p-1}(\rn)}|x-y|^{n(1/p-1)}.
\end{equation}
By \eqref{eq-con}, we find that $f$ is continuous. Moreover, for any $x\in\rn$ and $\ez\in(0,\fz)$, taking
$y_\ez\in\rn$ such that $x\in B(y_\ez,1+\ez)$, we know that
$$
|f(x)|\le\FN_{1/p-1}^{B(y_\ez,1+\ez)}(f)|B(y_\ez,1+\ez)|^{1/p-1}\ls\|f\|_{\ell_{1/p-1}(\rn)}(1+\ez)^n.
$$
Letting $\ez\to 0^+$, we then have $|f(x)|\ls\|f\|_{\ell_{1/p-1}(\rn)}$. This further implies that
\begin{equation}\label{eq-lfz}
\|f\|_{L^\fz(\rn)}\ls\|f\|_{\ell_{1/p-1}(\rn)}.
\end{equation}

Finally, suppose $x,\ y\in X$ with $x\neq y$. If $|x-y|\ge 1$, then, by \eqref{eq-lfz}, we know that
$$
|f(x)-f(y)|\le 2\|f\|_{L^\fz(\rn)}\ls\|f\|_{\ell_{1/p-1}(\rn)}|x-y|^{n(1/p-1)}.
$$
This, together with \eqref{eq-con} and \eqref{eq-lfz}, further implies that $f\in\lip_{n(1/p-1)}(\rn)$ and
$\|f\|_{\lip_{n(1/p-1)}(\rn)}\ls\|f\|_{\ell_{1/p-1}(\rn)}$. This finishes the proof of Proposition
\ref{prop-l=lip}.
\end{proof}

Now we give the dual space of $h^p(X)$.

\begin{theorem}\label{thm-dual}
Let $p\in(\om,1)$ with $\omega$ and $\eta$, respectively, as in \eqref{eq-doub} and Definition
\ref{def-iati}, and let $q\in[1,\fz]$. When $q\in(1,\fz]$, then the dual space of $h^{p,q}_\at(X)$ is
$c_{1/p-1,q'}(X)$ in the following sense:
\begin{enumerate}
\item for any $f\in c_{1/p-1,q'}(X)$, the operator $L_f$, defined by setting, for any local $(p,q)$-atom $a$,
\begin{equation}\label{7.5x}
L_f(a):=\int_X f(x)a(x)\,d\mu(x),
\end{equation}
satisfies
\begin{equation}\label{eq-fa}
\lf|\int_X f(x)a(x)\,d\mu(x)\r|\le\|f\|_{c_{1/p-1,q'}(X)},
\end{equation}
which can be extended to a unique bounded linear functional $L_f$ on $h^{p,q}_\at(X)$;
\item for any $L\in(h^{p,q}_\at(X))'$, there exists a unique function $f\in c_{1/p-1,q'}(X)$ such that
$$
\|f\|_{c_{1/p-1,q'}(X)}\le C_2\|L\|_{(h^{p,q}_\at(X))'}
$$
and $L=L_f$ in $(h^{p,q}_\at(X))'$, where $L_f$ is as in \eqref{7.5x}
and $C_2$ a positive constant, independent of $L$, and $L_f$ as in (i).
\end{enumerate}
When $q=1$, then $(h^{p,1}(X))'=\ell_{1/p-1}(X)$ in the same sense as (i) and (ii) but with
$c_{1/p,q'}(X)$ replaced by $\ell_{1/p-1}(X)$.
\end{theorem}
\begin{proof}
We only prove Theorem \ref{thm-dual} when $q=1$. The proof of the case $q\in(1,\fz]$ is similar and we omit
the details. We first prove (i). Let $f\in\ell_{1/p-1}(X)$. By the definition of $\ell_{1/p-1}(X)$, we
find that $f$ is continuous. Suppose that $a$ is a local $(p,1)$-atom supported on $B:=B(x_0,r_0)$ for some
$x_0\in X$ and $r_0\in(0,\fz)$. If $r_0\in(0,1]$, then, by the cancellation of $a$, we obtain
\begin{align*}
\lf|\int_X f(x)a(x)\,d\mu(x)\r|&\le\int_B |f(x)-f(x_0)||a(y)|\,d\mu(y)
\le\FN_{1/p-1}^B(f)[\mu(B)]^{1/p-1}[\mu(B)]^{1-1/p}\\
&\le \|f\|_{\ell_{1/p-1}(X)}.
\end{align*}
If $r_0\in(1,\fz)$, then we have
\begin{align*}
\lf|\int_X f(x)a(x)\,d\mu(x)\r|&\le\int_B |f(x)a(x)|\,d\mu(x)
\le\FN_{1/p-1}^B(f)[\mu(B)]^{1/p-1}[\mu(B)]^{1-1/p}\\
&\le \|f\|_{\ell_{1/p-1}(X)}.
\end{align*}
This proves \eqref{eq-fa} for $q=1$. Thus, define $L_f$ as a linear functional on $h^{p,1}_\fin(X)$ by
setting, for any $\vz\in h^{p,1}_\fin(X)$,
\begin{equation}\label{eq-deflf}
\langle L_f,\vz\rangle=\sum_{j=1}^N\lz_j\int_X f(x)a(x)\,d\mu(x),
\end{equation}
when $\vz$ has an atomic decomposition $\vz=\sum_{j=1}^N \lz_ja_j$ for some $N\in\nn$,
$\{\lz_j\}_{j=1}^N\subset\cc$ and local $(p,1)$-atoms $\{a_j\}_{j=1}^N$. Moreover, $L_f$ is well defined
because all the summations appearing in \eqref{eq-deflf} are finite. Then, by \eqref{eq-fa} and
$p<1$, we have
\begin{equation*}
|\langle L_f,\vz\rangle|\le\sum_{j=1}^N|\lz_j|\lf|\int_X f(x)a_j(x)\,d\mu(x)\r|
\le\|f\|_{\ell_{1/p-1}(X)}\sum_{j=1}^N|\lz_j|\le\|f\|_{\ell_{1/p-1}(X)}\lf(\sum_{j=1}^N|\lz_j|^p\r)^{1/p}.
\end{equation*}
Taking infimum over all finite atomic decompositions of $\vz$ as above and applying Proposition
\ref{prop-fin}(i), we conclude that
$$
|\langle L_f,\vz\rangle|\le\|f\|_{\ell_{1/p-1}(X)}\|\vz\|_{h^{p,1}_\fin(X)}
\ls\|f\|_{\ell_{1/p-1}(X)}\|\vz\|_{h^{p,1}_\at(X)}.
$$
Since $h^{p,1}_\fin(X)$ is dense in $h^{p,1}_\at(X)$, it then follows that
$L_f$ can be uniquely extended to a bounded linear function on $h^{p,1}_\at(X)$. This proves (i).

Now we prove (ii). Fix a ball $B:=B(x_0,r_0)$ for some $x_0\in X$ and $r_0\in(1,\fz)$. For any
$\vz\in L^1(B)$ with $\|\vz\|_{L^1(B)}>0$, $\vz[\mu(B)]^{1-1/p}/\|\vz\|_{L^1(B)}$ is a local
$(p,1)$-atom. Therefore,
$$
|\langle L,\vz\rangle|=[\mu(B)]^{1/p-1}\|\vz\|_{L^1(B)}
\lf|\lf<L,\vz[\mu(B)]^{1-1/p}/\|\vz\|_{L^1(X)}\r>\r|\le\|L\|_{(h^{p,1}_\at(X))'}[\mu(B)]^{1/p-1}
\|\vz\|_{L^1(B)}.
$$
Thus, $L$ is also a bounded linear functional on $L^1(B)$ and hence there exists a unique
$l^{(B)}\in L^\fz(B)$ such that $\|l^{(B)}\|_{L^\fz(B)}\le \|L\|_{(h^{p,1}_\at(X))'}[\mu(B)]^{1/p-1}$ and,
for any $g\in L^1(X)$ with $\supp g\subset B$,
$$
\langle L,g\rangle=\int_X l^{(B)}(x)\mathbf{1}_B(x)g(x)\,d\mu(x).
$$
Fix $x_0\in X$ and, for any $n\in\nn$, let $B_n:=B(x_0,n+1)$ and $l_n:=l^{(B_n)}$ which is defined as above.
Then, by the uniqueness of $l^{(B)}$, we find that, for any $m,\ n\in\nn$ with $m\ge n$, $l_m=l_n$ in $B_n$.
Thus, it is reasonable to define $l$ as $l=\lim_{n\to\fz}l_n$. Then $l$ has the following properties:
\begin{enumerate}
\item for any $g\in L^1(X)$ with bounded support, $\int_X l(y)g(y)\,d\mu(y)=\langle L,g\rangle$;
\item for any ball $B:=B(x,r)$ for some $x\in X$ and $r\in(1,\fz)$,
$$
\lf\|l\mathbf{1}_B\r\|_{L^\fz(B)}\le\|L\|_{(h^{p,1}_\at(X))'}[\mu(B)]^{1/p-1}.
$$
\end{enumerate}

Now we prove $\|l\|_{\ell_{1/p-1}(X)}\ls\|L\|_{(h^{p,1}_\at(X))'}$. Let $B:=B(x,r)$ be a ball for some
$x\in X$ and $r\in(0,\fz)$. If $r\in(1,\fz)$, then, by the above (ii), we have
$\FN_{1/p-1}^B(l)\le\|L\|_{(h^{p,1}_\at(X))'}$. If $r\in(0,1]$, then we claim that, for any $\vz\in L^1(X)$
supported on $B$ with $\|\vz\|_{L^1(B)}=1$, $\psi:=[\vz-m_B(\vz)]\mathbf{1}_B[\mu(B)]^{1-1/p}/2$ is a local
$(p,1)$-atom. Indeed, it is obvious that $\supp\psi\subset B$. Moreover, we have
$$
\|\psi\|_{L^1(B)}\le[\mu(B)]^{1-1/p}\int_X|\vz(y)|\,d\mu(y)+\mu(B)|m_B(\vz)|
\le 2[\mu(B)]^{1-1/p}
$$
and, by the definition of $\vz$, we have $\int_X\vz(y)\,d\mu(y)=0$. Thus, $\psi$ is a local
$(p,1)$-atom. From this and the above (ii), we deduce that
\begin{align*}
\lf|\int_X[l(y)-m_B(l)]\vz(y)\,d\mu(y)\r|&=\lf|\int_B[l(y)-m_B(l)][\vz(y)-m_B(\vz)]\,d\mu(y)\r|\\
&=\lf|\int_B l(y)[\vz(y)-m_B(\vz)]\,d\mu(y)\r|\\
&=2[\mu(B)]^{1/p-1}\lf|\int_B l(y)\psi(y)\,d\mu(y)\r|\\
&=2[\mu(B)]^{1/p-1}|\langle L,\psi\rangle|\ls [\mu(B)]^{1/p-1}\|L\|_{(h^{p,1}_\at(X))'}.
\end{align*}
Thus, $\|l-m_B(l)\|_{L^\fz(B)}\ls [\mu(B)]^{1/p-1}\|L\|_{(h^{p,1}_\at(X))'}$, which further implies that
$\FN^B_{1/p-1}(l)\ls \|L\|_{h^{p,1}_\at(X)}$. Combining this with the previous proved case when
$r\in(1,\fz)$, we obtain $l\in\ell_{1/p-1}(X)$ and $\|l\|_{\ell_{1/p-1}(X)}\ls \|L\|_{h^{p,1}_\at(X)}$. This
finishes the proof of Theorem \ref{thm-dual}.
\end{proof}

From Theorem \ref{thm-dual}, we immediately deduce the following conclusion and we omit the details.

\begin{corollary}\label{cor-lip}
Let $p\in(\om,1)$ with $\omega$ and $\eta$, respectively, as in \eqref{eq-doub} and Definition
\ref{def-iati}, and let $q\in [1,\fz)$. Then $c_{1/p-1,q}(X)=\ell_{1/p-1}(X)$ in the sense of equivalent norms.
\end{corollary}

The following observation is of independent interest.

\begin{remark}
Let $p\in(\om,1)$ with $\omega$ and $\eta$, respectively, as in \eqref{eq-doub} and Definition
\ref{def-iati}, and let $q\in [1,\fz]$. By Theorem \ref{thm-dual}, we find that, for any
$f\in c_{1/p-1,q'}(X)$ and $g\in h^{p,q}_\at(X)$ [if $q=1$, we then use $\ell_{1/p-1}(X)$ to replace
$c_{1/p-1,q'}(X)$],
\begin{equation}\label{eq-ex}
\langle L_f,g\rangle=\sum_{j=1}^\fz\lz_j\int_X f(x)a_j(x)\,d\mu(x),
\end{equation}
where $L_f$ is as in Theorem \ref{thm-dual} and $g$ has an atomic decomposition
$g=\sum_{j=1}^\fz\lz_ja_j$ with $\{\lz_j\}_{j=1}^\fz\subset\cc$ satisfying $\sum_{j=1}^\fz|\lz_j|^p<\fz$ and
local $(p,q)$-atoms $\{a_j\}_{j=1}^\fz$. The operator $L_f$ is well defined on $h^{p,q}_\at(X)$ because of
the uniqueness of the extension. Moreover, it is easy to show that
$$
|\langle L_f,g\rangle|\le\|f\|_{c_{1/p-1,q'}(X)}\|g\|_{h^{p,q}_\at(X)}.
$$
\end{remark}

\paragraph{Acknowledgments.} Dachun Yang would like to thank Professors Galia Dafni and Hong Yue for them to
provide him the reference \cite{dmy16}.

\bigskip

\noindent Ziyi He, Dachun Yang (Corresponding author) and Wen Yuan

\medskip

\noindent Laboratory of Mathematics and Complex Systems (Ministry of Education of China),
School of Mathematical Sciences, Beijing Normal University, Beijing 100875, People's Republic of China

\smallskip

\noindent{\it E-mails:} \texttt{ziyihe@mail.bnu.edu.cn} (Z. He)

\noindent\phantom{{\it E-mails:} }\texttt{dcyang@bnu.edu.cn} (D. Yang)

\noindent\phantom{{\it E-mails:} }\texttt{wenyuan@bnu.edu.cn} (W. Yuan)

\end{document}